\documentclass[11pt,a4paper]{amsart}

\usepackage{amsmath,amsfonts,amsthm,amssymb,mathrsfs,bbm,enumitem}
\usepackage{color,microtype}

\theoremstyle{plain}
\newtheorem{theorem}{Theorem}[section]
\newtheorem{lemma}[theorem]{Lemma}
\newtheorem{proposition}[theorem]{Proposition}
\newtheorem{corollary}[theorem]{Corollary}
\newtheorem*{claim}{Claim}
\theoremstyle{remark}
\newtheorem{remark}[theorem]{Remark}
\newtheorem{definition}[theorem]{Definition}
\newtheorem{problem}[theorem]{Problem}

\newcommand{\XX}{\mathcal{X}}
\newcommand{\dist}{\operatorname{dist}}

\DeclareMathOperator{\expectation}{\mathsf{E}}

\DeclareMathOperator{\var}{\mathrm{var}}
\DeclareMathOperator*{\esssup}{ess\,sup}

\allowdisplaybreaks[2]

\title[On eventually always hitting time statistics]{Dichotomy
  results for eventually always hitting time statistics and
  almost sure growth of extremes}

\author{Mark Holland}
\address{Mark Holland, University of Exeter, Department of Mathematics, North Park Road, Exeter EX4 4QF, UK}
\email{m.p.holland@exeter.ac.uk}

\author{Maxim Kirsebom}
\address{Maxim Kirsebom, University of Hamburg, Department of Mathematics, Bundesstrasse 55, 20146 Hamburg, Germany}
\email{maxim.kirsebom@uni-hamburg.de}

\author{Philipp Kunde}
\address{Philipp Kunde, University of Hamburg, Department of Mathematics, Bundesstrasse 55, 20146 Hamburg, Germany}
\email{pkunde.math@gmail.com}

\author{Tomas Persson}
\address{Tomas Persson, Centre for Mathematical Sciences, Lund University, Box 118, 221 00 Lund, Sweden}
\email{tomasp@maths.lth.se}

\thanks{This project was initiated at a ``workshop on shrinking
  targets'' at the University of Hamburg in November 2019. The
  workshop was funded by means of the University of Hamburg's
  status as a University of Excellence. We would like to thank
  the University for this financial support and
  opportunity. M. Holland acknowledges support from the EPSRC
  grant EP/P034489/1.}



\date{\today}

\begin{document}

\begin{abstract}
  Suppose $(f,\XX,\mu)$ is a measure preserving dynamical system
  and $\phi \colon \XX \to \mathbbm{R}$ a measurable function.
  Consider the maximum process $M_n:=\max\{X_1,\ldots,X_n\}$,
  where $X_i=\phi\circ f^{i-1}$ is a time series of observations
  on the system.  Suppose that $(u_n)$ is a non-decreasing
  sequence of real numbers, such that $\mu(X_1>u_n)\to 0$. For
  certain dynamical systems, we obtain a zero--one measure
  dichotomy for $\mu(M_n\leq u_n\,\textrm{i.o.})$
  depending on the sequence $u_n$. Specific examples are
  piecewise expanding interval maps including the Gau\ss{} map.
  For the broader class of non-uniformly hyperbolic dynamical
  systems, we make significant improvements on existing
  literature for characterising the sequences $u_n$.  Our
  results on the permitted sequences $u_n$ are commensurate with
  the optimal sequences (and series criteria) obtained by Klass
  (1985) for i.i.d.\ processes. Moreover, we also develop new
  series criteria on the permitted sequences in the case where
  the i.i.d.\ theory breaks down. Our analysis has strong
  connections to specific problems in eventual always hitting
  time statistics and extreme value theory.
\end{abstract}

\maketitle

\section{Introduction}\label{sec.introduction}
\subsection{General introduction and set up}
Consider a dynamical system $(\XX, \mathcal{B}, \mu, f)$, where
$(\XX, \mathcal{B}, \mu)$ is a measure space equipped with a
compatible metric which we denote by $\dist$ (that is, a metric
such that open subsets of $\XX$ are measurable), $f \colon \XX
\rightarrow \XX$ is a measurable transformation, and $\mu$ is an
$f$-invariant probability measure supported on $\XX$. Given an
observable $\phi \colon \XX \rightarrow \mathbbm{R}$, i.e.\ a
measurable function, we consider the stationary stochastic
process $X_1, X_2, \dots$ defined as
\begin{equation*} 
  X_i =\phi \circ f^{i-1}, \quad i \geq 1,  
\end{equation*}
and its associated maximum process $M_n$ defined as
\begin{equation*} 
  M_n = \max(X_1,\dots,X_n). 
\end{equation*}  
 
Extreme value theory is based on understanding the limiting
behaviour of $M_n$, either almost surely or in distribution. 
We focus on the former task of understanding almost sure growth
rates for $M_n$. This is a form of
strong law of large numbers for the maximum process $(M_n)$.  If
$\mu$ is ergodic and $\phi$ is essentially bounded then almost
surely, $M_n \to \esssup \phi$ while if $\esssup \phi = \infty$,
$M_n\to\infty$ almost surely.

A fundamental problem is to determine optimal bounding sequences
$u_n$ and $v_n$ such that almost surely there exists $N > 0$,
with $v_n\leq M_n(x)\leq u_n$, for all $n\geq N$. (Here $N$
depends on $x$).  For independent, identically distributed
(i.i.d.)\ random variables, this problem has been widely studied,
e.g.\ \cite{Barndorff, Embrechts, Galambos, Klass1, Klass2}. The
main difficultly is to find the lower bound sequence $v_n$.  The
upper bound sequence $u_n$ is generally easier to establish from
standard First and Second Borel--Cantelli Lemmas. Let us
introduce some standard notations.  For a sequence of sets
$(E_n)$, we define $(E_n \text{ i.o.})$ to be the set of points
$x\in\XX$ for which $x\in E_{n_k}$ for an infinite subsequence
$(n_k)$. Here `i.o.' means \emph{infinitely often}.  We define
$(E_n \text{ ev.})$ to be the set of points $x\in\XX$ for which
there exists $N>0$ such that $x\in E_{n}$ for all $n>N$. Here
`ev.' means \emph{eventually}.  Now, for general non-decreasing
sequences $u_n$ the events $\{M_n>u_n \text{ i.o.} \}$ and
$\{X_n>u_n \text{ i.o.}\}$ are equal (modulo a set of zero $\mu$
measure). Thus by the First Borel--Cantelli Lemma, if $\sum_{n}
\mu(X_1>u_n) < \infty$ we deduce that $\mu (M_n \leq u_n \text{
  ev.}) = 1$. Moreover if a dynamical Borel--Cantelli property
holds for $(X_n)$, with $\sum_{n} \mu(X_1>u_n)=\infty$ then
$\mu(M_n\geq u_n \text{ i.o.})=1$.

\subsection{Dichotomy results for maxima}

For i.i.d.\ processes, a relevant criterion for a sequence $(u_n)$ to
be an eventual lower bound for $M_n$ is given in particular by
\cite[Theorem~2]{Klass2}, via the \emph{Robbins--Siegmund series
  criterion}.  This can be stated as follows. Suppose that
$(\hat{X}_n)$ is an i.i.d.\ process, with probability measure $P$, and
let $u_n$ denote a non-decreasing sequence with $P(\hat{X}_1>u_n)\to
0$, $\sum_n P(\hat{X}_1>u_n) = \infty$ and
$nP(\hat{X}_1>u_n)\to\infty$. Then for the corresponding maximum
process $\hat{M}_n$ we have the dichotomy
\begin{align}
  \label{eq.rs1} \sum_{n=1}^{\infty} P (\hat{X}_1>u_n)
  e^{-nP(\hat{X}_1>u_n)} &<\infty \quad \Rightarrow \quad P (\hat{M}_n
  \geq u_n \text{ ev.})=1,\\
  \label{eq.rs2} \sum_{n=1}^{\infty} P (\hat{X}_1>u_n)
  e^{-nP(\hat{X}_1>u_n)} &=\infty \quad \Rightarrow \quad P
  (\hat{M}_n\geq u_n \text{ ev.})=0.
\end{align} 
Moreover, when $P(\hat{X}_1>u_n)\to c$, then
$P(\hat{M}_n\leq u_n \text{ i.o.})=0$, while if
\[
\liminf_{n\to\infty} nP(\hat{X}_1>u_n)<\infty,
\]
then $P(\hat{M}_n\leq u_n \text{ i.o.})=1$.

However, within a dynamical systems framework, and also for general
dependent random variables, optimal bounds on almost sure growth rates
of $M_n$ are unknown in general. Recent progress on this problem in
dynamical systems includes the works of \cite{GHPZ, GNO, HNT} where
dynamical Borel--Cantelli approaches are used to determine bounds on
$M_n$ for a wide class of dynamical systems, e.g.\ non-uniformly
expanding maps, and hyperbolic systems. More recently, this problem
has also been discussed indirectly in the analysis of eventual always
hitting time statistics \cite{GP,Kelm, OhKelmer, KeYu, KKP, Kleinbock}. 
For these latter papers, they consider a sequence of balls
$(B_n)$, and define an \emph{eventually always hitting} (EAH) event
$\mathcal{H}_{\mathrm{ea}}$ via
\begin{equation}
  \mathcal{H}_{\mathrm{ea}}
  =\bigcup_{n=1}^{\infty}\bigcap_{m=n}^{\infty}
  \bigcup_{k=0}^{m-1} f^{-k}(B_m).
\end{equation} 
Equivalently, $x\in\mathcal{H}_{\mathrm{ea}}$, if for the sequence
$\mathbf{B}=(B_n)$, there exists $m_0(x)\in\mathbbm{N}$, such that for
all $m\geq m_0(x)$ we have
\[
\{x,f(x),\ldots,f^{m-1}(x)\}\cap B_m\neq\emptyset.
\]
The term eventually always hitting was coined by Kelmer in \cite{Kelm}
where necessary and sufficient conditions for
$\mathcal{H}_{\mathrm{ea}}$ to be of full measure are established in
the context of discrete-time homogeneous flows on finite volume
hyperbolic manifolds of constant negative curvature. Shortly
afterwards Kelmer and Yu \cite{KeYu} extended the investigation to
flows on higher-rank homogeneous spaces while Kelmer and Oh considered
the case of geodesic flow on geometrically finite hyperbolic manifolds
of infinite volume \cite{OhKelmer}. Also, Kleinbock and Wadleigh
\cite{KleinWad} studied the concept in the context of higher
dimensional Diophantine approximations.

The problems addressed in \cite{KKP,Kleinbock} include conditions
placed on the sequence of measures $\mu(B_n)$ that lead to either
$\mu(\mathcal{H}_{\mathrm{ea}})=0$ or
$\mu(\mathcal{H}_{\mathrm{ea}})=1$. In fact, by ergodicity they show
that $\mu(\mathcal{H}_{\mathrm{ea}})$ can only take these zero--one
values. To link this directly to the maximum process $(M_n)$, consider
the observable $\phi(x) = \psi(\mathrm{dist}(x,\tilde{x}))$, where
$\psi \colon [0,\infty) \to \overline{\mathbbm{R}}$ is a decreasing
  continuous function with $\psi(y) \to \infty$ as $y \to 0$. (For
  example, one can take $\psi(y)=-\log(y)$.) Then the event
  $\{X_1>u_n\}$ corresponds directly to a target
  $B_n=B(\tilde{x},r_n)$ with $u_n=\psi(r_n)$, and the event
  $\{M_n\leq u_n\}$ is the event $\cap_{k=1}^{n}\{X_k\leq u_n\}$. It
  follows that
\begin{align*}
    \{M_n> u_n \text{ ev.}\} &:= \liminf_{n \to \infty}
    \Biggl( \bigcap_{k=1}^{n} \{X_k\leq u_n\} \Biggr)^\complement
    = \liminf_{n \to \infty} \bigcup_{k=1}^{n} \{X_k>
    u_n\}\\ &=\bigcup_{i=1}^{\infty}\bigcap_{n=i}^{\infty}
    \bigcup_{k=1}^{n} \{X_k> u_n \}
    =\mathcal{H}_{\mathrm{ea}}(\mathbf{B}).
\end{align*}

In this paper, we make several significant improvements on
finding almost sure bounds for the maximum process $M_n$, and
corresponding results for eventual always hitting time statistics
via zero--one laws for the measure of
$\mathcal{H}_{\mathrm{ea}}$. In particular we obtain dichotomy
results consistent with the Robbins--Siegmund criteria described
by Klass. Moreover we exhibit dynamical systems where the
Robbins--Siegmund criteria are not valid, and we propose modified
criteria beyond those stated in \eqref{eq.rs1} and
\eqref{eq.rs2}. We illustrate with a motivating example below.
The main techniques we use are based upon ideas in \emph{extreme
  value theory}, in particular on distributional convergence
results for maxima, \cite{Collet, FFT3, HN, HNT, Letal}. These
methods generally differ to those used in obtaining dynamical
Borel--Cantelli Lemmas alone.

In particular, we establish a dichotomy condition for
$\mathcal{H}_{\mathrm{ea}}$ to be of full or zero measure in
Theorem~\ref{the:lowerbound1} for a class of interval maps. This
is the first result on $\mathcal{H}_{\mathrm{ea}}$ with an exact
dichotomy that we know of (see also \cite[Question 7.1]{Kleinbock}).

\subsubsection*{Failure of Robbins--Siegmund series criterion}
We broadly ascertain that conditions \eqref{eq.rs1} and
\eqref{eq.rs2} are relevant to determine the almost sure growth
bounds for maximum processes as generated from dynamical
systems. However, we illustrate with a simple example to show
that these conditions don't always apply.  Let $(\hat{X}_n)$ be a
i.i.d.\ process with continuous probability distribution function
$F_{\hat{X}}(x)=1-1/x$, with $x\in (0,\infty)$.  For $n\geq 1$
  define a new process $(Y_n)$ by
  $Y_n=\max\{\hat{X}_{n},\hat{X}_{n+1}\}.$ The process $(Y_n)$ is
  correlated only at short time lags, and indeed $Y_n$ is
  independent of $Y_m$ when $|n-m|\geq 2$. Conditions
  \eqref{eq.rs1} and \eqref{eq.rs2} apply to the process
  $(\hat{X}_n)$. By \eqref{eq.rs1} we have for all $c<1$, 
\begin{equation*}
  \mu \Bigl(M^{\hat{X}}_n\geq \frac{cn}{\log\log
    n}\, \text{ev.} \Bigr) = 1,
		\end{equation*}
		while for any $c'>1$, \eqref{eq.rs2} implies that
\begin{equation*}		
\mu  \Bigl(M^{\hat{X}}_n\geq \frac{c'n}{\log\log n}\,\text{ev.}
  \Bigr) = 0.
\end{equation*} 
Here $M^{\hat{X}}_n=\max_{k\leq n}\hat{X}_k$. However, for the
process $(Y_n)$ we get corresponding statements for the maximum
process $M^{Y}_n=\max_{k\leq n}Y_k$ by taking instead $c<1/2$,
and $c'>1/2$. Such a result is inconsistent with conditions
\eqref{eq.rs1} and \eqref{eq.rs2} when applied to the probability
distribution for $Y_n$.  This example is discussed more formally
in Section~\ref{sec.example.dich}. To gain insight into why
conditions \eqref{eq.rs1} and \eqref{eq.rs2} fail for this
example, we appeal to \emph{extreme value theory} (EVT)
surrounding existence of distributional limit laws for maxima. We
next overview this topic.

\subsection{Background on distributional limit laws for extremes}
To obtain distributional limits in EVT, we seek sequences
$a_n,b_n \in \mathbbm{R}$ such that
\begin{equation*} 
  \mu (\{\, x\in\XX : a_n(M_n-b_n) \leq u \,\}) \to G(u),
\end{equation*}
for some non-degenerate distribution function $G(u)$, $-\infty <u
<\infty$. Several results have shown that for sufficiently
hyperbolic systems and for regular enough observables $\phi$
maximized at generic points $\tilde{x}$, the distribution limit
is the same as that which would hold if $\{X_i\}$ were
independent identically distributed (i.i.d.)\ random variables
with the same distribution function as
$\phi$~\cite{FFT1,GHN,HNT,Letal}. Particular cases include laws
towards Poisson type, described as follows. Suppose $\tau>0$, and
let $u_n(\tau)$ be a sequence such that
\begin{equation}\label{eq.un-seq}
  n\mu(X_1>u_n(\tau))\to\tau, \quad n \to \infty.
\end{equation}
Then we say that an extreme value law with extremal index
$\theta\in[0,1]$ holds for $M_n$ if
\begin{equation}\label{eq.Mn-law}
  \mu(M_n\leq u_n(\tau)) \to e^{-\theta\tau}, \qquad n \to
  \infty.
\end{equation}
If $(\hat{X}_n)$ is an i.i.d.\ process, then equation
\eqref{eq.Mn-law} holds for $\theta=1$. Thus a non-trivial extremal
index can only arise for dependent processes. Within EVT and wider
statistical theory of extremes, the index measures the degree of
\emph{clustering} for a time series of maxima, see \cite{LLR, Letal}
for details.  Various methods are available to prove the convergence
results above (in a dynamical systems context).  An important method
is a blocking algorithm approach, where in the context of general stationary stochastic
processes see \cite{Embrechts, LLR}. For dynamical systems, a blocking method approach is described
in \cite{Collet}. To determine almost sure
growth rates of maxima, we adapt the blocking method techniques that
led to the distributional convergence results given by equation
\eqref{eq.Mn-law}. As a naive approach, for a general sequence $u_n$,
equations \eqref{eq.un-seq} and \eqref{eq.Mn-law} lead us to compare
$\mu(M_n\leq u_n)$ with $e^{-n\theta\mu(X_1>u_n)}.$ In the
i.i.d.\ case, we have the exact relation:
\[\mu(M_n\leq u_n)=(1-\mu(X_1>u_n))^n.
\]
The right-hand side term is comparable to $e^{-n\mu(X_1>u_n)}$,
assuming $n\mu(X_1>u_n)^2 \to 0$. Thus, if we chose $u_n$ so that the
right-hand side is summable in $n$, then a First Borel--Cantelli Lemma
implies $\mu(M_n\geq u_n\,\text{ev.})=1$.  Thus, this
relation is not so far from the first half of the Robbins--Siegmund
criterion, namely \eqref{eq.rs1}. However, additional work is required
to get the additional multiplier $\mu(X_1>u_n)$ in \eqref{eq.rs1}. The
second half of the criterion, namely \eqref{eq.rs2} is much more
delicate to obtain, even in the i.i.d.\ case. The issue being that
$\{M_n\leq u_n\}$ is not a sequence of independent events, and hence a
Second Borel--Cantelli Lemma cannot be readily applied to conclude
whether or not $\mu(M_n\leq u_n\,\textrm{i.o.})=1$.

To obtain the relevant criteria \eqref{eq.rs1} and \eqref{eq.rs2} in
the dynamical systems context, we also require a convergence rate in
\eqref{eq.Mn-law}. This applies to the case $\theta=1$ and also for
case $\theta\neq 1$. We will treat these cases separately.  Moreover,
estimation of $\mu(M_n\leq u_n)$ is required for more general
sequences $u_n$ beyond those specified by equation \eqref{eq.un-seq}.

\begin{remark}
  Suppose $(\hat{X}_n)$ is an i.i.d.\ process with continuous
  distribution function $F_{\hat{X}}(x)=1 - 1/x$, with
  $x\in(0,\infty)$.  Then for the process
    $Y_n=\max\{\hat{X}_n,\hat{X}_{n+1}\}$ it can be shown that
    $\theta=1/2$ in \eqref{eq.Mn-law}, see
    Section~\ref{sec.example.dich}.
\end{remark}

This remark suggests that examples for which the
Robbins--Siegmund series criterion fails to apply are indeed
those processes having a non-trivial extremal index
$\theta\in(0,1)$. This discussion is made more rigorous in
Section~\ref{sec.clustering} where we develop modified versions
of \eqref{eq.rs1} and \eqref{eq.rs2} to account for processes
having a non-trivial extremal index.
 
\subsection{Organisation of the paper and overview of results}

A complete theory is yet to be developed regarding dichotomy results
on maxima and on eventually almost hitting time statistics.  We now
give an overview of the main results presented in this paper. In
Section~\ref{sec.dichotomy-pw} we present a dichotomy result for
piecewise expanding interval maps. This is Theorem~\ref{the:pw}, and
the statement is consistent with criteria \eqref{eq.rs1} and
\eqref{eq.rs2}.  As an application we consider the Gau\ss{} map, and
obtain criteria applicable to determining the growth of the
maximum for continued fraction expansion coefficients (associated to
typical real numbers $x\in[0,1]$).

In Section~\ref{sec.results}, we obtain dichotomy results for a
broader class of interval maps, such as those having exponential decay
of correlations in a suitable Banach space of functions. We show that
dichotomy results of type \eqref{eq.rs1} and \eqref{eq.rs2} are
applicable to a broad class of observable functions
$\phi(x)=\psi(\mathrm{dist}(x,\tilde{x}))$ maximised at generic points
$\tilde{x}\in \XX$. This is Theorem~\ref{the:lowerbound1}. Also within
Section~\ref{sec.results}, we consider dynamical systems having weaker
assumptions on the regularity of the invariant measure. For these
systems, we obtain conditions close to the optimal sequences governed
by \eqref{eq.rs1} and \eqref{eq.rs2}.  For example, we show that
$\mu(M_n\leq u_n\,\textrm{i.o.})=0,$ provided $u_n$ satisfies
$\mu(X_1>u_n)>c\log\log n/n$ for some $c>1$, see
Theorem~\ref{the:lowerbound3}. This gives improvements relative to the
methods derived from dynamical Borel--Cantelli Lemma analysis, such as
in \cite{HNT2, KKP}, where they require $u_n$ to satisfy conditions of
the form $\mu(X_1>u_n)>(\log n)^{\beta}/n$, for some $\beta>2$.

In Section~\ref{sec.clustering}, we obtain results that go beyond
what we expect to see for i.i.d.\ processes. For the systems we
consider we propose and apply modified criteria
relative to \eqref{eq.rs1} and \eqref{eq.rs2}. Such criteria
incorporate an extremal index $\theta$.  See
Theorem~\ref{thm.clustering1} for a precise statement. For the
dynamical systems and observables we consider, the mechanisms
leading to a non-trivial extremal index are driven by periodic
behaviour. We show that our conditions are applicable to other
dependent stochastic processes, where the extremal index is
created due to other (clustering) mechanisms. We conjecture that
our conditions are more widely applicable to other dependent
processes.

In Section~\ref{sec.alternative}, we discuss higher dimensional
dynamical systems such as those modelled by Young towers,
\cite{Young}. Again, we obtain criteria towards \eqref{eq.rs1}.
Relative to interval maps, we also need to consider regularity of the
ergodic invariant measure. In general this measure can be singular
with respect to Lebesgue measure. This creates obstacles when trying to
develop and apply a version of e.g.\ \eqref{eq.rs2}. We obtain partial
results, see Theorems~\ref{thm.hyp2} and \ref{thm.intermediate}.

Section~\ref{sec.overview-block} and onwards we devote to the
proofs. In particular for Sections~\ref{sec.overview-block} and
\ref{sec.proof.preliminary} we overview the proof strategy, including
an overview of the blocking argument, such as the one developed in
\cite{Collet}. In the later sections, such as
Section~\ref{sec.A2check} we show that the dynamical assumptions
stated in the main theorems are satisfied for a broad class of
systems.

\section{A dichotomy result for piecewise expanding maps}\label{sec.dichotomy-pw}

Our aim is to recover versions of the Robbins--Siegmund series
critera \eqref{eq.rs1} and \eqref{eq.rs2}, as applied to the
maximum process $M_n=\max_{k\leq n-1}\phi(f^{k})$, where we
consider a measure preserving system $(f,\XX,\mu)$, and $\mu$ is
an ergodic measure. In this section, we will present our results
for piecewise expanding interval maps.
\begin{theorem} \label{the:pw}
  Suppose that $f \colon \XX \to \XX$ is a piecewise expanding
  interval map with an ergodic measure $\mu$ which is absolutely
  continuous with respect to Lebesgue measure.

  Consider the observable $\phi(x) = \psi(\dist(x,\tilde{x}))$
  with $\psi(y)\to\infty$ as $y \to 0$, and a sequence $(r_n)$
  such that $n \mapsto n\mu(B(\tilde{x},r_n))$ is non-decreasing,
  $r_n = O(n^{-\sigma})$ for some $\sigma > \frac{4}{5}$, and
  such that for any $t > 0$ we have
  \begin{equation}\label{eq.mild-rn}
    \limsup_{k \to \infty} \frac{r_{k^t}}{r_{(k+1)^t}} < \infty.
  \end{equation}

  Then, for $\mu$-a.e.\ $\tilde{x}$ we have the following
  dichotomy:
  \begin{enumerate}
  \item\label{rn-ev-1}  If the sequence $(r_n)$ satisfies
    \[
    \sum_{n = 1}^\infty \mu (B(\tilde{x},r_n)) e^{- n \mu
      (B(\tilde{x},r_n))} < \infty,\;\text{with}\; \sum_{n =
      1}^\infty \mu (B(\tilde{x},r_n))=\infty,
    \]
    then
    \[
    \mu (\mathcal{H}_{\mathrm{ea}}) = \mu\bigl( \psi(r_n)\leq M_n\,
    \text{ev.}\bigr) = 1.
    \]  
  \item\label{rn-ev-0} If the sequence $(r_n)$ satisfies
    \[ \sum_{n = 1}^\infty \mu (B(\tilde{x},r_n))
    e^{- n \mu (B(\tilde{x},r_n))} = \infty,
    \]
    then
    \[
    \mu(\mathcal{H}_{\mathrm{ea}})=\mu\bigl( \psi(r_n)\leq M_n\,
    \text{ev.}\bigr) = 0.
    \]
  \end{enumerate}
\end{theorem}

\begin{remark}
  In case \eqref{rn-ev-1}, the divergence condition $\sum_{n =
    1}^\infty \mu (B(\tilde{x},r_n))=\infty$ is required, since
  the first summability constraint can be true without this
  assumption.  In case \eqref{rn-ev-0} we remark that the stated
  divergence condition implies $\sum_{n = 1}^\infty \mu
  (B(\tilde{x},r_n))=\infty$.  Hence we also have
  \[
  \mu\bigl( \psi(r_n)\leq M_n \text{ i.o.}\bigr) = \mu\bigl( \psi(r_n)\geq
  M_n \text{ i.o.}\bigr) = 1,
  \]
  provided that $\sum_{n = 1}^\infty \mu (B(\tilde{x},r_n)) e^{-
    n \mu (B(\tilde{x},r_n))} =\infty$.
\end{remark}

\begin{remark}
  As stated, the set of points $\tilde{x}\in\XX$ for which (1) and (2)
  hold has full $\mu$-measure. In general, it is not straightforward
  to know whether a particular $\tilde{x}$ is in this set. As we
  discuss in Section~\ref{sec.clustering} periodic points are not in
  this full measure set. For such periodic points alternative
  formulations of (1) and (2) are required.
\end{remark}

Theorem~\ref{the:pw} is a consequence of Theorem~\ref{the:lowerbound1}
in the next section, combined with
Proposition~\ref{prop:A2forpwexp}. We remark that the condition
\eqref{eq.mild-rn} is quite a mild condition on the sequence $r_n$. If
$r_n$ is regularly varying, (in the sense of \cite{BGT}), then it will
satisfy \eqref{eq.mild-rn}. Certain sequences with fast decay (such as
exponential) violate \eqref{eq.mild-rn}, but this becomes a moot issue
since we assume $\sum_n\mu (B(\tilde{x},r_n))=\infty$. Therefore $r_n$
cannot decay too quickly, unless the measure density is quite
degenerate at $\tilde{x}$.  For i.i.d.\ processes, and depending on
the criteria being used, mild regularity constraints are also
discussed (and imposed) in \cite{Barndorff,Klass1,Klass2}. The latter
reference gives the most freedom on the allowed sequences, as we have
already summarised in Section~\ref{sec.introduction}.

In the next section we present several results of this type which hold
under various more or less abstract assumptions on the systems. For
piecewise expanding systems, Proposition~\ref{prop:A2forpwexp} tells
us that these assumptions are satisfied for a.e.~ $\tilde{x}$. The
restriction $\sigma>4/5$ is a consequence of the methods of proof. We
have not tried to optimise this range on $\sigma$.
For completeness, upper bounds on the growth of maxima can be deduced from the following
theorem.
\begin{theorem} \label{the:upperbound}
  Suppose that $f\colon \XX \to \XX$ is a dynamical system with
  an ergodic probability measure $\mu$.
  If $(r_n)$ is a sequence such that
  \[
  \sum_{n = 1}^\infty \mu (B(\tilde{x},r_n)) < \infty,
  \]
  then
  \[
  \mu\bigl( \psi(r_n) \geq M_n \text{ ev.}\bigr) = 1.
  \]
\end{theorem}

\begin{proof}
  Since $\mu (B(\tilde{x},r_n))$ is summable, we get by the First
  Borel--Cantelli Lemma that almost surely, the event $\{ X_n
  \geq \psi (r_n) \}$ happens only finitely many times. It
  follows that almost surely, $M_n \leq \psi(r_n)$ holds for all
  large enough $n$.
\end{proof}

\subsubsection*{The Gau\ss{} Map and growth of continued
  fractions.}

We end this section with an important application of
Theorem~\ref{the:pw}, namely to the Gau\ss{} map. This allows us
to obtain a dichotomy result on the almost sure growth rates of
continued fraction expansion coefficients. Recall that for a
number $x\in[0,1]$, its continued fraction is given by
$x=[a_{0},a_1,a_2,\ldots]$, where
\[
a_k(x) = \biggl\lfloor \frac{1}{G^{k}(x)} \biggr\rfloor, \qquad
\text{and} \qquad G(x) = \frac{1}{x}\mod 1.
\]
The map $G \colon [0,1]\to[0,1]$ (with $G(0)=0$) is the
\emph{Gau\ss{} map}. The map is piecewise expanding, full branch,
with countable Markov partition. The map admits an ergodic
measure $\mu$ with invariant density $\rho(x)=(\log
2)^{-1}(1+x)^{-1}.$ In Philipp \cite[Theorem~1]{Philipp} it is
shown that for $\mu$-almost all $x$
\begin{equation}\label{eq.cf}
  \liminf_{n\to\infty} n^{-1}L_n(x)\log\log n=\frac{1}{\log 2},
\end{equation}
where $L_n(x)=\max_{i\leq n}a_i(x)$. In Corollary~\ref{cor.gauss}
below we obtain a result commensurate with that of Philipp's
dichotomy result \cite[Theorem~2]{Philipp}, which allows us to
obtain higher order terms in the convergence rate to the
limit. Philipp's dichotomy result naturally builds upon the
earlier works of Barndorff-Neilson~\cite{Barndorff} as applied in
the i.i.d.\ case.  In the recent work of \cite{Kleinbock} they
also obtain estimates which lead to the result of \eqref{eq.cf},
but they don't obtain a sharp dichotomy criterion.

\begin{corollary}\label{cor.gauss}
  Suppose that $G\colon [0,1] \to [0,1]$ is the Gau\ss{} map, and
  $\mu$ the ergodic absolutely continuous invariant
  measure. Consider the observable $\phi(x) =
  \psi(\dist(x,\tilde{x}))$ with $\psi(y)\to\infty$ as $y \to 0$,
  and a sequence $(r_n)$ such that $n \mapsto
  n\mu(B(\tilde{x},r_n))$ is non-decreasing, $r_n =
  O(n^{-\sigma})$ for some $\sigma > \frac{4}{5}$, and satisfying
  \eqref{eq.mild-rn}.  Then for $\mu$-a.e.\ $\tilde{x}$ cases
  \eqref{rn-ev-1} and \eqref{rn-ev-0} of Theorem~\ref{the:pw}
  apply.
\end{corollary}

To relate this corollary to continued fractions, we take the
observable $\psi(x)=\lfloor 1/x\rfloor$ with $\tilde{x}=0$ so that
$a_k(x)=\psi(G^k(x))$.  One then attempts to apply
Theorem~\ref{the:pw}, but in order to do so we need to know that we
can use the theorem for $\tilde{x} = 0$. Instead we use
Theorem~\ref{the:lowerbound1}, and we need only to check that
condition (A2) (see Section~\ref{sec.results}) is satisfied for
$\tilde{x} = 0$. To do so is standard, and is left out. See also
\cite{GKR}. The results of Philipp are recovered from
Corollary~\ref{cor.gauss} by first noting that
\[
\mu ([0,r_n))=\frac{1}{\log
  2}\log(1+r_n)=\frac{r_n}{\log 2}+O(r^{2}_n).
\]
Then consider each case (1) and (2) in the corollary using
$r_n=c\log\log n/n$ for the (one-sided) ball $[0,r_n]$, and taking in
turn $c>\log 2$, followed by $c<\log 2$.  When $c=\log 2$, further
(additive) error term refinements can be obtained.

\section{Towards a dichotomy result for the almost sure growth of
  maxima}\label{sec.results}

Our aim is to recover versions of the Robbins--Siegmund series
critera \eqref{eq.rs1} and \eqref{eq.rs2}, as applied to the maximum
process $M_n=\max_{k\leq n-1}\phi(f^{k})$, where we consider a
measure preserving system $(f,\XX,\mu)$, and $\mu$ is an ergodic
measure. The systems we consider include those that can be
modelled by a Young tower \cite{Young}, but in the statement of
our results we just require control on the rate of decay of
correlations.  We make these statements precise as follows.

\begin{definition}\label{def.dc}
  We say that $(f,\XX,\mu)$ has decay of correlations in (Banach
  spaces) $\mathcal{B}_1$ versus $\mathcal{B}_2$, with rate
  function $\Theta(j)\to 0$ if for all
  $\varphi_1\in\mathcal{B}_1$ and $\varphi_2\in\mathcal{B}_2$ we
  have
  \[
  \mathcal{C}_j(\varphi_1,\varphi_2,\mu):= \biggl|\int \varphi_1
  \cdot \varphi_2\circ f^j \, \mathrm{d} \mu -\int \varphi_1 \,
  \mathrm{d} \mu \int \varphi_2 \, \mathrm{d} \mu \biggr| \leq
  \Theta(j) \|\varphi_1\|_{\mathcal{B}_1}
  \|\varphi_2\|_{\mathcal{B}_2},
  \]
  where $\|\cdot\|_{\mathcal{B}_i}$ denote the corresponding
  norms on the Banach spaces.
\end{definition}

In particular, we consider the $L^1$ and $\mathrm{BV}$ norms of
functions $\varphi \colon \XX \subset \mathbbm{R} \to
\mathbbm{R}$, defined by
\begin{align*}
  \|\varphi \|_1 & = \int |\varphi| \, \mathrm{d}\mu,
  \\ \|\varphi\|_{\mathrm{BV}} & = \var(\varphi) +
  \sup(|\varphi|),
\end{align*}
where $\var(\varphi)$ denotes the total variation of
$\varphi$. Functions $\varphi \colon \XX \subset \mathbbm{R} \to
\mathbbm{R}$ with $\|\varphi\|_{\mathrm{BV}}<\infty$ are called
functions of bounded variation.

The first main assumption is the following.
\begin{enumerate}
\item[(A1)] \textbf{Exponential decay of correlations} with respect to
  notations of Definition~\ref{def.dc}. We assume that
  $(f,\XX,\mu)$ has exponential decay of correlations in Banach spaces
  $\mathcal{B}_1=\mathrm{BV}$ versus $\mathcal{B}_2=L^{\infty}$.
\end{enumerate}
Next, we define for a sequence $r_n\to 0$ and integer
$p\in\mathbb{N}$ the following quantity
\begin{equation}
  \Xi_{p,n}\equiv\Xi_{p,n}(r_n):=\sum_{j=1}^{p}\mu\bigl(
  B(\tilde{x},r_n)\cap f^{-j} B(\tilde{x},r_n)\bigr).
\end{equation}

Our second important assumption is the following.
\begin{enumerate}
\item[(A2)] \textbf{Short Return Times estimate.} Let
  $s,\gamma\in(0,1)$, and suppose $\gamma+s<1$. Furthermore, suppose
  that $B(\tilde{x},r_n)$ is a sequence of balls with $\sum_n \mu
  (B(\tilde{x},r_n))=\infty$ and $\mu
  (B(\tilde{x},r_n))=O(n^{-\sigma})$ for some $\sigma\in(0,1)$
  satisfying
	\begin{equation}\label{eq.short-constants}
	\sigma> \max \bigl\{\frac{1}{2}(1+\gamma+s),1-\gamma,\frac{1}{3}(2+s),1-\frac{s}{2}\bigr\}.
	\end{equation}
	Then along the sequence $p_n=n^s$, we have
  \begin{equation}\label{eq.short1}
    \Xi_{p_n,n}(r_n)=O\bigl(\frac{1}{n^{1+\gamma}}\bigr).
  \end{equation}
\end{enumerate}

Condition (A1) is known to hold for a wide class of dynamical systems,
such as uniformly expanding maps \cite{LasotaYorke, Liveranietal}, the
Gau\ss{} map \cite{Philipp, Rychlik}, and also certain non-uniformly
expanding quadratic maps \cite{Young_quadratic}. In other
applications, modified versions of (A1) include taking $\mathcal{B}_1$
as the space of H\"older continuous functions (maintaining
$\mathcal{B}_2 = L^{\infty}$). We will do this on a case-by-case
basis.
	
Condition (A2) gives a restriction on the recurrence properties of
$\tilde{x}\in\XX$, and for a broad class of systems, this condition
can be proved to hold for $\mu$-a.e.\ $\tilde{x} \in \XX$ along the
lines of \cite{Collet, HNT, HRS}. For readers familiar with extreme
value theory, equation \eqref{eq.short1} is similar to the $D'(u_n)$
condition considered in \cite{Embrechts, LLR, Letal}, where $u_n$ plays
the role of $\psi(r_n)$. Some of the restrictions on the constants
within equation \eqref{eq.short-constants} arise through requiring
self-consistency of equation \eqref{eq.short1}. Indeed, by exponential
decay of correlations (A1) we have $\mu( B(\tilde{x},r_n)\cap f^{-k}
B(\tilde{x},r_n))\approx \mu(B(\tilde{x},r_n))^2$ for $k\gg\log n$.
Hence, the bound in equation \eqref{eq.short1} forces
$2\sigma-s>1+\gamma$. In Section~\ref{sec.proof.preliminary} we
discuss the role for the other lower bounds on $\sigma$ within
\eqref{eq.short-constants}. Thus, condition (A2) mainly applies to
control the measure $\mu( B(\tilde{x},r_n)\cap f^{-k}
B(\tilde{x},r_n))$ for $k=O(\log n)$, i.e.\ for short return times.
In turn this condition is a restriction on the recurrence statistics
of $\tilde{x}$ and nearby points. In Section~\ref{sec.A2check} we
provide general techniques to verify (A2). For a wide class of
dynamical systems we show that (A2) holds for $\mu$-a.e.\ $x\in
\XX$. The techniques we discuss build upon and complement arguments
used in Collet~\cite{Collet}. Examples include piecewise expanding
maps with absolutely continuous invariant measures or more general
Gibbs measures, and quadratic maps with Benedicks--Carleson
parameters. For the latter, (A1) is also satisfied by Young
\cite{Young_quadratic}.

As stated, Condition (A2) is mainly applicable for systems satisfying
(A1). For systems with polynomial decay of correlations, see
\cite{HNT,HRS} for conditions similar to (A2). In this article, we
consider mainly systems with exponential decay of correlation.  The
exception is the Manneville--Pomeau map which we treat in
Section~\ref{sec.mp-map}. We remark further that the exceptional set
of points $\tilde{x}$ for which (A2) fails includes periodic
points. We discuss this further in Section~\ref{sec.clustering}.

To state the next theorem we consider an interval map $f$ and we
recall that a measure $\nu$ is called conformal for the
non-negative function $g \colon \XX \to \mathbbm{R}^+$ if for
every measurable set $E\in \mathcal{B}$, on which $f$ acts as a
measurable isomorphism, we have
\[
\nu(f(E)) = \int_E g \, \mathrm{d}\nu.
\]
One often refers to $\log g$ as a potential. Furthermore, the
transfer operator $\mathscr{L} \colon L^1(\XX,\nu) \to
L^1(\XX,\nu)$ (sometimes called Ruelle operator or
Perron--Frobenius operator) is defined by
\[
\mathscr{L} \psi (x) = \sum_{f(y) = x} g (y) \psi(y).
\]
We often restrict $\mathscr{L}$ to the functions of bounded
variation.

\begin{theorem} \label{the:lowerbound1}
  Suppose that $f\colon \XX \to \XX$ is an interval map with an
  ergodic probability measure $\mu$. Assume that condition (A1)
  holds. Furthermore, assume that $\mu$ is a Gibbs measure which has a
  density $h$ with respect to a conformal measure, that $h$ is an
  eigenfunction of a transfer operator with a spectral gap when acting
  on function of bounded variation, and that $h$ is the unique (up to
  scalar multiples) eigenfunction of maximal modulus of the
  eigenvalue.  Consider the observable $\phi(x) =
  \psi(\dist(x,\tilde{x}))$ with $\psi(y) \to \infty$ as $y \to 0$,
  and suppose that condition (A2) holds for the sequence of balls
  $\{B(\tilde{x},r_n)\}$, $r_n\to 0$ centered at $\tilde{x}$. Moreover,
	suppose the sequence $(r_n)$ is such that $n\mapsto
  n\mu(B(\tilde{x},r_n))$ is non-decreasing. Then we
  have the following dichotomy:
\begin{enumerate}
\item  If the sequence $(r_n)$ satisfies
  \[
   \sum_{n = 1}^\infty \mu (B(\tilde{x},r_n))
  e^{- n \mu (B(\tilde{x},r_n))} < \infty,
  \]
  then
  \[
  \mu (\mathcal{H}_{\mathrm{ea}}) = \mu\bigl( \psi(r_n)\leq M_n\,
  \text{ev.}\bigr) = 1.
  \]  
\item If the sequence $(r_n)$ satisfies
  \[ \sum_{n = 1}^\infty \mu (B(\tilde{x},r_n))
  e^{- n \mu (B(\tilde{x},r_n))} = \infty,
  \]
  then
  \[
  \mu(\mathcal{H}_{\mathrm{ea}})=\mu\bigl( \psi(r_n)\leq M_n\,
  \text{ev.}\bigr) = 0.
  \]
\end{enumerate}
\end{theorem}

\begin{remark}
  Liverani, Saussol and Vaienti \cite{Liveranietal} studied a general
  class of piecewise expanding interval maps and a Gibbs measure $\mu$
  with respect to a potential. They proved that under some mild
  regularity conditions and a ``covering'' condition, that the
  transfer operator related to the potential has a spectral gap and a
  unique (up to scaling) eigenfunction associated to the eigenvalue of
  maximal modulus. A special case is that $\mu$ is a measure which is
  absolutely continuous with respect to Lebesgue measure. We discuss
  various dynamical system case studies in Section~\ref{sec.A2check},
  including piecewise differentiable maps satisfying (A1). For these
  systems we show also that condition (A2) holds for
  a.e.\ $\tilde{x}$.
\end{remark}

There are systems which do not satisfy the
assumptions of Theorem~\ref{the:lowerbound1}, but for which we are
able to prove a similar result.  To state this result, we introduce
the following complexity growth condition.

\begin{enumerate}
\item[(A3)] There exists $K_f \in \mathbbm{N}$, such that for all
  $r \geq 0$ and all $x$ the set $f^{-1}(B(x,r))$ has at most
  $K_f$ connected components.
\end{enumerate}

Condition (A3) is satisfied by maps with finitely many monotone
branches (such as unimodal maps). For these systems the density of
$\mu$ is not necessarily bounded and hence not in $\mathrm{BV}$ as required by
Theorem~\ref{the:lowerbound1}. We have the following result.

\begin{theorem} \label{the:lowerbound3}
  Suppose that $f\colon \XX \to \XX$ is an interval map with
  ergodic probability measure $\mu$, and assume that condition
  (A1) holds.

  Consider the observable $\phi(x) = \psi(\dist(x,\tilde{x}))$
  with $\psi(y) \to \infty$ as $y \to 0$, and suppose that
  condition (A2) holds for sequences of balls centered at
  $\tilde{x}$. Moreover, suppose the sequence $(r_n)$ is such that $n\mapsto
  n\mu(B(\tilde{x},r_n))$ is non-decreasing. We have the following cases.
\begin{enumerate}
\item
  Suppose the sequence $(r_n)$ satisfies
  \[
  \qquad \sum_{n = 1}^\infty \mu (B(\tilde{x},r_n))
  e^{- n \mu (B(\tilde{x},r_n))} < \infty.
  \]
  Then for any $a > 1$, we have
  \[
  \mu\bigl( M_n \geq \psi(r_{[\frac{n}{a}]}) \text{ ev.}\bigr) =
  1.
  \]  
  
  In particular, if $\mu \{\phi \geq u_n\} \geq c \frac{\log \log
    n}{n}$ for some constant $c > 1$, then
  \[
  \mu (\mathcal{H}_{\mathrm{ea}}) = \mu\bigl( u_n \leq M_n \text{
    ev.} \bigr) = 1.
  \]
	\item
	 Suppose condition (A3) holds, and the sequence $(r_n)$ satisfies
  \[
  \sum_{n = 1}^\infty \mu (B(\tilde{x},r_n))
  e^{- n \gamma \mu (B(\tilde{x},r_n))} = \infty,
  \]
  for some $\gamma>1$. Then we have
  \[
  \mu (\mathcal{H}_{\mathrm{ea}}) = \mu\bigl(\psi(r_n) \leq M_n
  \text{ ev.}\bigr) = 0.
  \]
  
  In particular, if $\mu (\phi \geq u_n) \leq c \frac{\log \log
  	n}{n}$ for some constant $c < 1$, then
  \[
  \mu (\mathcal{H}_{\mathrm{ea}}) = \mu\bigl( u_n \leq M_n \text{
  	ev.} \bigr) = 0.
  \]
	\end{enumerate}
	\end{theorem}

\begin{remark}
  An example fitting this theorem is a quadratic map with
  Benedicks--Carleson parameter. For such systems it can be
  proved along the lines of the estimates by Collet \cite{Collet}
  that (A2) holds for almost all $\tilde{x}$. See
  Section~\ref{sec.A2check} for precise statements.
\end{remark}

\begin{remark}
  The proof of Theorem~\ref{the:lowerbound3} uses a
  Cauchy-condensation method. The method turns out to be quite
  versatile in obtaining the lower bound sequence $u_n=\psi(r_n)$
  for $M_n$, but is less applicable for establishing dichotomy
  results, i.e.\ to understand when $\nu(M_n\leq u_n \text{ i.o.})=1.$ 
Thus to prove case (2) we instead follow a method similar to that used
for case (2) of Theorem~\ref{the:lowerbound1}.		
	The contrasting bounds obtained from
  these two theorems are very fine. Indeed, from
  Theorem~\ref{the:lowerbound1}, a lower bound sequence $u_n$
  satisfies
  \[
  \mu(X_1>u_n)\geq\frac{\log\log n}{n}+\frac{c\log\log\log n}{n},
  \qquad \text{for some } c > 2.
  \]
  This sequence is a narrow improvement on the range of lower bound
  sequences implied by Theorem~\ref{the:lowerbound3}. Similarly for
  the sequences that determine when $\mu
  (\mathcal{H}_{\mathrm{ea}})=0$.
\end{remark}

It is worth to compare these bounds with Corollaries~1.3 and 1.4 in
\cite{Kleinbock}. Imposing a long-term independence property on
the shrinking target system they obtain tight conditions on the
shrinking rate of the targets so that $\mathcal{H}_{\mathrm{ea}}$
has zero or full measure. In particular, their
assumptions are satisfied for specific choices of targets in
product systems and Bernoulli shifts. In the case of product
systems, \cite[Corollary~1.3]{Kleinbock} yields that the
shrinking rate $\mu(B_n)\geq \frac{c \log\log n}{n}$ for some
$c>1$ implies $\mu(\mathcal{H}_{\mathrm{ea}}(\mathbf{B}))=1$,
while $\mu(B_n)\leq \frac{\log\log n}{n}$ for all but finitely
many $n$ implies $\mu(\mathcal{H}_{\mathrm{ea}}(\mathbf{B}))=0$.

In Section~\ref{sec.clustering}, we discuss results and examples
in the case where the short return time condition (A2) fails. For
these examples the maximum process has a non-trivial extremal
index $\theta\in(0,1)$. We also show that the Robbins--Siegmund
series criteria fails, and propose more general criteria on what
the bounding sequences $u_n$ and $v_n$ should satisfy. 

\section{Almost sure growth of maxima for processes having an
  extremal index}\label{sec.clustering}

In this section we consider again measure preserving dynamical
system $(f,\mathcal{X},\mu)$, and the maximum process
$M_n=\max_{k\leq n-1}\phi\circ f^{k}$. However, we consider the
situation where equations \eqref{eq.un-seq} and \eqref{eq.Mn-law}
apply for a non-trivial extremal index $\theta\in(0,1)$. We show
that the Robbins--Siegmund series criteria as stated in
\eqref{eq.rs1} and \eqref{eq.rs2} are no longer valid for
producing the (almost sure) bounding sequences for the process
$M_n$.  We obtain modified series criteria based on inclusion of
the parameter $\theta$. Towards the end of this section, we
propose a general question on the validity of such series
criteria for bounding $M_n$ in the case of general dependent
processes, i.e.\ beyond dynamical systems.

More explicitly, we consider situations where the short return
time condition (A2) fails. For dynamical systems, this can happen
in the case for a sequence of shrinking targets limiting onto a
periodic point. To state our main results, we shall focus on this
case. Indeed, for observables of the form
$\phi(x)=\psi(\mathrm{dist}(x,\tilde{x}))$ a non-trivial extremal
index tends to only arise in these cases, especially for the
dynamical systems we consider.  However, there are many other
mechanisms that can give rise to a non-trivial $\theta$, for an
overview see \cite{Letal}.
 
For an observable maximized at a periodic point
$\tilde{x}\in\mathcal{X}$, assumption (A2) can be shown to fail
as follows.  Suppose $f^{p}(\tilde{x})=\tilde{x}$, for some
$p\geq 1$. Then
\[
\mu(B(\tilde{x},r) \cap f^{-p}(B(\tilde{x},r))) > C_p \mu
(B(\tilde{x},r)),
\]
where $C_p>0$ depends on the derivative of $f^p$ at $\tilde{x}$
and the measure $\mu$. (The constant is non-zero, if
$(f^p)'(\tilde{x})<\infty$ and $\mu$ is equivalent to Lebesgue
measure, at least locally at $\tilde{x}$).  From the view of
extreme value theory, the maximum process has a distribution
governed by a non-trivial extremal index, as described by
equations \eqref{eq.un-seq} and \eqref{eq.Mn-law}.

The blocking arguments that we use to prove
Theorems~\ref{the:lowerbound1}--\ref{the:lowerbound3} must be
adapted to account for the failure of condition (A2). The
relevant modifications are discussed in (for example)
\cite{FFT3}, and we review the relevant constructions. We
introduce the events
\begin{align*} 
  U(u) &= \{X_0>u\},\\ A^{(q)}(u) &= U(u) \cap \bigcap_{k=1}^{q}
  f^{-k}(U(u)^\complement). 
\end{align*}
In the case $u:=u_n$ we write $U_n=U(u_n)$, and
$A^{(q)}_n=A^{(q)}(u_n).$ Define a sequence $\theta_n$ via
\begin{equation*} 
  \theta_n=\frac{\mu(A^{(q)}_n)}{\mu(U_n)}.
\end{equation*}
In the setting of a dynamical system with a $q$-periodic point
$\tilde{x}$ and an observable $\phi(x) =
\psi(\mathrm{dist}(x,\tilde{x}))$, the event $A^{(q)}(u_n)$ gives
the points in $U_n= B(\tilde{x},r_n)$, where $u_n = \psi(r_n)$,
that do not return to $B(\tilde{x},r_n)$ within $q$
iterates. Accordingly, $\theta_n$ is the proportion of points in
$B(\tilde{x},r_n)$ that do not return within $q$ iterates.

When $\theta=\lim_{n \to \infty}\theta_n$ exists, then this
constant $\theta\in(0,1]$ takes the role as the extremal index.
  We make the following convergence assumption.
\begin{enumerate}
\item[(A4)] There exists $\hat\sigma>0$ such that
  \begin{equation*} \label{eq:ei-rate}
    |\theta_n-\theta|=O(n^{-\hat\sigma}),
  \end{equation*}
  and $\theta\neq 0$.
\end{enumerate}
Assumption~(A4) is an assumption on the local property of the
dynamical system at $\tilde{x}$.  Verification depends on
assumptions of the regularity of the invariant density and
derivative of $f$ at $\tilde{x}$.  It is easy to see that
Assumption~(A4) is valid when $\tilde{x}$ is a hyperbolic
repelling periodic point for a piecewise (linear) expanding map,
and $\mu$ is Lebesgue measure. In these cases,
for the limit we get
\begin{equation}\label{eq.ei-formula}
\theta = 1 - \frac{1}{|(f^q)'(\tilde{x})|}.
\end{equation}
Indeed, if $f(x)$ is the doubling map
$f(x)=2x\mod 1$ on $[0,1]$, then we get the exact formula
\[
\theta_n=\theta=1-\frac{1}{2^q},
\]
for all $n$ sufficiently large. The result of equation
\eqref{eq.ei-formula} also holds in greater generality, see
\cite{FFT2}. We remark that $\theta\neq 0$ is assumed in (A4).
In Section~\ref{sec.theta-zero} we consider an example where
$\theta=0$.

For a broad class of non-uniformly expanding dynamical systems,
see \cite{FFT2} for examples where $\theta$ is computed together
with establishing the corresponding limit law for the maxima.  We
have the following result.  

\begin{theorem} \label{thm.clustering1}
  Suppose that $f\colon \XX \to \XX$ is an interval map with an
  ergodic probability measure $\mu$. Assume that condition (A1) holds.
  Furthermore, assume that $\mu$ is a Gibbs measure which has a
  density $h$ with respect to a conformal measure, that $h$ is an
  eigenfunction of a transfer operator with a spectral gap when acting
  on function of bounded variation, and that $h$ is the unique (up to
  scalar multiples) eigenfunction of maximal modulus of the
  eigenvalue.

  Consider the observable $\phi(x) = \psi(\dist(x,\tilde{x}))$ with
  $\psi(y) \to \infty$ as $y \to 0$, and $\tilde{x}$ is a hyperbolic
  periodic point of period $q$, and
  $|(f^q)'(\tilde{x})|<\infty$. Consider a sequence of balls of radii
  $r_n \to 0$, each centered at $\tilde{x}$.  Suppose that (A4)
  holds, and that the sequence $(r_n)$ is such that $n\mapsto
  n\mu(B(\tilde{x},r_n))$ is non-decreasing. Then we have the following dichotomy:
  \begin{enumerate}
  \item  If the sequence $(r_n)$ satisfies
    \[
    \sum_{n = 1}^\infty \mu (B(\tilde{x},r_n)) e^{- n\theta \mu
      (B(\tilde{x},r_n))} < \infty,
    \]
    then
    \[
    \mu (\mathcal{H}_{\mathrm{ea}}) = \mu\bigl( \psi(r_n)\leq M_n
    \text{ ev.}\bigr) = 1.
    \]  
  \item If the sequence $(r_n)$ satisfies	
    \[\qquad \sum_{n = 1}^\infty \mu (B(\tilde{x},r_n))
    e^{- n \theta \mu (B(\tilde{x},r_n))} = \infty,
    \]
    then
    \[
    \mu(\mathcal{H}_{\mathrm{ea}})=\mu\bigl( \psi(r_n)\leq M_n
    \text{ ev.}\bigr) = 0.
    \]
  \end{enumerate}
  In both cases $\theta$ denotes the corresponding extremal
  index.
\end{theorem}

We prove Theorem~\ref{thm.clustering1} in
Section~\ref{sec.pf.thm.clustering1}. See also \cite[Section~4]{FFT2}
for similar discussions in the distributional convergence cases. The
following result concerns the eventual lower bound for the maximum
process, and can be contrasted to Theorem~\ref{the:lowerbound3}.
\begin{theorem}\label{thm.hyp1}
Suppose that $f\colon \XX \to \XX$ is an interval map with ergodic
probability measure $\mu$, and assume that condition (A1)
holds. Consider the observable
$\phi=\psi(\mathrm{dist}(x,\tilde{x}))$, where $\tilde{x}$ is a
hyperbolic periodic point of period $q$ with
$|(f^q)'(\tilde{x})|<\infty$, and suppose (A4) holds at $\tilde{x}$.
Moreover, suppose the sequence $(r_n)$ is such that $n\mapsto n\mu(B(\tilde{x},r_n))$ is
non-decreasing. We have the following cases.
\begin{enumerate}
\item
  Suppose the sequence $(r_n)$ satisfies
  \[
  \qquad \sum_{n = 1}^\infty \mu (B(\tilde{x},r_n))
  e^{- n\theta \mu (B(\tilde{x},r_n))} < \infty.
  \]
  Then for any $a > 1$, we have
  \[
  \mu\bigl( M_n \geq \psi(r_{[\frac{n}{a}]}) \text{ ev.}\bigr)
  = 1.
  \]  
  
  In particular, if $\mu (\phi \geq u_n) \geq c\theta^{-1} \frac{\log \log
    n}{n}$ for some constant $c > 1$, then
  \[
  \mu (\mathcal{H}_{\mathrm{ea}}) = \mu \bigl( u_n \leq M_n \text{
    ev.} \bigr) = 1.
  \]
\item
  Suppose condition (A3) holds, and the sequence $(r_n)$ satisfies
  \[
  \sum_{n = 1}^\infty \mu (B(\tilde{x},r_n))
  e^{- n\theta \gamma \mu (B(\tilde{x},r_n))} = \infty,
  \]
  for some $\gamma>1$. Then we have
  \[
  \mu (\mathcal{H}_{\mathrm{ea}}) = \mu\bigl(\psi(r_n) \leq M_n
  \text{ ev.}\bigr) = 0.
  \]
  
  In particular, if $\mu (\phi \geq u_n) \leq c\theta^{-1} \frac{\log \log
    n}{n}$ for some constant $c < 1$, then
  \[
  \mu (\mathcal{H}_{\mathrm{ea}}) = \mu \bigl( u_n \leq M_n \text{
  	ev.} \bigr) = 0.
  \]
\end{enumerate}
\end{theorem}

The proof of Theorem~\ref{thm.hyp1} is also given in
Section~\ref{sec.pf.thm.clustering1}. The method of proof uses
Cauchy-condensation techniques.

\subsection{Examples and general discussion of dichotomy criteria.}\label{sec.applications}

In this section we consider further examples, including a
dichotomy criterion applied to a non-dynamical example.  We also
consider an example where the extremal index is zero.
\subsubsection{An example of a stochastic process satisfying the dichotomy
  for a non-trivial $\theta \in (0,1)$.}\label{sec.example.dich}

We consider a stationary stochastic process which has a non-trivial
extremal index.  The mechanism giving rise to the extremal index
here is different to the periodic phenomena arising in the
dynamical systems.

Suppose $(X_n)$ is an i.i.d.\ process with distribution function
$F_X (x) = P (X<x) = 1-x^{-1}$. We let $\overline{F}_X(x) = 1 -
F_X (x)$ (the tail distribution). Consider the process $Y_n =
\max (X_{n},aX_{n+1})$, where $a \geq 1$ is fixed.  We claim the
following:
\begin{enumerate}
\item The extremal index for the process $M^{Y}_n=\max_{k\leq
  n}Y_k$ is given by $\theta=a/(a+1)$.
\item Dichotomy criteria, as consistent with the conclusion of
  Theorem~\ref{thm.clustering1} hold.
\end{enumerate}
First of all (since $a\geq 1$) we have
\[
M^{Y}_n = \max_{k\leq n} Y_n = \max \{X_1,aX_2,aX_3,\ldots,
aX_{n+1}\}.
\]
Now we compute the extremal index for the process $Y_n$. Let
$\overline{F}_{Y}(y):= 1 - P(Y_1<y)$. Then
\begin{align*}
  \overline{F}_{Y}(y) &= 1 - P(Y_n<y) \\ & = 1 - P(X_n<y,X_{n+1}<y/a) \\ &=
  1 - (1-y^{-1})(1-ay^{-1})\\ &=\frac{a+1}{y}+O(y^{-2}).
\end{align*}
Given $\tau>0$ let $u_n=(a+1)n/\tau$. Then this sequence $u_n$
satisfies $nP(Y_1>u_n) \to \tau.$ Now consider the distribution for
the process $M^{Y}_n$. We have
\begin{align*}
  P(M^{Y}_n<u_n) &= P(X_1<u_n)P(X_2<u_n/a) \ldots
  P(X_{n+1}<u_n/a)\\ &= \biggl( 1 - \frac{1}{u_n} \biggr) \biggl(
  1 - \frac{a}{u_n} \biggr)^n\\ &= \biggl( 1 -
  \frac{\tau}{(a+1)n} \biggr) \biggl( 1 - \frac{a\tau}{(a+1)n}
  \biggr)^n \\ & \to e^{-\theta\tau}, \qquad n\to\infty,
\end{align*}
with $\theta=a/(a+1)$.  Consider almost sure bounds for the
maxima of $(X_n)$ via general sequences $\tilde{u}_n$ and
$\tilde{v}_n$, so that $\tilde{v}_n<M^{X}_n<\tilde{u}_n$ (almost
surely). Consider also the intermediate sequence $\tilde{w}_n$
with $P(M^{X}_n>\tilde{w}_n \text{ i.o.}) = P(M^{X}_n<\tilde{w}_n
\text{ i.o.}) = 1$. By definition of the process $(Y_n)$ we see
that (almost surely), for all $n$ sufficiently large
\[
a\tilde{v}_{n+1}<M^{Y}_n<a\tilde{u}_{n+1},
\]
and (almost surely) there are infinite subsequences
$\tilde{w}_{n_k}$ with
\[
M^{Y}_{n_k} > a\tilde{w}_{1+n_k},\qquad M^{Y}_{n_k} < a
\tilde{w}_{1+n_k}.
\]
Using the explicit regularity of the tails $\overline{F}_X(x)$,
and $\overline{F}_Y(y)$, we see that divergence or convergence of
the sum
\begin{equation}\label{eq.iidseries}
  \sum_{n}\overline{F}_X(z_n)e^{-n\overline{F}_X(z_n)}
\end{equation}
is equivalent to divergence and convergence of the sum
\begin{equation*}
  \sum_{n} \overline{F}_Y(az_n) e^{-\frac{an}{a+1}
    \overline{F}_Y(az_n)}.
\end{equation*}
We now apply this to $z_n=\tilde{v}_n$ and $z_n=\tilde{w}_n$. A
simple index relabelling of the series in equation
\eqref{eq.iidseries} shows that convergence/divergence is
unaffected by translation of the index $n$. Since
$e^{\overline{F}(z_n)} = e^{o(1)}$, divergence/convergence of
equation \eqref{eq.iidseries} is equivalent to
divergence/convergence of
\begin{equation*}
  \sum_{n} \overline{F}_X (z_{n-1}) e^{-n \overline{F}_X
    (z_{n-1})}.
\end{equation*}
This concludes the example.

\subsubsection{The Manneville--Pomeau map}\label{sec.mp-map}

In this section we consider how the method of inducing allows us
to extend the conclusions of Theorem~\ref{the:lowerbound1} and
Theorem~\ref{thm.clustering1} to a wider range of examples,
e.g.\ to dynamical systems whose transfer operator does not have
a spectral gap. We illustrate using the Manneville--Pomeau map $f
\colon [0,1] \to [0,1]$ given by
\begin{equation}\label{eq:LSV}
f(x)=
\begin{cases}
x(1+2^{a}x^{a}) &\mbox{ if } 0\leq x<1/2,\\
2x-1 &\mbox{ if } 1/2\leq x\leq 1,
\end{cases}
\end{equation}
with $a \in(0,1)$. In the following we let $A=[1/2,1]$, but the
construction we describe extends to any interval of the form
$[\epsilon,1]$ for $\epsilon>0$. See \cite{GHPZ,KKP}.  Consider the
first return map $\hat{f} \colon A\to A$ given by
$\hat{f}(x)=f^{R(x)}(x)$, with $R(x)= \inf \{\, n\geq 1 : f^n(x)\in A
\,\}$, (and $x\in A$).  The map $\hat{f}$ preserves an absolutely
continuous invariant measure $\hat{\mu}$, with density in
$\mathrm{BV}$, and has exponential decay of correlations of
$\mathrm{BV}$ against $L^1$.
Thus
Theorems~\ref{the:lowerbound1} and~\ref{thm.clustering1} apply to
$(\hat{f},\hat{\mu})$, for observables of the form
$\phi(x)=\psi(\dist(x,\tilde{x}))$ with $\tilde{x}\in A$. Let
$\widehat{M}_n(x)=\max_{k\leq n}\phi(\hat{f}^k(x))$, and (as before)
$M_n(x)=\max_{k\leq n}\phi(f^{k}(x))$. Then we have
\[
M_n(x)=\max_{j\leq k(n,x)}\phi(\hat{f}^j(x))=\widehat{M}_{k(n,x)}(x),
\]
where $x\in A$, and $k(n,x)$ satisfies
\[
\sum_{j=0}^{k(n,x)-1}R(\hat{f}^j(x))\leq
n<\sum_{j=0}^{k(n,x)}R(\hat{f}^j(x)).
\]
Since $f$ preserves an absolutely continuous invariant measure
$\mu$, we deduce that for $\mu$-almost all $x\in A$ that
$k(n,x)/n\to\mu(A)$.  We get the following result.

\begin{proposition}\label{prop.lsv-ind}
  Suppose $f \colon [0,1]\to[0,1]$ is given by equation
  \eqref{eq:LSV}. Let $A=[1/2,1]$ and $a \in (0,1)$. Let
  $\phi(x)=\psi(\dist(x,\tilde{x}))$, with $\psi \colon
        [0,\infty)\to\mathbbm{R}$ monotone decreasing. Then we have
          the following result for almost every $\tilde{x} \in A$.
  \begin{enumerate}
  \item Suppose that the
    sequence $(r_n)$ satisfies $\mu(B(\tilde{x},r_n))>c\log\log
    n/n$ for all $c>1$, then for all $\tilde{x}\in A$
    \[
    \mu (\mathcal{H}_{\mathrm{ea}}) = \mu\bigl( \psi(r_{n})\leq M_n
    \text{ ev.}\bigr) = 1.
    \]  
    
  \item Suppose that the
    sequence $(r_n)$ satisfies $\mu(B(\tilde{x},r_n))<c\log\log
    n/n$ for all $c<1$, then for all $\tilde{x}\in A$
    \[
    \mu (\mathcal{H}_{\mathrm{ea}}) = \mu\bigl( \psi(r_{n})\leq M_n
    \text{ ev.}\bigr) = 0.
    \]  
  \end{enumerate}
\end{proposition}

\begin{remark}
  We have stated the result only for $a \in (0,1)$. It is possible to
  consider also $a \geq 1$ where the $f$-invariant ergodic measure
  $\mu$ is no longer finite. However, we do not get significant
  improvements over results obtained in \cite{KKP} due to the (almost
  sure) fluctuations in $R(x)$.
\end{remark}

\begin{proof} 
  By Proposition~\ref{prop:A2forpwexp}, condition (A2) holds for
  $\hat{f}$ and $\mu$-a.e.\ $\tilde{x}\in[0,1]$. Thus $\hat{f}$
  satisfies the assumptions of Theorems~\ref{the:lowerbound1}.
  
  The proof then follows step by step that of
  \cite[Theorem~2]{KKP}. The only modification is that bounding
  sequences for $M_n(x)$ are determined via
  Theorem~\ref{the:lowerbound1} as
  applied to the map $\hat{f}$ to bound $\widehat{M_k}(x)$.  Here
  $k\equiv k(n,x)$ with $k(n,x)/n\to\mu(A)$ almost surely.
\end{proof}

\subsubsection{An example with extremal index
  $\theta=0$.}\label{sec.theta-zero}

We consider again the Manne\-ville--Pomeau map given by equation
\eqref{eq:LSV}, and take an observable function of the form
$\phi(x)=\psi(d(x,0))$, thus maximized at the point
$\tilde{x}=0$.  It is shown in \cite{FFTV} that the distribution
for the maxima follows a degenerate extreme value law with
extremal index equal to zero (under a scaling sequence given by
\eqref{eq.un-seq}). We obtain the following almost sure result.

\begin{proposition}\label{prop.lsv-ei}
  Suppose $f \colon [0,1]\to[0,1]$ is given by equation
  \eqref{eq:LSV} with $a \in (0,1)$, and $\phi(x)=\psi(\dist(x,0))$,
  with $\psi \colon [0,\infty)\to\mathbbm{R}$ monotone
    decreasing. Then we have the almost sure result $\mu(v_n\leq
    M_n\leq u_n \text{ ev.})=1$, where $u_n$ and $v_n$
    satisfy
  \[
  \mu(X_1>v_n) \geq \biggl(\frac{c\log\log n}{n}
  \biggr)^{1-a},\quad \mu(X_1>u_n) \leq
  \biggl(\frac{1}{n(\log n)^{2}} \biggr)^{1-a},
  \]
  with $c>c_0$, for some $c_0>0$ (depending on the density of
  $\mu$).
\end{proposition}

We prove Proposition~\ref{prop.lsv-ei} in
Section~\ref{sec.pf.prop.lsv-ei}.  This result also refines
estimates obtained in \cite{GHPZ,HNT2}, especially on the lower
bound sequence $v_n$. The upper bound sequence (as stated) can be
refined easily using the First Borel--Cantelli Lemma. This result
illustrates that when $\theta=0$, we can expect non-standard
growth rates for the maxima.  The exponent $1-a$ arises
due to the presence of the non-hyperbolic fixed point
$\tilde{x}=0$.

\subsubsection{On a general dichotomy criteria}

Within this section we have exhibited processes giving rise to a
non-trivial extremal index. These processes are created by using
underlying periodic phenomena of the dynamical system
process. More broadly, clustering can arise in more general
settings, see \cite{Embrechts,Letal}, and it is therefore natural
to ask whether the conclusion of Theorem~\ref{thm.clustering1} is
applicable in wider scenarios. We have given in
Section~\ref{sec.example.dich} a (non-dynamical) example to
illustrate that this is still the case. However, a general
criteria is yet to be fully developed on determining the
sequences $u_n$ for which a zero--one law applies to $\mu(M_n\leq
u_n\,\mathrm{i.o.})$. We consider the following question.

\begin{problem} For what class of stationary stochastic processes $(X_n)$ does the following hold?
There exists a constant
$\theta\in(0,1]$ such that
\begin{enumerate}
\item If $u_n$ is such that
  \[
  \sum_{n=1}^{\infty} \mu(X_1>u_n)=\infty \quad \textrm{and} \quad
  \sum_{n=1}^{\infty} \mu (X_1>u_n) e^{-n\theta\mu(X_1>u_n)} < \infty
  \]
  then $\mu(M_n\geq u_n \text{ ev.})=1$;
\item If $u_n$ is such that
  \[
  \sum_{n=1}^{\infty} \mu(X_1>u_n) = \infty \quad \textrm{and} \quad
  \sum_{n=1}^{\infty} \mu(X_1>u_n) e^{-n\theta\mu(X_1>u_n)} = \infty
  \]
  then $\mu(M_n\leq u_n \text{ i.o.}) = \mu (M_n\geq u_n
  \text{ i.o.}) = 1$.
\end{enumerate}
\end{problem}

For the i.i.d.\ case, these items are both valid for $\theta =
1$. However a full classification of processes $(X_n)$ which
satisfy these criteria for given $\theta \in (0,1)$ is
unknown. For certain maximum processes with extremal index $\theta
\in (0,1)$ we have shown that these dichotomy conditions apply.

\section{Almost sure bounds for $M_n$ for non-uniformly
  hyperbolic systems}\label{sec.alternative}
	
In this section we consider almost sure bounds on $M_n$
for a wide class of hyperbolic systems. These include systems where
$(f,\mathcal{X},\mu)$ is modelled by a Young tower \cite{Young}. We
emphasize that the results in this section are valid for some higher
dimensional systems, while results of previous sections are for
interval maps. For higher dimensional systems, e.g.\ such as those
that admit Sinai--Ruelle--Bowen measures, obtaining distributional
convergence of the maxima requires the blocking arguments used for
one-dimensional systems to be significantly modified. This
includes having additional regularity constraints placed on $\mu$
(depending upon the strength of result obtained). With the present
techniques available, we obtain results towards Case (1) of
Theorems~\ref{the:lowerbound3} and~\ref{thm.hyp1}. We consider the
following assumption on the distributional convergence.
\begin{enumerate}
\item[(A5)] Given constants $\sigma_1,\sigma_2>1$, there is a set
  $\mathcal{M}_r$, with $\mu(\mathcal{M}_r)\leq C_1|\log
  r|^{-\sigma_1}$, such that for all
  $\tilde{x}\not\in\mathcal{M}_r$,
  \begin{equation}\label{eq.mn-error2}
    \Biggl| \mu \Biggl(\, x : \sum_{k=0}^{t/\mu(B_r(\tilde{x}))}
    \mathbbm{1}_{B_r(\tilde{x})}(f^k(x))=0 \, \Biggr) -
    e^{-\theta t} \Biggr| \leq C_2|\log r|^{-\sigma_2},
  \end{equation}
  for all $t\geq 0$.
\end{enumerate}
Assumption (A5) is recognised as an approximate \emph{exponential law}
for entrance times to shrinking balls. For certain non-uniformly
hyperbolic systems modelled by Young towers, (A5) is shown to hold,
see for example \cite{CC, HW} (in the case $\theta=1$), where a more
general Poisson laws result can also be obtained.  Examples of such
systems include those with Axiom A attractors, and the H\'enon map
family for Benedicks--Carleson parameters \cite{BY}. To bring
assumption (A5) in line with distribution results for maxima $M_n$,
consider the observable function
$\phi(x)=\psi(\mathrm{dist}(x,\tilde{x}))$, and a sequence $r_n\to
0$. Set $u_n=\psi(r_n)$ and $t\equiv
t_n=n\mu(B(\tilde{x},r_n))$. Here we allow the possibility that
$t_n\to 0$ or $t_n\to\infty$.  Then equation \eqref{eq.mn-error2}
becomes
\begin{equation}\label{eq.mn-hyp-error}
  \Bigl| \mu \bigl(\, x : M_n(x) < u_n \,\bigl) - e^{-n \theta \mu (
    B(\tilde{x},r_n) )} \Bigr| \leq C_2|\log r_n|^{-\sigma_2}.
\end{equation}
This has similarities to the results obtained in
Section~\ref{sec.proof.preliminary} for one-dimensional systems, in
particular Corollary~\ref{cor.lem:Ml}. However, the approximation of
\eqref{eq.mn-hyp-error} is not uniform over the ball center
$\tilde{x}$. In order to mirror Corollary~\ref{cor.lem:Ml}, we require
$\tilde{x}\not\in\mathcal{M}_{r_n}$ for all $n$, and clearly this
condition depends on the sequence $(r_n)$. Further arguments are
therefore required to avoid the existence of an infinite subsequence
$(r_{j_k})$ for which $\tilde{x}\in\mathcal{M}_{r_{j_k}}$. For
hyperbolic systems, the presence of the set $\mathcal{M}_r$ arises
from the regularity assumptions (or lack thereof) placed on the
measure $\mu$. These issues are discussed in \cite{CC,HRS,HW}. When
stronger regularity properties are known (or assumed) about the
measure $\mu$, then it can be shown that $\mu(\limsup_{r\to
  0}\mathcal{M}_r)=0$. This applies for certain uniformly hyperbolic
systems and billiard models, see \cite{CFFHN,CNZ,GHN, PS}.

To state our result, we also need existence of a local dimension
$d_{\mu}(\tilde{x})$ at $\tilde{x}$. This is defined to be the
limit
\[
d_{\mu}(\tilde{x})=\lim_{r\to 0}\frac{\log
  \mu(B(\tilde{x},r))}{\log r},
\]
whenever this limit exists. For a wide range of hyperbolic
systems, the value $d_{\mu}(\tilde{x})$ exists for $\mu$-a.e.\
$\tilde{x}\in\mathcal{X}$, see \cite{BPS}.

\begin{theorem}\label{thm.hyp2}
  Suppose $(f,\mathcal{X},\mu)$ is a measure preserving system,
  and (A5) holds.  Given $\tilde{x}\in\mathcal{X}$, let
  $\phi(x)=\psi(\mathrm{dist}(x,\tilde{x}))$.  Suppose that $u_n$
  is a sequence such that $\mu(X_1>u_n)\geq c\theta^{-1}\log\log n/n$
  for some $c>1$, and that the local dimension
  $d_{\mu}(\tilde{x})$ exists.  Then for
  $\mu$-a.e.\ $\tilde{x}\in\mathcal{X}$ we have $\mu(M_n\geq
  u_n \text{ ev.})=1.$
\end{theorem}
We make several remarks on the proof and scope of this result. For the
proof of the result, we by-pass the influence of the set
$\mathcal{M}_r$ to obtain a result similar to
Corollary~\ref{cor.lem:Ml}. We can then apply the Cauchy-condensation
method used for proving Theorem~\ref{the:lowerbound3}. With the
current techniques available this is the best we can achieve. The
arguments used within Section~\ref{sec.pf.main-case2} cannot easily be
adapted and new ideas are needed.

Within Theorem~\ref{thm.hyp2} we require $\sigma_1,\sigma_2>1$. However, in
certain applications the possibility $\sigma_1,\sigma_2\in(0,1)$ can
arise, see \cite{HW}.  In this case, we get weaker bounds on sequence
$u_n=\psi(r_n)$, namely having the requirement $\mu(X>u_n)\geq
e^{(\log n)^{\gamma'}}/n$ for some $\gamma'\in(0,1)$, see
Section~\ref{sec.pf.thm.hyp2}.

A further remark is that having an assumption on the existence of a
local dimension can be weakened. From the proof, we generally require
quantitative bounds on the decay of $\mu(B(\tilde{x},r_n))$ along
certain sequences $r_n\to 0$.

\subsection{Intermediate growth rate sequences for maxima}\label{sec.intermediate}

In this section we consider non-decreasing sequences $(u_n)$ for
which the following statement applies
\[
\mu(M_n > u_n \text{ i.o.})>0 \qquad \text{and} \qquad
\mu(M_n\leq u_n \text{ i.o.})>0.
\]
Clearly the dichotomy results obtained in
e.g.\ Theorems~\ref{the:lowerbound1} and \ref{thm.clustering1}
fully describe these sequences. However, in the case of the
dynamical systems for which Theorem~\ref{thm.hyp2} applies, we
can obtain partial results using ergodicity of the dynamical
system. Since $(u_n)$ is non-decreasing, we have that $\{X_n >
u_n \text{ i.o.}\}$ is invariant mod $\mu$. Moreover, if the set
$\{ \tilde{x}\}$ has zero measure, we have $\{M_n > u_n \text{
  i.o.}\} = \{X_n > u_n \text{ i.o.}\}$ mod $\mu$.  It then
follows by ergodicity that if $\mu(M_n > u_n \text{ i.o.})>0$,
then $\mu(M_n > u_n \text{ i.o.})=1$.  On the other hand,
$\mu(M_n\leq u_n \text{ i.o.})>0$ gives
$\mu(\mathcal{H}_{\mathrm{ea}}(\mathbf{B})) = \mu(M_n> u_n
\text{ ev.})<1$. Thus, $\mu(M_n> u_n \text{ ev.})=0$ by
ergodicity and invariance of
$\mathcal{H}_{\mathrm{ea}}(\mathbf{B})$ (see
\cite[Lemma~1]{KKP}). This yields $\mu(M_n\leq u_n \text{
  i.o.})=1$. Hence, both of the measures above will be 1 (if
they are positive). We state the following result whose proof is
similar to that of \cite[Proposition~2]{KKP}.

\begin{theorem}\label{thm.intermediate}
  Suppose that $(f,\mathcal{X},\mu)$ is an ergodic measure
  preserving system satisfying (A5).  Consider the observable
  $\phi(x)=\psi(\dist(x,\tilde{x}))$ with $\psi(y)\to\infty$ as
  $y\to 0$.  Suppose that $B(\tilde{x},r_n)$ are such that
  $\mu(B(\tilde{x},r_n))\leq c/n$, for $c>0$.  For the sequence
  $u_n=\psi(r_n)$, we have
  \[
  \mu(M_n\leq u_n \text{ i.o.})=1,
  \]
  that is,
  \[
  \mu(\mathcal{H}_{\mathrm{ea}}) = \mu(M_n> u_n \text{ ev.})=0.
  \]
\end{theorem}

\begin{proof}
  We may assume that $\mu(B(\tilde{x},r_n)) = c/n$, and hence
  that $(u_n)$ is non-decreasing. Assumption (A5) yields that
  (\ref{eq.mn-hyp-error}) holds true. Hereby, we get $\mu(M_n\leq
  u_n)\to e^{-c'} >0$ for some $c'<\infty$. Furthermore, we
  recall that
  \[
  \mu(\mathcal{H}_{\mathrm{ea}}) = \mu(M_n> u_n \text{
    ev.})=\mu \Biggl( \bigcup^{\infty}_{i=1}
  \bigcap^{\infty}_{n=i} \{M_n \leq u_n\}^\complement \Biggr).
  \]
  We observe that $\mu(M_n> u_n \text{ ev.})= \lim_{i\to
    \infty} \mu(\bigcap^{\infty}_{n=i} \{M_n \leq u_n\}^c) \leq
  1-e^{-c'}<1$ by nestedness.  By ergodicity and invariance mod
  $\mu$ of $\mathcal{H}_{\mathrm{ea}}$ \cite[Lemma~1]{KKP}, we
  conclude the statement.
\end{proof}

In the theorem above, it is possible that we have $\mu(M_n\leq
u_n \text{ ev.})=1.$ However if
$\sum_{n}\mu(B(\tilde{x},r_n))=\infty$, and $B(\tilde{x},r_n)$ is
a dynamical Borel--Cantelli sequence, then we instead have
$\mu(M_n> u_n \text{ i.o.})=1$. For a wide class of dynamical
systems, and $\mu$-typical $\tilde{x}$ this property is known to
hold, see \cite{Athreya, ChernovKleinbock, GNO, HNPV, HNT2}.

\section{Overview of the proofs and the blocking
  argument}\label{sec.overview-block}

Here we give an overview of our proofs. We first use an argument by
Galambos: From a dynamical Borel--Cantelli Lemma (for instance
\cite{Athreya, ChernovKleinbock, Kim}) we get from $\sum \mu (B
(\tilde{x}, r_n)) = \infty$ that almost surely $X_n \geq u_n$
infinitely often, and hence that $M_n \geq u_n$ holds infinitely often
almost surely. Therefore the set $\{ M_n < u_n \text{ ev.}  \}$ has
measure zero. We obtain
\begin{align*}
  \mu ( M_n < u_n \text{ i.o.} ) &= \mu ( \{ M_n < u_n \text{
    i.o.}  \} \setminus \{ M_n < u_n \text{ ev.} \}) \\ &= \mu \bigl(
  M_n < u_n \text{ and } M_{n+1} \geq u_{n+1} \text{ i.o.}  \bigr)
  \\ &= \mu \bigl( M_n < u_n \text{ and } X_{n+1} \geq u_{n+1} \text{
    i.o.} \bigr).
\end{align*}
We will use this equality to prove that $\mu ( M_n < u_n
\text{ i.o.} ) = 0$ and hence that $\mu ( M_n \geq u_n \text{
  ev.} ) = 1$. The idea is to use that for $l < n$
\[
\{ M_n < u_n \text{ and } X_{n+1} \geq u_{n+1} \} \subset \{ M_l
< u_n \text{ and } X_{n+1} \geq u_{n+1} \},
\]
and if $n - l$ is large, then
\begin{align}
  \mu \bigl( M_n < u_n \text{ and } X_{n+1} \geq u_{n+1} \bigr) & \leq
  \mu\bigl( M_l < u_n \text{ and } X_{n+1} \geq u_{n+1} \bigr) \nonumber
  \\ &\approx \mu ( M_l < u_n ) \mu ( X_{n+1} \geq u_{n+1}),
  \label{eq:splitMandX}
\end{align}
and this will be made precise using decay of correlation
estimates. Then, Proposition~\ref{prop:blocking} below is used to
estimate $\mu (M_l < u_n )$. This results in an estimate on
$\mu ( M_n < u_n \text{ and } X_{n+1} \geq u_{n+1} )$. Using
this estimate, it is shown that $\sum \mu( B(\tilde{x}, r_n))
\exp (- n \mu (B(\tilde{x},r_n)) ) < \infty$ implies that
\[
\sum_{n=1}^\infty \mu (M_n < u_n \text{ and } X_{n+1} \geq
u_{n+1} ) < \infty
\]
and this implies by Borel--Cantelli that
\[
\mu (M_n < u_n \text{ and } X_{n+1} \geq u_{n+1} \text{ i.o.}
) = 0.
\]
Hence $\mu ( M_n < u_n \text{ i.o.} ) = 0$ and $\mu ( M_n \geq
u_n \text{ ev.} ) = 1$. In this way we obtain the proof
of Theorem~\ref{the:lowerbound1}.

For some systems, it is difficult to get a good enough error
bound in the approximation
\[
\mu (M_l < u_n \text{ and } X_{n+1} \geq u_{n+1} ) \approx \mu
( M_l < u_n ) \mu ( X_{n+1} \geq u_{n+1} )
\]
which was used in \eqref{eq:splitMandX}. For such systems we use
instead the estimate
\[
  \mu (M_n < u_n \text{ and } X_{n+1} \geq u_{n+1} ) \leq \mu
  ( M_n < u_n ).
\]
In the end this only leads to a slightly weaker result. This is
the path taken to prove Theorem~\ref{the:lowerbound3}.  

We will now explain how the so-called blocking argument is used
to estimate $\mu (M_l < u )$. For general stochastic processes see
\cite{Embrechts, LLR}. Relevant to dynamical systems, the approach we describe is adapted 
from~\cite{Collet}. 

We have
\[
\{ M_l < u \} = \bigcap_{k=1}^l \{ X_k < u \} \subset \bigcap_{k
  \in I_l} \{ X_k < u \},
\]
where $I_l \subset \{1,2,\ldots,l\}$. We let $I_l$ consist of $q$
blocks of $p$ consecutive numbers in $\{1,2,\ldots, l\}$, each
block separated by $t$ numbers. Writing $I_l = \bigcup_{j=1}^q
J_j$ where $J_j$ are the blocks, we have
\[
\{ M_l < u \} \subset \bigcap_{j = 1}^q \bigcap_{k \in J_j} \{
X_k < u \}.
\]
The measure of $\bigcap_{k \in J_j} \{ X_k < u \}$ is
approximated by $1 - p \mu (X_1 \geq u) = 1 - p \mu (\phi \geq
u)$ and in this way we can estimate $\mu ( M_l < u )$ by $(1 -
p \mu (\phi \geq u))^q$. The error obtained by this estimate is
expressed in the following proposition by Collet.

\begin{proposition}[Collet] \label{prop:blocking}
  Suppose that $\mu$ is an ergodic probability measure. Let $s
  \in (0,\frac{1}{2}]$. With $l = q p + r$, $p = [l^s]$, $0 \leq
    r < p$ and $l$ large, we have for any $u$ and $t \in
    \mathbbm{N}$ that
  \[
  \bigl| \mu ( M_l < u ) - (1 - p \mu ( \phi \geq u ))^q
  \bigr| \leq \sum_{j = 1}^q (1 - p \mu (\phi \geq u ))^{q -
    j} \Gamma_j,
  \]
  where $\Gamma_{j}$ is given by
  \begin{multline*}
  \Gamma_j = \bigl| p \mu (\phi \geq u ) \mu (M_{(j - 1) (p +
    t)} < u ) - \Sigma_j \bigr| \\ + t \mu ( \phi \geq u ) + 2
  p \sum_{k=1}^p \expectation (\mathbbm{1}_{\phi \geq u}
  \mathbbm{1}_{\phi \geq u} \circ f^k)
  \end{multline*}
  and
  \[
  \Sigma_j = \sum_{k=1}^p \expectation ( \mathbbm{1}_{\phi \geq u }
  \mathbbm{1}_{M_{ (j - 1) (p + t)} < u } \circ f^{p + t - k} ),
  \]
	where $\expectation(\cdot)$ denotes expectation.
\end{proposition}

The proof of Proposition~\ref{prop:blocking} is by Collet
\cite{Collet}. For completeness, we include it in the appendix.

\section{Application of Proposition~\ref{prop:blocking} and preliminary estimates}\label{sec.proof.preliminary}

In this section we collect several key estimates that we use for proving
the main results.  We start with an immediate consequence of
assumptions (A1) and (A3), where the decay rate is given by
$\Theta(j)=\exp(-\tau j)$. This result will be useful for proving
Theorem~\ref{the:lowerbound3}.
\begin{lemma} \label{lem:split}
  There is a constant $c_1$ such that for $l \leq n$,
  \[
  \mu ( M_n \leq u_n \text{ and } X_{n+1} > u_{n+1} ) \leq \mu
  ( M_l \leq u_n ) \mu ( X_{n+1} > u_{n+1} ) + c_1 K^l
  e^{-\tau n}.
  \]
\end{lemma}

\begin{proof}
  Let $\varphi_1 (x) = \mathbbm{1}_{\{M_{l}<u_n\}}(x)$ and
  $\varphi_2(x) = \mathbbm{1}_{\{X_1>u_{n+1}\}}$. We estimate the
  $\mathrm{BV}$-norm of $\varphi_1(x)$. Since for any interval
  $A$, $f^{-1}(A)$ has at most $K_{f}$ connected components (by
  (A3)), it follows that the $\mathrm{BV}$-norm of $\varphi_1(x)$
  is bounded by $K^l$, for some constant $K$.

  Using decay of correlations, we get that
  \begin{align*}
    \mu ( M_n \leq u_n \text{ and } & X_{n+1} > u_{n+1} ) \leq
    \int \varphi_1 \varphi_2 \circ f^n \, \mathrm{d} \mu \\ &
    \leq \int \varphi_1 \, \mathrm{d} \mu \int \varphi_2 \,
    \mathrm{d} \mu + C e^{-\tau n} \lVert \varphi_1
    \rVert_{\mathrm{BV}} \lVert \varphi_2 \rVert_\infty \\ &\leq
    \mu ( M_l \leq u_n ) \mu ( X_{n+1} > u_{n+1} ) + c_1 K^l
    e^{-\tau n}. \qedhere
  \end{align*}
\end{proof}

In the next step, we use Collet's blocking argument from
Proposition~\ref{prop:blocking} and the assumptions (A1) and (A2) to
obtain an estimate on $\mu ( M_l < u_n )$. The following lemma and
subsequent corollaries will be used in the proof of most of the
Theorems stated within Sections~\ref{sec.results} and
\ref{sec.clustering}. With a slight change of notation, for integers
$p,n\geq 1$ we take
\[
\Xi_{p,n}\equiv\Xi_{p,n}(u_n) := \sum_{j=1}^{p} \mu(\phi>u_n,\phi\circ
f^j>u_n).
\]
This is consistent with the notation of (A2). That is for the sequence
$r_n$ as defined in (A2), we have $u_n=\psi(r_n)$. Recall that the
observable function is $\phi(x)=\psi(\mathrm{dist}(x,\tilde{x}))$.
\begin{lemma}\label{lem:Ml}
  We assume (A1) with rate function $\Theta(j)=\exp(-\tau j)$. Then
  there are constants $C$, $c_1>0$ such that for $l\leq n$
  sufficiently large, $p=[l^s]$ and $t=[\log(l^{\frac{2}{\tau}})]$ we
  have
  \begin{equation}\label{eq:muMl}
    \mu ( M_l < u_n ) \leq e^C \exp (-l \mu ( \phi \geq u_n )+
    qt\mu ( \phi \geq u_n ) + 
		n\Xi_{p,n}+\frac{c_2}{l^{1+s}}. 
    \end{equation}
\end{lemma}

\begin{proof}
  We start by estimating $\Gamma_j$ in
  Proposition~\ref{prop:blocking}. First of all we let $p\leq n$. By
  the decay of correlation alone, we have
  \begin{align} \label{eq:Gammaestimate2}
    \bigl| \Sigma_j & - p \mu (\phi \geq u_n ) \mu (M_{(j-1)
      (p + t)} < u_n ) \bigr| \\ & = \biggl| \sum_{k=1}^p \bigl(
    \expectation ( \mathbbm{1}_{\phi \geq u_n } \mathbbm{1}_{M_{
        (j - 1) (p + t)} < u_n } \circ f^{p + t - k} ) \nonumber
    \\ & \phantom{xxxxxxxxxx} - \mu (\phi \geq u_n ) \mu
    (M_{(j-1) (p + t)} < u_n ) \bigr) \biggr| \nonumber \\ &
    \leq \sum_{k=1}^p 3 C e^{-\tau (p + t - k)} \leq c_0 e^{-
      \tau t}. \nonumber
  \end{align}
  Hence, by \eqref{eq:Gammaestimate2} and definition of $\Xi_{p,n}$, we have
  \[
    \Gamma_j \leq t\mu ( \phi \geq u_n ) +2p\Xi_{p,n} + c_0 e^{- \tau t}.
  \]

  As in Proposition~\ref{prop:blocking}, we take $p = [l^s]$ with
  $s>0$. Furthermore, we let $t = [\log
    (l^{\frac{2}{\tau}})]$.  We therefore have
  \[
  \Gamma_j \leq\tilde\Gamma:= t\mu ( \phi \geq u_n )
  +2p\Xi_{p,n}+\frac{c_2}{l^2}.
  \]
  for some constant $c_2$.
	
  Proposition~\ref{prop:blocking} now implies that
  \[
  \mu ( M_l < u_n ) \leq (1 - p \mu ( \phi \geq u_n ))^q +
  \tilde{\Gamma} \sum_{j=0}^{q-1} (1 - p \mu ( \phi \geq u_n ))^j.
  \]

  Using that for $0<x<1$
  \[
  (1 - x)^k = \exp (k \log (1 - x)) \leq \exp (- k x),
  \]
  we obtain that
  \begin{align*}
    \mu ( M_l < u_n ) & \leq \exp (-pq \mu ( \phi \geq u_n )
    ) + \tilde{\Gamma} \sum_{j = 0}^{q-1} \exp( - j p \mu ( \phi
    \geq u_n )) \\ & \leq \exp (-pq \mu ( \phi \geq u_n ) ) +
    q \tilde{\Gamma}.
  \end{align*}
  Hence
  \begin{equation*}
    \mu ( M_l < u_n ) \leq \exp (-pq \mu ( \phi \geq u_n )
    ) + qt\mu ( \phi \geq u_n ) + 2pq\Xi_{p,n}+\frac{c_2q}{l^2}.
   \end{equation*}
  
  Since $p = [l^s]$ we have $r \leq l^s$ and $p q \geq l -
  l^s$. We obtain
  \begin{align*}
    \mu ( M_l < u_n ) & \leq \exp (- (l - l^s) \mu ( \phi \geq u_n
    ))+ qt\mu ( \phi \geq u_n )
    +2pq\Xi_{p,n}+\frac{c_2}{l^{1+s}}\\ & \leq e^{C}\exp (- l\mu (
    \phi \geq u_n ))+ qt\mu ( \phi \geq u_n
    )+2pq\Xi_{p,n}+\frac{c_2}{l^{1+s}}. \qedhere
  \end{align*}
\end{proof}
So far we have not used (A2). The next corollary gives a key
bound that we use in the proof of the main results.  It further
quantifies the $e^C$ multiplier in equation \eqref{eq:muMl}, and hence
gives an error bound for  estimating $\mu(M_l<u_n)$ in terms of
$e^{-l\mu(\phi>u_n)}$.
\begin{corollary}\label{cor.lem:Ml}
  Suppose the hypothesis of (A2) holds with parameters $\sigma,\gamma$
  and $s$.  With the notations of Lemma~\ref{lem:Ml} we have
  \begin{multline}
    |\mu ( M_l < u_n )-\exp (-l \mu ( \phi \geq u_n ) )| \leq
    qt\mu ( \phi \geq u_n ) +
    n\Xi_{p,n}\\ +\frac{c_2}{l^{1+s}}+c_3l^s\mu(\phi>u_n). \label{eq:muMlb}
  \end{multline}
  Here $c_3>0$ is a constant. Furthermore, there exists $\gamma'>0$
  such that for all $\beta>0$ and $l>\beta n$
  \begin{equation}\label{eq.cor.error1}
    |\mu ( M_l < u_n )-\exp (-l \mu ( \phi \geq u_n ) )| \leq
    C_{\beta} n^{-\gamma'}.
  \end{equation}
  Here, the constant $C_{\beta}>0$ depends on $\beta$.
\end{corollary}

\begin{proof}
  In fact, from Proposition~\ref{prop:blocking} and as in the
  proof of Lemma~\ref{lem:Ml}, we get that
  \[
    | \mu (M_l < u_n ) - \exp (- (l - l^s) \mu ( \phi \geq u_n
    ) )|\leq qt\mu ( \phi \geq u_n ) + n\Xi_{p,n} +
    \frac{c_2}{l^{1+s}}.
  \]
  Noting that $\mu(\phi>u_n)=O(n^{-\sigma})$, the constraint
  $\sigma>(2+s)/3$ from (A2) implies that $l^s \mu ( \phi \geq u_n)$
  is bounded. Hence,
  \begin{multline*}
    | \mu (M_l < u_n ) - e^{-l \mu ( \phi \geq u_n ) }| - |
    \mu (M_l < u_n ) - e^{- (l - l^s) \mu ( \phi \geq u_n )
    }| \\ \leq | e^{- l\mu ( \phi \geq u_n ) } - e^{- (l
    - l^s) \mu ( \phi \geq u_n ) }|  \leq |1 - e^{l^s \mu (
      \phi \geq u_n)}| \leq c_3 l^s \mu ( \phi \geq u_n ).
  \end{multline*}
  This gives the equation \eqref{eq:muMlb} stated in the
  corollary. For equation \eqref{eq.cor.error1}, the existence of the
  constant $\gamma'>0$ follows from equation
  \eqref{eq.short-constants}. Indeed, to see this we consider each
  right-hand term of \eqref{eq:muMlb}, and note that (by hypothesis)
  $l>\beta n$ for some $\beta>0$, and hence $c_2l^{-1-s}$ within
  \eqref{eq:muMlb} is $O(n^{-1-s})$.  By (A2), we have
  $n\Xi_{p,n}<n^{-\gamma}$. Similarly we have $qt\mu(\phi>u_n)=
  O((\log n)n^{1-s-\sigma})$. The latter term is $O(n^{-\gamma_1})$
  for some $\gamma_1$. This follows from the constraint $\sigma>1-s/2$
  in \eqref{eq.short-constants}. We have already considered the term
  $c_3l^s \mu ( \phi \geq u_n)$. Hence this completes the proof.
\end{proof}

We state the following further corollary, which is an easy consequence
of the results developed so far.

\begin{corollary}\label{cor.intermediate}
  Suppose the hypothesis of (A2) holds with parameters $\sigma,\gamma$
  and $s$. Suppose that
  \[
  \sum_{n = 1}^\infty \mu ( \phi \geq u_n ) \exp (- \beta n \mu (
  \phi \geq u_n )) < \infty,
  \]
  for some $\beta < \frac{\tau}{\log K}$, where $K$ is the constant in
  Lemma~\ref{lem:split}.  Then
  \[
  \mu(M_n\geq u_n\,\mathrm{ev.})=1.
  \]
\end{corollary}
We remark that the conclusion of this corollary is not optimal
relative to the statements within Theorems~\ref{the:lowerbound1} and
\ref{the:lowerbound3}. The ideas presented here will be used in the
proofs of these theorems, but optimised accordingly. We also clarify
within the need for the constrains imposed by equation
\eqref{eq.short-constants} within (A2).

\begin{proof}[Proof of Corollary~\ref{cor.intermediate}]
  Combining Lemma~\ref{lem:split} and Lemma~\ref{lem:Ml} we now
  get that
  \begin{multline} \label{eq:split}
    \mu ( M_n < u_n \text{ and } X_{n+1} > u_{n+1} ) \\ \leq e^C
    \exp (-l \mu ( \phi \geq u_n ) ) \mu ( \phi \geq u_n
    )+n\Xi_{p,n}\mu(\phi>u_n) \\ +qt\mu ( \phi \geq u_n )^2 +
    \frac{c_2}{l^{1+s}}\mu(\phi>u_n) + c_1 K^l e^{-\tau n}.
  \end{multline}
  
  We take $p=n^{s}$, $pq\approx n$ and $l = [\beta n]$ where $\beta <
  \frac{\tau}{\log K}$. This makes the term $c_1 K^l e^{-\tau n}$
  summable over $n$. Also, the term $c_2 l^{-1-s}\mu(\phi>u_n)$ is
  summable over $n$.

  For the term $qt\mu ( \phi \geq u_n )^2$, the relation
  $\sigma>1-s/2$ implies this is summable over $n$ (noting that the
  contribution from $t$ is $O(\log n)$). Consider the term
  $n\Xi_{p,n}\mu(\phi>u_n)$. By (A2), this term is summable by the
  assumption $\Xi_{p,n}<n^{-1-\gamma}$, and the fact that
  $\sigma>1-\gamma$. However, we still have to check a
  self-consistency condition involving $\sigma$ and $s$, since we also
  know by exponential decay of correlations (A1) that
  $\Xi_{p,n}>c'n^s\mu(\phi>u_n)^2$ for some $c'>0$. By equation
  \eqref{eq.short-constants}, we have $\sigma>(2+s)/3$, and therefore
  it follows that $c'n^{s+1}\mu(\phi>u_n)^3$ is also summable.

  Hence, we have showed that $\mu ( M_n < u_n \text{ and } X_{n+1}
  \geq u_{n+1} )$ is summable provided that
  \[
  \sum_{n = 1}^\infty \mu ( \phi \geq u_n ) \exp (- \beta n \mu (
  \phi \geq u_n )) < \infty,
  \]
  for some $\beta < \frac{\tau}{\log K}$, which finishes the proof.
\end{proof}

\section{Proof of Theorem~\ref{the:lowerbound1}, Case
  (1).}\label{sec.pf.main-case1}

The proof of Case~(1) in Theorem~\ref{the:lowerbound1} follows the
same ideas of Section~\ref{sec.proof.preliminary} that led to
Corollary~\ref{cor.intermediate}. The only thing missing is that
Lemma~\ref{lem:split} need not be true since $\mathbbm{1}_{M_l < u}$
is not of bounded variation. The use of Lemma~\ref{lem:split} is
therefore replaced by the following lemma.

\begin{lemma} \label{lem:gausssplit}
  \begin{multline*}
    \mu (M_l < u \text{ and } X_{n+1} \geq u ) \\ \leq \mu (
    M_l < u ) \mu ( X_{n+1} \geq u ) + c_1 l e^{- \tau (n-l)}
    \mu ( X_{n+1} \geq u ).
  \end{multline*}
\end{lemma}

\begin{proof}
  We let $\mathscr{L}$ be the transfer operator
  \[
  \mathscr{L} \psi (x) = \sum_{f(y) = x} g (y) \psi(y),
  \]
  where $\log g$ is the potential of the Gibbs measure. The
  assumptions of the theorem mean the following. 
  The density $h$ of $\mu$
  with respect to the conformal measure $\nu$ is an eigenfunction
  of $\mathscr{L}$ with eigenvalue $\lambda = e^P$, where $P$ is
  the pressure. The eigenfunction $h$ is of bounded variation.

  The operator $\mathscr{L}$ has the following useful properties (see
  \cite{Liveranietal}). It satisfies $\int \psi \, \mathrm{d} \nu =
  \lambda^{-1} \int \mathscr{L} (\psi) \, \mathrm{d} \nu$. If $\int
  \psi \, \mathrm{d} \nu = 0$, then
  \[
  \lambda^{-n} \sup | \mathscr{L}^n (\psi)| \leq \lambda^{-n}
  \lVert \mathscr{L}^n (\psi) \rVert_\mathrm{BV} \leq C \lVert
  \psi \rVert_\mathrm{BV} e^{- \tau n}.
  \]
  
  Using the first of these properties, we have
  \begin{align*}
    \int (\mathbbm{1}_{M_l < u} - c) \mathbbm{1}_{X_{n+1} \geq u}
    \, \mathrm{d} \mu &= \lambda^{-n} \int \mathscr{L}^{n} (
    (\mathbbm{1}_{M_l < u} - c) \mathbbm{1}_{X_{n+1} \geq u} h
    )\, \mathrm{d}\nu \\ &= \lambda^{-n} \int \mathscr{L}^{n}
    ((\mathbbm{1}_{M_l < u} - c) h) \mathbbm{1}_{X \geq u} \,
    \mathrm{d} \nu \\ & \leq \lambda^{-n} \sup |
    \mathscr{L}^{n} ((\mathbbm{1}_{M_l < u} - c) h) | \cdot \nu
    (X \geq u). 
  \end{align*}
  Letting $c = \mu ( M_l < u )$ this implies that
  \begin{multline} \label{eq:beforeClaim}
    \mu ( M_l < u \text{ and } X_{n+1} \geq u ) = \int
    \mathbbm{1}_{M_l < u} \mathbbm{1}_{X_{n+1} \geq u} \,
    \mathrm{d} \mu \\ \leq \mu ( M_l < u ) \mu ( X_{n+1} \geq
    u )  \\+ \lambda^{-n} \sup | \mathscr{L}^{n}
    ((\mathbbm{1}_{M_l < u} - c) h) | \cdot \nu (X \geq u).
  \end{multline}
  It remains to estimate the above supremum. We put $\psi =
  (\mathbbm{1}_{M_l < u} - c) h$.
  
  \begin{claim}
    We have $\lVert \mathscr{L}^l \psi \rVert_\mathrm{BV} \leq
    C_4 \lambda^l l$ for some constant $C_4$ that does not depend
    on $l$ or $u$.
  \end{claim}

\begin{proof}[Proof of Claim]
  Clearly, we have
  \begin{equation} \label{eq:supLpsi}
  \sup | \mathscr{L}^{l} (\psi)| \leq \sup | \mathscr{L}^{l}
  (h)| = \lambda^l\sup |h| < \infty.
  \end{equation}

  We shall now estimate $\var \mathscr{L}^{l} (\psi)$. Since
  $\mathbbm{1}_{M_l < u} = \prod_{k=0}^{l-1} \mathbbm{1}_{X < u}
  \circ f^k$, we have
  \[
  \mathscr{L}^l (\mathbbm{1}_{M_l < u} h) (x) = \sum_{f^l (y) = x}
  g_l (y) h(y) \prod_{k=0}^{l-1} \mathbbm{1}_{X < u} (f^k y),
  \]
  where $g_l (y) = g(y) g(f (y)) \ldots g(f^{l-1}(y))$.  We let
  $(f^l)_j$ denote the branches of $f^l$ and write
  \[
  \mathscr{L}^l (\mathbbm{1}_{M_l < u} h) (x) = \sum_j g_l
  ((f^l)_j^{-1} (x)) h( (f^l)_j^{-1} (x)) \prod_{k=0}^{l-1}
  \mathbbm{1}_{X < u} (f^k ((f^l)_j^{-1} (x)).
  \]
  Then
  \begin{multline*}
    \var \mathscr{L}^l (\mathbbm{1}_{M_l < u} h) (x) \\ \leq
    \sum_j \var \biggl( g_l ( (f^l)_j^{-1} (x)) h( (f^l)_j^{-1}
    (x)) \prod_{k=0}^{l-1} \mathbbm{1}_{X < u} (f^k ((f^l)_j^{-1}
    (x)) \biggr).
  \end{multline*}

  Since $k < l$ we have
  \[
  \var \mathbbm{1}_{X < u} (f^k \circ (f^l)_j^{-1}) \leq \var
  \mathbbm{1}_{X < u} = 2.
  \]
  Using that
  \[
  \var (\phi \psi) \leq \sup |\phi| \var \psi + \sup |\psi| \var
  \phi,
  \]
  this implies that
  \[
  \var \prod_{k=0}^{l-1} \mathbbm{1}_{X < u} (f^k \circ
  (f^l)_j^{-1}) \leq 2l.
  \]

  Let
  \[
  G_j (x) = g_l ( (f^l)_j^{-1} (x)) h( (f^l)_j^{-1} (x))
  \]
  and
  \[
  F_j (x) = \prod_{k=0}^{l-1} \mathbbm{1}_{X < u} (f^k ((f^l)_j^{-1}
  (x))).
  \]
  With this notation, we have from above that
  \[
  \var \mathscr{L}^l (\mathbbm{1}_{M_l < u} h) (x) \leq \sum_j
  \var ( G_j F_j) \leq \sum_j \bigl( \var G_j \sup F_j + \sup G_j
  \var F_j \bigr).
  \]
  Since
  \[
  \lambda^l h (x) = \mathscr{L}^l (h) (x) = \sum_j G_j (x),
  \]
  we have $\sum_j \var G_j = \lambda^l \var h$.
  
  Hence,
  \begin{align*}
    \var \mathscr{L}^l (\mathbbm{1}_{M_l < u} h) (x)
    & \leq \sum_j \bigl( \var G_j + \sup G_j \var F_j \bigr)\\
    & \leq C_0 \lambda^l + C_1 \lambda^l l \leq C_2 \lambda^l l,
  \end{align*}
  where the constant $C_2$ does not depend on $l$.

  Since $\var \mathscr{L}^l ((\mathbbm{1}_{M_l < u} - c) h) \leq
  \var \mathscr{L}^l (\mathbbm{1}_{M_l < u} h) + c \lambda^l \var
  h \leq C_3 \lambda^l l$, we have now proved with the aid of
  \eqref{eq:supLpsi} that
  \[
  \lVert \mathscr{L}^l ((\mathbbm{1}_{M_l < u} - c) h)
  \rVert_\mathrm{BV} \leq C_4 \lambda^l l
  \]
  for some constant $C_4$. 
  \end{proof}
  
  We recall that the constant $c$ was chosen so that $\int \psi
  \, \mathrm{d} \nu= \int (\mathbbm{1}_{M_l < u} - c) h \,
  \mathrm{d} \nu = \int (\mathbbm{1}_{M_l < u} - c) \, \mathrm{d}
  \mu = 0$. Hence we also have $\int \mathscr{L}^l \psi \,
  \mathrm{d} \nu = \lambda^{l}\int \psi \, \mathrm{d} \nu
  =0$. Then we estimate
  \begin{align*}
    \lambda^{-n} \sup | \mathscr{L}^{n} (\psi)| & =
    \frac{\lambda^{-l}}{\lambda^{n-l}} \sup | \mathscr{L}^{n-l}
    (\mathscr{L}^l\psi)| \leq \lambda^{-l} \cdot C \lVert
    \mathscr{L}^l \psi \rVert_\mathrm{BV} \cdot e^{-\tau (n-l)}
    \\ & \leq C_5 l e^{-\tau (n-l)}.
  \end{align*}
   
  It follows from \eqref{eq:beforeClaim} that
  \begin{multline*}
    \mu (M_l < u \text{ and } X_{n+1} \geq u ) \\ \leq \mu (
    M_l < u ) \mu ( X_{n+1} \geq u ) + C_5 l e^{- \tau (n-l)}
    \nu ( X \geq u ).
  \end{multline*}
  Finally, since $h$ is the density of $\mu$ with respect to
  $\nu$, and $h$ is bounded, we have $C_3 \nu ( X \geq u ) \leq
  c_1 \mu ( X \geq u ) = c_1 \mu ( X_{n+1} \geq u )$ for some
  constant $c_1$.
\end{proof}

Following the proofs in Section~\ref{sec.proof.preliminary} that led
to Corollary~\ref{cor.intermediate}, and using
Lemma~\ref{lem:gausssplit} instead of Lemma~\ref{lem:split}, we get
instead of \eqref{eq:split} that
\begin{multline} \label{eq:gausssplit}
  \mu ( M_n < u_n \text{ and } X_{n+1} > u_{n+1} ) \\ \leq e^C \exp
  (-l \mu ( \phi \geq u_n ) ) \mu ( \phi \geq u_n ) +
  n\Xi_{p,n}\mu(\phi>u_n)\\ + qt\mu ( \phi \geq u_n )^2 +
  \frac{c_4\mu(\phi>u_n)}{l^{1+s}} + c_1 l e^{-\tau (n-l)}.
\end{multline}
We take $l = n - [n^\beta]$, where $\beta < \sigma$, and
$p=n^{s}$, $pq\approx n$. This makes the term $c_1 l e^{-\tau
  (n-l)}$ as well as the term $c_4 l^{-1-s}$ in
\eqref{eq:gausssplit} summable over $n$. The other terms are
summable as in the proof of Corollary~\ref{cor.intermediate}.

The rest is the same as in the proofs outlined in
Section~\ref{sec.proof.preliminary}, and we obtain that $\mu (M_n <
u_n \text{ and } X_{n+1} \geq u_n )$ is summable provided that
\[
\sum_{n=1}^{\infty} \mu ( \phi \geq u_n) \exp (- n \mu ( \phi
\geq u_n )) < \infty.
\]
This finishes the proof.

\section{Proof of Theorem~\ref{the:lowerbound1}, Case
  (2)}\label{sec.pf.main-case2}

To prove Theorem~\ref{the:lowerbound1} we follow
\cite[Section~4]{Galambos}, in particular we follow the proof of
Theorem~4.3.2 within.  Given $\lambda>0$, consider the sequence
$a_n:=a(n)=\exp\{\lambda n/\log n\}$. This choice of sequence has
several properties which we elaborate on in the course of the
proof. Now, for a given sequence $(v_n)$, showing $\mu(M_{n}\leq v_{n}
\text{ i.o.})=1$ can be reduced to showing $\mu(M_{b_n}\leq v_{b_n}
\text{ i.o.})>0$ for some subsequence $b_n$.  This follows from a
zero--one law for eventually almost hitting sets under the assumption of ergodicity (see
\cite[Lemma~1]{KKP}).

The following reductions are elementary manipulations, and do not
depend on the precise form of $(a_n)$, nor on the dependency
structure of the process. To show $\mu(M_{n}\leq
v_{n} \text{ i.o.})=1$, we can first reduce this to finding $c>0$,
and $M_0$, such that for all $M\geq M_0$ we have
\[
\mu \Biggl( \bigcup_{n=M}^{\infty}\{M_{a_n}\leq v_{a_n}\}
\Biggr)\geq c.
\]
This can be reduced further to showing that for all $M>0$, there
exists $M'>M$ such that
\begin{equation}\label{eq.MM}
  \mu \Biggl(\bigcup_{n=M}^{M'}\{M_{a_n}\leq v_{a_n}\} \Biggr)
  \geq c.
\end{equation}
Now for arbitrary events $(A_n)$, we have
\[
\mu \Biggl( \bigcup_{n=M}^{M'}A_n \Biggr) = \sum_{n=M}^{M'}\mu
(A_n) - \sum_{n=M}^{M'} \mu \Biggl( A_n \cap \Biggl(
\bigcup_{i=n+1}^{M'}A_i \Biggr) \Biggr).
\]
Thus equation \eqref{eq.MM} holds if there exists $\Delta>0$,
independent of $M,M'$ such that
\[
\sum_{n=M}^{M'}\mu(M_{a_n}\leq v_{a_n})\geq \Delta>0,
\]
and $\delta<1$, such that for all $M_0\leq M\leq n\leq M'$,
\[
\mu \Biggl( M_{a_n}\leq v_{a_n}, \; \text{and} \;
\bigcup_{i=n+1}^{M'}\{M_{a_i}\leq v_{a_i}\}
\Biggr)\leq\delta\mu(M_{a_n}\leq v_{a_n}).
\]
Thus a requirement placed on the choice of sequence $(a_n)$ is
that
\begin{equation}\label{eq.MM2}
  \sum_{n=1}^{\infty}\mu(M_{a_n}\leq v_{a_n})=\infty,
\end{equation}
and
\begin{equation}\label{eq.MM3}
  \sum_{t=n+1}^{M'}\mu \bigl( \{M_{a_n}\leq
  v_{a_n}\}\cap\{M_{a_t}\leq v_{a_t}\} \bigr) \leq \delta
  \mu(M_{a_n}\leq v_{a_n}).
\end{equation}
In the i.i.d.\ case, these conditions are shown to hold for the
sequence $a_n=e^{\lambda n/\log n}$ for suitable $\lambda>0$.
The approach followed is that we can realise each term in
the sum of \eqref{eq.MM3} as the product
\begin{equation}\label{eq.MM4}
	\mu(M_{a_n} \leq v_{a_n}) \mu (M_{a_t-a_n} \leq
	v_{a_t}).
\end{equation}
This uses the fact that $v_{a_n}$ is non-decreasing. The
remainder of the proof in the i.i.d.\ case is elementary, and
uses further facts, such as
\begin{equation}\label{eq.MM5}
  \mu(M_{a_n}\leq v_{a_n}) = F_{X}(v_{a_n})^{a_n},
  \ \ \mu(M_{a_t-a_n} \leq v_{a_t}) =
  F_{X}(v_{a_t})^{a_t-a_n},
\end{equation}
where $F_X$ is the probability distribution function. For the
dependent case, we need to recover approximate versions of
\eqref{eq.MM4} and \eqref{eq.MM5}, and show that the same proof
goes through. This can be done using the mixing properties of the
dynamical system, and the blocking arguments. To do this, we
consider a further sequence $\ell(t)$, with $\ell(t)<a(t)-a(n)$,
and defined for $t>n$. Since $a_n=e^{\lambda n/\log n}$, we can
choose $\ell(t)$ to grow at various speeds, such as power law of
$t$.  The role of $\ell(t)$ is to de-correlate successive maxima
in the following way:
\begin{align}\label{eq.MM6}
  \mu & (\{ M_{a_n}\leq v_{a_n}\} \cap \{M_{a_t}\leq v_{a_t}\} )
  \\ &= \mu (\{M_{a_n}\leq v_{a_n}\}\cap\{M_{a_t-a_n}\circ
  f^{a_n}\leq v_{a_t}\} ) \nonumber \\ & \leq \mu ( \{M_{a_n}\leq
  v_{a_n}\}\cap\{M_{a_t-a_n-\ell_t}\circ f^{a_n+\ell_t}\leq
  v_{a_t}\} ) \nonumber \\ &\leq \mu(M_{a_n}\leq v_{a_n}) \mu
  (M_{a_t-a_n-\ell_t}\leq v_{a_t} ) + c_1 a_n e^{- \tau
    (\ell_t)}, \nonumber
\end{align}
where in the last line we have used Lemma~\ref{lem:gausssplit}.
We choose $\ell(t)=\kappa t$ for some $\kappa>0$ to be
specified in the proof below.  It suffices to consider $v_n$
such that $\mu(X_1>v_n)\approx\log\log n/n$, with $\approx$
denoting multiplication by a constant within $[1/2,2]$. (See
\cite[Lemma~4.3.2]{Galambos} on taking this reduction). 
	
Then the dynamical blocking arguments in
Lemma~\ref{lem:Ml} give
\begin{align}
  \mu (\{M_{a_n}\leq v_{a_n}\}) &= C e^{-a_n \mu (X_1>v_{a_n})} +
  O (a_{n}^{-\beta}), \label{eq:extra8}\\ \mu (
  M_{a_t-a_n-\ell_t} \leq v_{a_t} ) &= C e^{-(a_t-a_n-\ell_t)
    \mu (X_1>v_{a_t})} + O ( {\scriptstyle
    (a_t-a_n-\ell_t)^{-\beta}} ), \nonumber
\end{align}
for some constants $C, \beta > 0$. Our choice of $\ell(t)$ grows
fast enough to ensure decay of correlations gives a good
approximation to \eqref{eq.MM4}, but slow enough to ensure a good
approximation to \eqref{eq.MM5}. We state the following result.

\begin{lemma}
  Assume that \eqref{eq.MM2} holds.  Then equation \eqref{eq.MM3}
  holds.
\end{lemma}

\begin{proof}
  We
  summarise some properties of $a_n=e^{\lambda n/\log n}$.  For
  all $n\to\infty$, and moderate values of $t>0$
  \begin{equation}\label{eq.MM7}
    \frac{a(n+t)-a(n)}{a(n+t)}\log\log a(n+t)\geq C\lambda t.
  \end{equation}
  To see this apply the mean value theorem:
  \begin{align*}
    \frac{a(n+t)-a(n)}{a(n+t)} & = 1 - \exp\{-\lambda
    (t+n)/\log(t+ n)+\lambda n/\log n\} \\ & = 1-\exp \biggl\{
    \lambda \biggl(-\frac{1}{\log x}+\frac{1}{(\log x)^2} \biggr)
    t \biggr\},\quad (x\in[n,t+n]),\\ & \geq \frac{C\lambda
      t}{\log (n+t)}.
\end{align*}
Then note that $\log\log a(n+t)$ is $\approx\log(n+t)$ for large
$n$.

By assumption of \eqref{eq.MM2}, and given any $\Delta>0$ we can
choose $M'>M$ so that
\begin{equation}\label{eq.MM9}
  \Delta \leq \sum_{n=M}^{M'}\mu(M_{a_n}\leq v_{a_n}) \leq 2
  \Delta.
\end{equation}
(This is valid when $\mu(M_{a_n}\leq v_{a_n})\to 0$, which is
true in our case).  Now, let us consider the right hand terms of
\eqref{eq.MM6}. We factor out $\mu(M_{a_n}\leq v_{a_n})$ as
follows,
\begin{multline*}
  \mu(M_{a_n}\leq v_{a_n})\mu(M_{a_t-a_n-\ell_t}\leq
  v_{a_t})+c_1 a_n e^{- \tau (\ell_t)}\\ =\mu(M_{a_n}\leq
  v_{a_n}) \biggl( \mu(M_{a_t-a_n-\ell_t}\leq v_{a_t})
   +\frac{ c_1 a_n e^{- \tau (\ell_t)}}{\mu(M_{a_n}\leq
    v_{a_n})} \biggr),
\end{multline*}
and, hence, to show \eqref{eq.MM3} it is sufficient to show the final
bracketed term can be bounded by $\delta<1$, when $\ell(t)=\kappa t$,
and after summing over $t\in[n+1, M']$. Consider the exponential decay
of correlation term (with rate $\tau_1=e^{-\tau}<1$) within the
bracket. This is bounded as follows.
\begin{equation*}
  c_1 a_n e^{- \tau (\ell_t)}\cdot\mu(M_{a_n}\leq v_{a_n})^{-1}
  \leq C a_n\tau^{\ell_t}_1 \cdot \bigl( e^{-a_n\mu(X_1>v_{a_n})}
  + O(a_{n}^{-\beta}) \bigr)^{-1}.
\end{equation*}
Using $\mu(X_1>v_n)\approx\log\log n/n$, $a_n=e^{\lambda n/\log
  n}$ and $\ell(t)=\kappa t$ gives a bound
\[
C\exp \Bigl\{\frac{\lambda n}{\log n} \Bigr\} \cdot \tau^{\kappa
	t}_1 \cdot \Bigl(\frac{1}{(\log
	a_n)^D} + O(a_{n}^{-\beta} )\Bigr)^{-1}
\]
with $D\in [1/2,2]$. Now choose $\kappa$ so that $e^\lambda
<\tau_1^{-\kappa}$. This term decays exponentially fast in
$t\in[n+1,M']$, and the first term (for $t=n+1$) is $o(1)$ for
large $n$. Thus the sum of this contribution is bounded by
$\delta_1 < 1$, for large $n$. It now suffices to consider
\begin{multline*}
  \mu(M_{a_t-a_n-\ell(t)} \leq v_{a_t}) \\ =
  C\exp\{-(a_t-a_n-\ell(t)) \mu(X_1>v_{a_t})\} + O(
    \scriptstyle{(a_t - a_n - \ell(t))^{-\beta}}).
\end{multline*}
The $O(\cdot)$ term is again summable, and decays exponentially
fast with $(a_t-a_n-\ell(t))^{-\beta}=o(1)$ for $t=n+1$.  This is
therefore bounded by $\delta_2 < 1$, for large $n$.  We also have
\begin{align*}
  \exp\{-(a_t &- a_n-\ell(t))\mu(X_1>v_{a_t})\} \\ & =
  \exp\{-(a_t-a_n)\mu(X_1>v_{a_t})\} \cdot
  \exp\{\ell(t) \mu(X_1>v_{a_t})\}\\ & = \exp \{-(a_t-a_n)
  \mu(X_1>v_{a_t})\} \cdot \exp\{D\kappa t\log\log a_t/a_t\}\\ &
  = \exp\{-(a_t-a_n)\mu(X_1>v_{a_t})\}\cdot
  e^{o(1)},\quad(n\to\infty).
\end{align*}
The latter $e^{o(1)}$ comes from the precise form of $a_t$.
Hence it suffices to show that
\[
\sum_{t=n+1}^{M'}\exp\{-(a_t-a_n)\mu(X_1>v_{a_t})\}<\delta_3,
\]
for $\delta_3$ sufficiently small. However this is now the same
argument as used in \cite{Galambos}, as it depends only on $a_t$,
and the assumption on the asymptotics of $\mu(X_1>v_n)$. The
formalities depend on splitting $t\in[n+1,M']$ into three time
windows, and the bounds utilise equations
\eqref{eq.MM7} and \eqref{eq.MM9}.
\end{proof}

From this lemma, we can deduce first the weaker conclusion,
namely that if $\mu(X_1>v_n)\leq c\log\log n/n$ for $c<1$, then
$\mu(M_n\leq v_n \text{ i.o.})=1$. This follows from the fact
that by equation (\ref{eq:extra8})
\[
\mu(M_{a_n}\leq v_{a_n})=\frac{C}{(\log
	a_n)^c}+O(a_{n}^{-\beta})\approx \frac{C}{n^c},
\]
(which is not summable and, hence, \eqref{eq.MM2} is satisfied),
and that \eqref{eq.MM3} holds for this sequence.

To complete the proof of Theorem~\ref{the:lowerbound1}, it is
enough to show that the choice of $a_n$ is enough to conclude
that the 2nd half of the Robbins--Siegmund condition implies
$\mu(M_n\leq v_n \text{ i.o.})=1$. Following \cite{Galambos}, we
show that for a sequence $v_n$ satisfying
\[
\sum_{n = 1}^\infty \mu (X_1>v_n) = \infty \qquad
\text{and} \qquad \sum_{n = 1}^\infty \mu (X_1>v_n)
e^{- n \mu (X_1>v_n)} = \infty,
\]
then $\sum_{n=1}^{\infty}\mu(M_{a_n}\leq v_{a_n})=\infty$. 

Since $\sum_n (a_n)^{-\beta}<\infty$, we have to show by
(\ref{eq:extra8}) that
$\sum_{n=1}^{\infty}\exp\{-a_n\mu(X_1>v_{a_n})\}=\infty$. By
monotonicity considerations, we have
\begin{align*}
  \infty &=\sum_{n=1}^{\infty} \sum_{j=a_n}^{a_{n+1}} \mu
  (X_1>v_j) e^{- j \mu (X_1>v_j)}\\ &\leq \sum_{n=1}^{\infty}
  \mu(X_1>v_{a_{n}})(a_{n+1}-a_n) \exp \{ - a_n \mu
  (X_1>v_{a_{n+1}})\}.
\end{align*}
Hence the implication follows as in \cite{Galambos}.

\section{Proof of Theorem~\ref{the:lowerbound3}}\label{sec.pf.thm.lb3}
We split this section up into two parts, and treat cases (1) and (2)
of Theorem~\ref{the:lowerbound3} separately.

\subsection{Proof of Theorem~\ref{the:lowerbound3} Case (1)}\label{sec.pf.thm.lb3case1}
Following the methods of Section~\ref{sec.proof.preliminary} leading
to Corollary~\ref{cor.lem:Ml}, we obtain
\[
  \mu ( M_l < u_n ) \leq e^C \exp (-l \mu ( \phi \geq u_n ) )
  +n\Xi_{p,n} + qt\mu( \phi \geq u_n ) +\frac{c_2}{l^{1+s}},
\]
where $p=[l^s]$, and $t=[\log(l^{\frac{2}{\tau}})]$. Take $l=n$ and
for $\rho>0$ suppose that $\mu (\phi \geq u_n ) \leq (\log n)^\rho
/n$. Then by (A2) we have that
\begin{equation}\label{eq.mn-errorsplit}
  \mu ( M_n < u_n ) \leq c_6 \exp (-n \mu ( \phi \geq u_n ) )
  + O (n^{-\gamma'})
\end{equation}
for some $\gamma' \in (0,1)$. We remark here that it is indeed
sufficient to restrict to $\mu (\phi \geq u_n ) \leq (\log n)^\rho
/n$ rather than the more general case $\mu (\phi \geq u_n )\approx
n^{-\sigma}$ for some $\sigma\in(0,1)$. For the latter case the error
term $n^{-\gamma'}$ in equation \eqref{eq.mn-errorsplit} would
dominate.

Let $a > 1$ and take $n_k = [a^k]$. Then $n_k^{-\gamma}$ is
summable over $k$. We state the following result, which we prove
at the end of this section.

\begin{proposition} \label{prop:Cauchycondensation}
  Suppose that $n \mapsto n \mu ( \phi \geq u_n )$ is
  non-decreasing and positive. Then for $a > 1$ and $\theta > 0$ we
  have
  \begin{align*}
    \sum_{k=1}^\infty e^{- \theta [a^k] \mu ( \phi \geq
      u_{[a^{k+1}]} ) } < \infty \qquad & \Rightarrow \qquad
    \sum_{n=1}^\infty \mu ( \phi \geq u_n ) e^{- n \theta \mu
      ( \phi \geq u_n)} < \infty \\ & \Rightarrow \qquad
    \sum_{k=1}^\infty e^{- \theta [a^k] \mu ( \phi \geq
      u_{[a^{k}]} ) } < \infty.
  \end{align*}
\end{proposition}

Using Proposition~\ref{prop:Cauchycondensation} and
\eqref{eq.mn-errorsplit} we have almost surely that there exists
a $k_0$ such that $M_{n_k} \geq u_{n_k}$ for all $k \geq k_0$.

Suppose that such a $k_0$ exists. Let $n > n_{k_0}$, and take $k$
such that $n_k \leq n < n_{k+1}$. Then
\[
M_n \geq M_{n_k} \geq u_{n_k}.
\]
Since $u_n$ is an increasing sequence, we obtain that
\[
M_n \geq u_{[n/a]}
\]
holds for all $n > n_{k_0}$. This proves the first statement of
the theorem.

Finally, suppose that $\mu (\phi \geq u_n) \geq c \frac{\log
  \log n}{n}$ for some $c > 1$. Put $\tilde{u}_n =
u_{[\tilde{a}n]}$. Since $c > 1$, we can choose $\tilde{a} > 1$
close to one so that  $\mu (\phi \geq \tilde{u}_n ) \geq
\tilde{c} \frac{\log \log n}{n}$ for large $n$, where $\tilde{c}
> 1$.

Then
\[
\sum_{k=1}^\infty e^{-[a^k] \mu ( \phi \geq \tilde{u}_{[a^{k+1}]}
  )} \leq \sum_{k=1}^\infty e^{- \tilde{c} \log \log [a^{k+1}]} <
\infty
\]
holds for $a > 1$. It follows that almost surely $M_n \geq
\tilde{u}_{[n/a]}$ holds eventually. Take $1 < a <
\tilde{a}$. Then $\tilde{u}_{[n/a]} = u_{[\tilde{a}[n/a]]} > u_n$
holds when $n$ is large, and the result follows.

\subsubsection*{Proof of
  Proposition~\ref{prop:Cauchycondensation}}

We first prove a variant of Cauchy condensation. Suppose that $a
> 1$ and that $c_k$ is a sequence of positive numbers with
$c_{k+1} \leq c_k$ for all $k$. Let $C = c_1 + c_2 + \ldots +
c_{[a] - 1}$.  Then
\begin{align*}
  \sum_{n=1}^\infty c_n &= C + \sum_{k=1}^\infty \sum_{j =
    [a^k]}^{[a^{k+1}] - 1} c_j \leq C + \sum_{k=1}^\infty \sum_{j
    = [a^k]}^{[a^{k+1}] - 1} c_{[a^k]} \\ &= C +
  \sum_{k=1}^\infty ([a^{k+1}] - [a^k]) c_{[a^k]} \leq C +
  \sum_{k=1}^\infty (a^{k+1} - a^k + 1) c_{[a^k]} \\ &\leq C +
  \sum_{k=1}^\infty a^{k+1} c_{[a^k]}.
  \end{align*}
Hence,
\[
\sum_{k=1}^\infty a^k c_{[a^k]} < \infty \qquad \Rightarrow
\qquad \sum_{n=1}^\infty c_n < \infty.
\]

Similarly, we have
\begin{align*}
  \sum_{n=1}^\infty c_n &= C + \sum_{k=1}^\infty \sum_{j =
    [a^k]}^{[a^{k+1}] - 1} c_j \geq C + \sum_{k=1}^\infty
  \sum_{j = [a^k]}^{[a^{k+1}] - 1} c_{[a^{k+1}]} \\ &= C +
  \sum_{k=1}^\infty ([a^{k+1}] - [a^k]) c_{[a^{k+1}]} \geq C
  + \sum_{k=1}^\infty (a^{k+1} - 1 - a^k) c_{[a^{k+1}]} \\ &= C +
  \sum_{k=1}^\infty (a^{k+1} (1 - a^{-1}) - 1) c_{[a^{k+1}]}.
\end{align*}
Hence, we have proved
\[
\sum_{k=1}^\infty a^k c_{[a^k]} < \infty \qquad \Leftrightarrow
\qquad \sum_{n=1}^\infty c_n < \infty.
\]
  
Now, since $n \mapsto n \mu ( \phi \geq u_n )$ is non-decreasing,
$n \mapsto \exp (- n \theta \mu ( \phi \geq u_n ))$ is
non-increasing. Hence, since $[a^k] \theta \mu ( X \geq u_{[a^k]} )
\geq c$ for some $c > 0$, we have
\begin{align*}
  \sum_{k=1}^\infty e^{-[a^k] \theta \mu ( X \geq u_{[a^k]} )}
  = \infty \ & \Rightarrow \ \sum_{k=1}^\infty [a^k] \theta \mu
  ( X \geq u_{[a^k]} ) e^{- [a^k] \theta \mu ( X \geq
    u_{[a^k]} )} = \infty \\ & \Leftrightarrow
  \ \sum_{k=1}^\infty \theta \mu ( X \geq u_n ) e^{- n \theta
    \mu ( X \geq u_n )} = \infty.
\end{align*}

Finally, provided that there exists a constant $K > 0$ such that
\begin{equation} \label{eq:condensationcondition}
  [a^k] \mu ( X \geq u_{[a^k]} ) e^{- [a^k] \theta \mu ( X
    \geq u_{[a^k]} )} \leq K e^{-[a^{k-1}] \theta \mu ( X \geq
    u_{[a^k]} )}
\end{equation}
for all large $k$, we have
\begin{align*}
  \sum_{k=1}^\infty \mu ( X \geq u_n ) &e^{- n \theta \mu ( X
    \geq u_n )} = \infty \\ & \Leftrightarrow \quad
  \sum_{k=1}^\infty [a^k] \mu ( X \geq u_{[a^k]} ) e^{- [a^k]
    \theta \mu ( X \geq u_{[a^k]} )} = \infty \\ & \Rightarrow
  \quad \sum_{k=1}^\infty e^{-[a^{k-1}] \theta \mu ( X \geq
    u_{[a^k]} )} = \infty \\ & \Leftrightarrow \quad
  \sum_{k=1}^\infty e^{-[a^{k}] \theta \mu ( X \geq
    u_{[a^{k+1}]} )} = \infty.
\end{align*}
The condition \eqref{eq:condensationcondition} is implied by the
condition
\[
  [a^k] \mu ( X \geq u_{[a^k]} ) \leq K e^{(1 -
    [a^{k-1}]/[a^k]) [a^k] \theta \mu ( X \geq u_{[a^k]} )}.
\]
Since $(1 - [a^{k-1}]/[a^k])$ is bounded away from $0$ for large
$k$, it is possible to find a $K$ such that
\eqref{eq:condensationcondition} holds for all large $k$. \qed

\subsection{Proof of Theorem~\ref{the:lowerbound3} Case (2)}\label{sec.pf.thm.lb3case2}

The arguments
mirror those used for proving case (2) of
Theorem~\ref{the:lowerbound1}, but with fine adjustments used for the
sequences.  We start with an immediate consequence of assumptions (A1)
and (A3), where the decay rate is given by $\Theta(j)=\exp(-\tau
j)$. The following lemma builds upon Lemma~\ref{lem:split}
\begin{lemma} \label{lem:split2}
  There is a constant $c_1$ such that for $m>n$ and $\ell \leq m-n$,
  \begin{multline*}
    \mu (\{ M_n \leq u_n\} \cap \{ M_m \leq u_m\} ) \\ \leq \mu ( M_n
    \leq u_n ) \mu ( M_{m-n-\ell}\leq u_{m} ) + c_1 K^n e^{-\tau
      \cdot (n+\ell)}.
  \end{multline*}
\end{lemma}

\begin{proof}
  Let $\varphi_1 (x) = \mathbbm{1}_{\{M_{n}\leq u_n\}}(x)$ and
  $\varphi_2(x) = \mathbbm{1}_{\{ M_{m-n-\ell}\leq u_{m} \}}(x)$. We
  estimate the $\mathrm{BV}$-norm of $\varphi_1$. Since for any
  interval $A$, $f^{-1}(A)$ has at most $K_{f}$ connected components
  (by (A3)), it follows that the $\mathrm{BV}$-norm of $\varphi_1$ is
  bounded by $K^n$, for some constant $K$.

  Using decay of correlations, we get that
  \begin{align*}
    \mu (\{ M_n \leq u_n\} & \cap \{ M_m \leq u_m\}) \\ &=
    \mu (\{ M_n \leq u_n\} \cap \{ M_{m-n} \circ f^n \leq
    u_m\} ) \\ & \leq \mu (\{ M_n \leq u_n\} \cap \{
    M_{m-n-\ell} \circ f^{n+\ell} \leq u_m\} ) \\ &= \int
    \varphi_1 \cdot \varphi_2 \circ f^{n+\ell} \, \mathrm{d} \mu
    \\ &\leq \int \varphi_1 \, \mathrm{d} \mu \int \varphi_2 \,
    \mathrm{d} \mu + C e^{-\tau \cdot (n+\ell)} \lVert \varphi_1
    \rVert_{\mathrm{BV}} \lVert \varphi_2 \rVert_\infty \\ & \leq \mu
    ( M_n \leq u_n ) \mu ( M_{m-n-\ell}\leq u_{m} ) + c_1 K^n
    e^{-\tau \cdot (n+\ell)}. \qedhere
  \end{align*}
\end{proof}

We modify the proof of part (2) of Theorem~3.2, still following
\cite[Section~4]{Galambos}. This time we define the sequence $(a_n)$
recursively:
\[
a_{n+1}= a(n+1) = \bigl(1+(\log\log(a_{n}))^3 \bigr)\cdot a_n, \quad
\quad a_0=\exp(\lambda)
\]
for a given $\lambda>1$. 

As before we want to show that 
\begin{equation}\label{eq.MM22}
  \sum_{n=1}^{\infty}\mu(M_{a_n}\leq v_{a_n})=\infty,
\end{equation}
and
\begin{equation}\label{eq.MM32}
  \sum_{t=n+1}^{M'}\mu \bigl( \{M_{a_n}\leq
  v_{a_n}\}\cap\{M_{a_t}\leq v_{a_t}\} \bigr) \leq \delta
  \mu(M_{a_n}\leq v_{a_n})
\end{equation}
for all $M'>n$. Once again, we consider a further sequence
$\ell(t)$, with $\ell(t)<a(t)-a(n)$, and defined for $t>n$. Since
$a_{n+1}-a_n=(\log\log(a_{n}))^3\cdot a_n$, we can choose
$\ell(t)=(\log\log(a_{t-1}))\cdot a_{t-1}$ for all $t>n$. As
before, $\ell(t)$ is used to de-correlate successive maxima. By
Lemma~\ref{lem:split2} we have:
\begin{multline}\label{eq.MM62}
  \mu (\{ M_{a_n}\leq v_{a_n}\} \cap \{M_{a_t}\leq v_{a_t}\} )
  \\ \leq \mu(M_{a_n}\leq v_{a_n}) \mu
  (M_{a_t-a_n-\ell_t}\leq v_{a_t} ) + c_1 K^{a_n} e^{- \tau
    \cdot (a_n+\ell_t)}.
\end{multline}

By \cite[Lemma~4.3.2]{Galambos} it suffices to consider $v_n$
such that $\mu(X_1>v_n)\approx\log\log n/n$, with $\approx$
denoting multiplication by a constant within $[1/2,2]$.  In
addition, the dynamical blocking arguments in Lemma~7.2 give
\begin{align}
  \mu (M_{a_n}\leq v_{a_n}) &= C e^{-a_n \mu (X_1>v_{a_n})} +
  O (a_{n}^{-\beta}), \label{eq:extra82}\\ \mu (
  M_{a_t-a_n-\ell_t} \leq v_{a_t} ) &= C e^{-(a_t-a_n-\ell_t)
    \mu (X_1>v_{a_t})} + O ( {\scriptstyle
    (a_t-a_n-\ell_t)^{-\beta}} ), \nonumber
\end{align}
for some constants $C, \beta > 0$. 

\begin{lemma}
  Assume that \eqref{eq.MM22} holds.  Then equation
  \eqref{eq.MM32} holds.
\end{lemma}

\begin{proof}
  By assumption of \eqref{eq.MM22}, and given any $\Delta>0$ we
  can choose $M'>M$ so that
  \begin{equation*} 
    \Delta \leq \sum_{n=M}^{M'}\mu(M_{a_n}\leq v_{a_n}) \leq
    2\Delta.
  \end{equation*}
  (This is valid when $\mu(M_{a_n}\leq v_{a_n})\to 0$, which is
  true in our case by \eqref{eq:extra82} and
  $\mu(X_1>v_n)\approx\log\log n/n$).  Now, let us consider the
  right hand terms of \eqref{eq.MM62}. We factor out
  $\mu(M_{a_n}\leq v_{a_n})$ as follows,
  \begin{multline*}
    \mu(M_{a_n}\leq v_{a_n})\mu(M_{a_t-a_n-\ell_t}\leq
    v_{a_t})+c_1 K^{a_n} e^{- \tau \cdot
      (a_n+\ell_t)}\\ =\mu(M_{a_n}\leq v_{a_n}) \biggl(
    \mu(M_{a_t-a_n-\ell_t}\leq v_{a_t}) +\frac{ c_1 K^{a_n} e^{-
        \tau \cdot (a_n+\ell_t)}}{\mu(M_{a_n}\leq v_{a_n})} \biggr),
  \end{multline*}
  and, hence, to show \eqref{eq.MM32} it is sufficient to show the
  final bracketed term can be bounded by $\delta<1$, when
  $\ell(t)=(\log\log(a_{t-1}))\cdot a_{t-1}$, and after summing over
  $t\in[n+1,M']$. Consider the exponential decay of correlation term
  within the bracket. With equation \eqref{eq:extra82} this is bounded
  as follows.
  \begin{equation*}
    \frac{ c_1 K^{a_n} e^{- \tau \cdot (a_n+\ell_t)}}{\mu(M_{a_n}\leq
      v_{a_n})} \leq C K^{a_n} e^{- \tau \cdot (a_n+\ell_t)} \cdot
    \bigl( e^{-a_n\mu(X_1>v_{a_n})} + O(a_{n}^{-\beta}) \bigr)^{-1},
  \end{equation*}
  Using $\mu(X_1>v_n)\approx\log\log n/n$ and our choices for $a_n$ as
  well as $\ell(t)$ gives a bound
  \[
  C e^{\log(K)a_n-\tau \cdot (a_n+\ell_t)} \cdot \Bigl(\frac{1}{(\log
    a_n)^D} + O(a_{n}^{-\beta} )\Bigr)^{-1}
  \]
  with $D\in [1/2,2]$. This term decays exponentially fast in
  $t\in[n+1,M']$, and the first term (for $t=n+1$) is $o(1)$ for large
  $n$. Thus the sum of this contribution is bounded by $\delta_1 < 1$,
  for large $n$. It now suffices to consider
  \begin{multline*}
    \mu(M_{a_t-a_n-\ell(t)} \leq v_{a_t}) \\ =
    C\exp\{-(a_t-a_n-\ell(t)) \mu(X_1>v_{a_t})\} + O(
    \scriptstyle{(a_t - a_n - \ell(t))^{-\beta}}).
  \end{multline*}
  The $O(\cdot)$ term is again summable, and decays exponentially fast
  with $(a_t-a_n-\ell(t))^{-\beta}=o(1)$ for $t=n+1$.  This is
  therefore bounded by $\delta_2 < 1$, for large $n$.  We also have
  \begin{align*}
    \exp\{-(a_t &- a_n-\ell(t))\mu(X_1>v_{a_t})\} \\ & =
    \exp\{-(a_t-a_n)\mu(X_1>v_{a_t})\} \cdot \exp\{\ell(t)
    \mu(X_1>v_{a_t})\}\\ & = \exp \{-(a_t-a_n) \mu(X_1>v_{a_t})\}
    \cdot \exp\{D\ell(t)\log\log a_t/a_t\}\\ & =
    \exp\{-(a_t-a_n)\mu(X_1>v_{a_t})\}\cdot
    e^{o(1)},\quad(n\to\infty).
  \end{align*}
  The latter $e^{o(1)}$ comes from the precise form of $a_t$ and 
  \begin{multline*}
    D\ell(t) \frac{\log\log a_t}{a_t} \\ =
    D\log\log(a_{t-1})\frac{\log\log\bigl(1+(\log\log
      a_{t-1})^3\bigr)+\log\log(a_{t-1})}{\bigl(1+(\log\log
      a_{t-1})^3\bigr)} \to 0.
  \end{multline*}
  Hence it suffices to show that
  \[
  \sum_{t=n+1}^{M'}\exp\{-(a_t-a_n)\mu(X_1>v_{a_t})\}<\delta_3,
  \]
  for $\delta_3$ sufficiently small. This follows from
  $\mu(X_1>v_n)\approx\log\log n/n$ and our growth of $a_t$.
\end{proof}

Now we complete the proof of case (2) in Theorem~\ref{the:lowerbound3}. It is
enough to show that the choice of $a_n$ is sufficient to conclude
that for a sequence $v_n$ satisfying $n \mapsto n\mu(X_1>v_n)$ is non-decreasing,
\[
\sum_{n = 1}^\infty \mu (X_1>v_n) = \infty, \qquad
\text{and} \qquad \sum_{n = 1}^\infty \mu (X_1>v_n)
e^{- n \gamma \mu (X_1>v_n)} = \infty
\]
for some $\gamma>1$, then $\sum_{n=1}^{\infty}\mu(M_{a_n}\leq v_{a_n})=\infty$. 

Since $\sum_n (a_n)^{-\beta}<\infty$, we have to show by
(\ref{eq:extra82}) that
\[
\sum_{n=1}^{\infty}\exp\{-a_n\mu(X_1>v_{a_n})\}=\infty.
\]
Since $n \mapsto n\mu(X_1>v_n)$ is non-decreasing, we have by
monotonicity considerations that
\begin{align*}
  \infty &=\sum_{n=1}^{\infty} \sum_{j=a_n}^{a_{n+1}} \mu (X_1>v_j)
  e^{- j \gamma \mu (X_1>v_j)}\\ &\leq \sum_{n=1}^{\infty}
  \mu(X_1>v_{a_{n}})(a_{n+1}-a_n) \exp \{ - a_n \gamma \mu
  (X_1>v_{a_{n}})\} \\ & \leq \sum_{n=1}^{\infty} C
  \frac{\log\log(a_n)}{a_n}(a_{n+1}-a_n) \cdot e^{-(\gamma-1)
    D\log\log(a_n)} \cdot e^{ - a_n \mu (X_1>v_{a_{n}})} \\ & =
  \sum_{n=1}^{\infty} C (\log\log(a_n))^4 \cdot \frac{1}{(\log
    a_n)^{(\gamma-1)D}} \cdot \exp \{ - a_n \mu (X_1>v_{a_{n}})\} \\ &
  \leq C_1 \cdot \sum_{n=1}^{\infty} \exp \{ - a_n
  \mu(X_1>v_{a_{n}})\},
\end{align*}
where we used $\gamma>1$ in the last step. Hence, we conclude that
\[
\sum_{n=1}^{\infty}\exp\{-a_n\mu(X_1>v_{a_n})\}=\infty
\]
as required. \qed

\section{Proof of Theorems~\ref{thm.clustering1} and \ref{thm.hyp1}}\label{sec.pf.thm.clustering1}

To prove the theorems stated in Section~\ref{sec.clustering}, we
need a version of equation \eqref{eq.mn-errorsplit} incorporating
the extremal index $\theta$. Before proving each theorem in turn,
we collect relevant results from \cite{FFT3} which adapt the
blocking methods of Section~\ref{sec.overview-block} to the case
$\theta\in(0,1)$.

We use the notations of Section~\ref{sec.clustering}, and for
$s,\ell \geq 0$, and an event $B\subset\mathcal{X}$ we write
\[
\mathcal{W}_{s,\ell}(B) := \cap_{i=s}^{s+\ell-1} f^{-i}
(B^\complement).
\]
(Notice that $\mathcal{W}_{0,n}(U(u)) = \{M_n\leq u\}$.) In the
following, we recall also that $q$ denotes the period of the
hyperbolic periodic point $\tilde{x}$.

Using \cite[Corollary~2.4]{FFT3} combined with
\cite[Proposition~5.1]{F} leads to the following approximation
results. Consider sequences $t_n,k_n \to \infty$, and in the following we take $\lesssim$ to mean '$\leq$' up to a uniform positive
multiplying constant. Then
\begin{multline} \label{eq.escaperate}
  \biggl| \mu \bigl( \mathcal{W}_{0,n}(A^{(q)}_n) \bigr) -
  \biggl(1 - \frac{n}{k_n}\mu(A^{(q)}_n) \biggr)^{k_n} \biggr|
  \\ \lesssim k_n t_n \mu(U_n) + n \gamma_1 (q,n,t_n) + 
	n\gamma_2 (q,n,k_n), 
\end{multline}
where
\begin{align*}
  \gamma_1 (q,n,t_n) & = |\mu(A^{(q)}_n\cap
  \mathcal{W}_{t_n,\ell}(A^{(q)}_n)) - \mu (A^{(q)}_n) \mu
  (\mathcal{W}_{0,\ell} (A^{(q)}_n))|,\\ \gamma_2 (q,n,k_n) & =
  \sum_{j=q+1}^{n/k_n} \mu(A^{(q)}_n \cap f^{-j} (A^{(q)}_n)).
\end{align*}
In addition, we have
\begin{equation}\label{eq.escaperate2}
  \bigl| \mu(M_n\leq u_n) - \mu \bigl(\mathcal{W}_{0,n}(A^{(q)}_n)
  \bigr) \bigr| \leq q \mu (U_n \setminus A^{(q)}_n).
\end{equation}
Putting these results together leads to the following lemma.

\begin{lemma}
  Under the assumptions of equations \eqref{eq.escaperate},
  \eqref{eq.escaperate2} the following formula is valid:
  \begin{multline} \label{eq.escaperate3}
    | \mu(M_n\leq u_n) - \exp (-n\theta \mu ( \phi \geq u_n ))| \\ 
    \lesssim k_n t_n \mu(U_n) + n \gamma_1 (q,n,t_n) +
     n\gamma_2(q,n,k_n)\\ +q \mu (U_n \setminus
     A^{(q)}_n)+|\theta_n-\theta|+\frac{n^2\mu(A^{(q)}_n)^2}{k_n}.
  \end{multline}
\end{lemma}

\begin{proof}
  The proof requires justification of the inclusion of the last two
  terms.  From equation \eqref{eq.escaperate}, and using
  $e^{x}=1+x+O(x^2)$ we have:
  \begin{align*}
    \biggl( 1 - \frac{n}{k_n} \mu(A^{(q)}_n) \biggr)^{k_n}
    &=e^{-n\mu(A^{(q)}_n)} + O \biggl(
    \frac{n^2\mu(A^{(q)}_n)^2}{k_n}
    \biggr)\\ &=e^{-n\theta\mu(X_1>u_n)} + O \biggl(
    \frac{n^2\mu(A^{(q)}_n)^2}{k_n} \biggr) +
    O(|\theta-\theta_n|). \qedhere
  \end{align*}
\end{proof}

\subsection{Completing the proofs of Theorem~\ref{thm.clustering1} and Theorem~\ref{thm.hyp1}}

In the case of proving
Theorem~\ref{thm.clustering1} we follow the proof of
Theorem~\ref{the:lowerbound1}, while in the case of proving
Theorem~\ref{thm.hyp1} we follow the proof of
Theorem~\ref{the:lowerbound3}. In the first instance we use
equation~\eqref{eq.escaperate3} to prove the following result
analogous to the conclusion of Corollary~\ref{cor.lem:Ml}.

\begin{proposition}\label{prop.ei-conclusion}
  Under the assumptions of Theorem~\ref{thm.clustering1} we have
  \begin{equation*} 
    \mu ( M_n < u_n ) = \exp (-n\theta \mu ( \phi \geq u_n )
    ) + O (n^{-\gamma'})
  \end{equation*}
  for some $\gamma' \in (0,1)$.
\end{proposition}
\begin{proof}
  To prove the proposition we estimate each term on the right hand
  side of equation~\eqref{eq.escaperate3}. The term which requires
  careful analysis is the one involving $\gamma_2(q,n,k_n)$. First of
  all from (A4), we have $|\theta-\theta_n|=O(n^{-\hat\sigma})$. We
  take $t_n=c\log n$ for some $c\gg-1/\log\tau$ and apply exponential
  decay of correlations via (A1) to get a polynomial decay in $n$ for
  $n\gamma_1 (q,n,t_n)$ within equation~\eqref{eq.escaperate3}. For
  the sequence $u_n$, we can restrict to the case
  $\mu(A^{(q)}_n)=O((\log n)^{\rho}/n)$ for some $\rho>0$. The reasons
  are similar to the choice of $u_n$ made in
  equation~\eqref{eq.mn-errorsplit}. Since we assume that $\theta\neq 0$
  (again from (A4)) the term $q\mu(U_n\setminus A^{(q)}_n)$ also
  decays to zero at the same rate $O((\log n)^{\rho}/n)$. Similarly
  the term $\frac{n^2\mu(A^{(q)}_n)^2}{k_n}=O(n^{-\gamma_1})$ for some
  $\gamma_1>0$. Thus we are left to estimate the remaining term
  $\gamma_2(q,n,k_n)$.

  By inspecting equation \eqref{eq.escaperate3} it suffices to
  show existence of $B>0$ such that
  \begin{equation}\label{eq.ei.blocksplit1}
    n\sum_{j=q+1}^{n/k_n} \mu(A^{(q)}_n \cap f^{-j}
    (A^{(q)}_n))=O(n^{-B}).
  \end{equation}
  The constant $k_n$ plays the same role as the blocking number $p$
  used in the proof of Theorem~\ref{the:lowerbound1}.  We take
  $k_n=\sqrt{n}$, but other rates can be chosen. We split
  \eqref{eq.ei.blocksplit1} into the following two sums:
  \begin{multline}\label{eq.ei.blocksplit2}
    n\sum_{j=q+1}^{n/k_n} \mu(A^{(q)}_n \cap f^{-j}
    (A^{(q)}_n))=\\ n\sum_{j=q+1}^{g(n)} \mu(A^{(q)}_n \cap
    f^{-j} (A^{(q)}_n))+n\sum_{j=g(n)}^{n/k_n} \mu(A^{(q)}_n \cap
    f^{-j} (A^{(q)}_n)),
  \end{multline}
  and take $g(n)=\kappa\log n$ for some
  $\kappa>0$ to be determined. The following results will be useful.
  
  \begin{lemma}
    Suppose that $\tilde{x}$ is a hyperbolic periodic point, and
    $|(f^q)'(\tilde{x})|\in(1,\infty)$. Then there is a time $R_n\geq
    c_0\log n$ with $f^j(A^{(q)}_n)\cap A^{(q)}_n=\emptyset$ for all
    $j\leq R_n$.
  \end{lemma}
  This is an elementary calculation based on estimating the time taken
  for orbits to escape from a fixed neighbourhood of the (hyperbolic)
  periodic orbit, see \cite{FFT2}. The constant $c_0$ depends on
  $(f^q)'(\tilde{x})$, and on the size of the neighbourhood around
  $\tilde{x}$ for which $f^q$ is a diffeomorphism.

  Hence if we choose $g(n)=\kappa\log n$, with $\kappa<c_0$, then the
  first term on the right hand side of \eqref{eq.ei.blocksplit2} is
  zero.  To deal with the second term for this choice of $g(n)$, we
  use decay of correlations (A1) and Proposition~\ref{prop.decay-bvlp}
  (see Section~\ref{sec.appendix2}). Taking $k_n=\sqrt{n}$, we obtain
  that this is bounded by
  \begin{align*} 
    \sum_{j=g(n)}^{\sqrt{n}} \mu(A^{(q)}_n \cap f^{-j}
    (A^{(q)}_n)) &\leq \sqrt{n}\mu(U_n)^2 + \|1_{U_n}\|_{L^{p'}}
    \|1_{U_n}\|_{\mathrm{BV}}\sum_{j=g(n)}^{\sqrt{n}}
    \Theta(j),\\ &\leq c_1 \sqrt{n} \biggl(\frac{(\log
      n)^{2\rho}}{n^2}\biggr) + c_2\frac{(\log
      n)^{\rho/p'}}{n^{1/p'}}\cdot \frac{1}{n^{\kappa_1}},
  \end{align*}
  with $\kappa_1=\kappa\tau.$ Hence, by choosing $p'$
  sufficiently close to 1, equation~\eqref{eq.ei.blocksplit1} holds
  for some $B>0$. Therefore the conclusion of
  Proposition~\ref{prop.ei-conclusion} holds.
\end{proof}

To complete the proof of Theorem~\ref{thm.clustering1} we now follow
the proof of Theorem~\ref{the:lowerbound1} step by step, as detailed
in Sections~\ref{sec.pf.main-case1} and
\ref{sec.pf.main-case2}. Similarly, in the case of proving
Theorem~\ref{thm.hyp1} we follow the proof of
Theorem~\ref{the:lowerbound3}. The inclusion of the parameter $\theta$
in the distribution for $\mu(M_n\leq u_n)$ causes no further technical
obstacles in applying these methods of proof.

\subsection{Proof of
  Proposition~\ref{prop.lsv-ei}}\label{sec.pf.prop.lsv-ei}

To prove Proposition~\ref{prop.lsv-ei}, we begin with a local analysis
of the dynamics near the neutral fixed point $\tilde{x}=0$ to estimate
$\mu(A^{(q)}_n)$, and then use the identities given in equations
\eqref{eq.escaperate} and \eqref{eq.escaperate2}.

An estimate for $\mu(A^{(q)}_n)$ is given in \cite{FFTV}, and we
repeat the main steps here for completeness. Take $q=1$, and consider
the ball $B(0,r_n)$, with $v_n=\psi(r_n)$. The set $A^{(1)}_n$ is
precisely the set $[\hat{r}_n,r_n]$, with $f(\hat{r}_n)=r_n$. Using
the fact that the density $\rho(x)$ takes the form
\[
\rho(x)=\frac{H(x)}{x^{a}},\quad\text{with}\quad H\in L^{1+\epsilon},
\]
we obtain
\[
\mu(A^{(1)}_n)\sim C_{1}(r^{1-a}_n-\hat{r}^{1-a}_n).
\]
Using the fact that $r_n=\hat{r}_n+2^{a}\hat{r}^{1+a}_n$, an
asymptotic analysis yields
\[
\mu(A^{(1)}_n)\sim C_{2}r_n\sim C_{3}\mu(X_1>v_n)^{\frac{1}{1-a}}.
\]
The constants $C_i$ are generic constants that depend on $\rho$
through $H$.  Hence using equations \eqref{eq.escaperate} and
\eqref{eq.escaperate2}, we obtain
\[
\mu(M_n\leq
v_n)=\exp\{-C'n\mu(X_1>v_n)^{\frac{1}{1-a}}\}+O(n^{-\sigma'}),
\]
with $C'>0$ and $\sigma'>0$.  Using a Cauchy-condensation argument as
in the proof of Theorem~\ref{the:lowerbound3}, we go along the
sequence $n_k=b^k$ for $b>1$. This leads to $\mu(M_{n_k}\leq
v_{n_k}\,\textrm{i.o.})=0$ in the case where
\[
C'\mu(X_1>v_n)^{\frac{1}{1-a}}>\frac{c\log\log n}{n},
\]
for any $c>1$. The conclusion of the proof of
Proposition~\ref{prop.lsv-ei} follows by taking $c>(C')^{-1}$.

\section{Proof of Theorem~\ref{thm.hyp2}}\label{sec.pf.thm.hyp2}

The idea is to use the approach of
Theorem~\ref{the:lowerbound3}. The main step is to by-pass the
influence of the set $\mathcal{M}_r$. As in the proof of
Theorem~\ref{the:lowerbound3} we consider the sequence $(r_n)$ such that $\mu
(X_1>\psi(r_n) ) = (c\theta^{-1}\log\log n)/n$ for some
$c>1$. By the local dimension estimate at $\tilde{x}$, we have
for all $\epsilon>0$, and $r<r_0(\tilde{x})$, 
\[
n^{-1/d_{\mu}-\epsilon}<r_n<n^{-1/d_{\mu}+\epsilon}
\]
holds for all large enough $n$.  We now go along the subsequence
$(r_{n_k})$ with $n_k=a^k$, and any $a>1$.  This leads to
$\mu(\mathcal{M}_{r_{n_k}})<Ck^{-\sigma_1}$ which is summable, and
hence $\mu(\limsup_{k}\mathcal{M}_{r_{n_k}})=0$. Thus for
$\mu$-a.e.\ $\tilde{x}\in\mathcal{X}$ equation
\eqref{eq.mn-error2} applies along the subsequence $r_{n_k}$ for
all $k\geq k_0(\tilde{x})$. The error term is $O(k^{-\sigma_2})$,
which is again summable. To complete the proof, we follow the
same approach of proving Theorem~\ref{the:lowerbound3}, except
here the error term is $O((\log n)^{-\sigma_2})$ rather than
$n^{\gamma}$. However, along the sequence $n_k=a^k$, the error
term remains summable. Again, this leads to
\[
M_n(x) \geq
v_{n/a},\;\textrm{with}\;\mu(X_1>v_{n/a})=c\theta^{-1}a\log(\log
(n/a))/n,
\]
where both $c$ and $a$ can be made arbitrarily close to 1. Let
$\tilde{v}_n=v_{n/a}$. Then for any $c'>1$ we have
$\mu(X_1>\tilde{v}_n)>c'\theta^{-1}\log\log n/n$, and
$\mu(M_n\geq \tilde{v}_n \text{ ev.})=1$, as required.

For the cases
$\sigma_1<1$ or $\sigma_2<1$, then we must go along a faster
growing subsequence $n_k=e^{k^{\gamma}}$ with $\gamma>1$. This is to ensure that
the First Borel--Cantelli Lemma can be applied in the proof
above. In particular, we must choose $\gamma>1$ so that $\gamma\sigma_1>1$ and
$\gamma\sigma_2>1$. However in the window $n\in[n_k,n_{k+1}]$,
the value $n$ is not uniformly comparable to $n_k$. In particular
we have
\[
n_{k+1}/n_{k}=e^{(k+1)^{\gamma}-k^{\gamma}}\leq
e^{ck^{\gamma-1}},
\]
where $c$ depends on $\gamma$. This gives
\[
n_{k} \leq n_{k+1}\exp\{-c (\log
n_{k})^{\frac{\gamma-1}{\gamma}}\} \leq n\exp\{-c(\log
n)^{\frac{\gamma-1}{\gamma}}\},
\]
and leads to the bound $\mu ( X>v_n )
\geq e^{(\log n)^{\gamma'}}n^{-1}$, valid for
$\gamma'>(\gamma-1)/\gamma$.

\section{On Condition (A2) and its verification for selected dynamical systems}\label{sec.A2check}

Our main arguments used to prove condition (A2) go back to Collet
\cite[Corollary~2.4 and Lemma~2.5]{Collet}, where similar estimates
are proved for some non-uniformly hyperbolic maps of an interval,
including quadratic maps for Benedicks--Carleson parameters. These
arguments have also been carried out for other types of systems by
Gupta, Holland and Nicol \cite[Section~4]{GHN}, for instance for Lozi
and Lorenz maps.

The argument starts by first estimating the measure of the set
\[
\{\, x : d(x,f^j x) < r_n \text{ for some } j \leq g(n) \,\},
\]
for a suitable function $g(n)$, such as $g(n)=(\log n)^{\gamma}$ for
$\gamma> 1$.  The choice of $g(n)$ is chosen to grow fast enough to
combat decay of correlations, i.e.\ so that $\Theta(g(n))\to 0$
sufficiently fast. One then obtains localised estimates using the
Hardy--Littlewood maximal inequality. The result is that for many
systems, including quadratic maps for Benedicks--Carleson parameters
and Lorenz maps, condition (A2) holds for $\mu$ a.e.\ point
$\tilde{x}$ when $\mu$ is a measure which is absolutely continuous
with respect to Lebesgue measure. Relative to the aforementioned
literature, a technical aspect in our case is that we need to assume a
wider class of sequences $r_n$ to check (A2), in particular allowing
for $\mu(B(\tilde{x},r_n))\approx n^{-\sigma}$ for $\sigma<1$. In the
usual extreme value theory literature, the sequences $r_n$ are chosen
so that $n\mu(B(\tilde{x},r_n))\to \ell\in(0,\infty)$, such as in equation \eqref{eq.un-seq} (see
\cite{Letal}).

To verify (A2), we give the argument in detail for the following systems:
piecewise differentiable maps satisfying assumptions on decay of
correlations, and piecewise expanding maps with an invariant measure
which is absolutely continuous with respect to Lebesgues measure. We
also explain how (A2) is obtained for quadratic maps with
Benedicks--Carleson parameters, contrasting to the methods given by Collet
\cite{Collet}.

\subsection{Condition (A2) for piecewise differentiable maps}\label{sec.piece-diff}

We consider an interval map $f \colon \XX\to\XX$ preserving an ergodic
measure $\mu$ which is piecewise differentiable.  That is, we assume
that the derivative of $f$ is uniformly bounded, so that there is a
constant $\Lambda_+\in\mathbb{R}$ with $|f'(x)|<\Lambda_+$ for all
$x\in\XX$. We allow for $f$ to have a finite number of
discontinuities, and we let $\mathcal{S}$ denote the finite set of
discontinuity points.  We also assume the following regularity
condition on the measure $\mu$:  there exist 
$c_1$, $c_2>0$ and $s_1>s_2>0$ such that for $\mu$-a.e.\ $x\in\XX$,
there exists $r_0>0$ such that for all $r<r_0$,
\begin{equation}\label{eq:measure-reg}
  c_1r^{s_1}\leq \mu(B(x,r))\leq c_2 r^{s_2}. 
\end{equation}
Furthermore, we assume that the upper bound holds for all $x\in\XX$ so
that $r_0$ is independent of $x$ relative to the constants
$c_2,s_2$. Examples include beta-transformations, $x\mapsto \beta
x\mod 1$ ($\beta>1$), and the quadratic map for Benedicks--Carleson
parameters. For these maps it is known that the density of $\mu$ is in
$L^p$ for $p<2$,~\cite{Young_quadratic}. Hence the upper bound of
\eqref{eq:measure-reg} holds for some $s_2,c_2,r_0>0$ and all
$x\in\XX$, $r<r_0$.  We have the following proposition.

\begin{proposition}\label{prop.piece-diff}
  Suppose that $f \colon \XX\to\XX$ is a piecewise differentiable
  interval map, preserving an ergodic measure $\mu$. Suppose that (A1)
  holds, and $\mu$ satisfies equation~\eqref{eq:measure-reg}.  Then
  for $\mu$-a.e.\ $\tilde{x}\in\XX$ condition (A2) holds. That is, for
  $\mu$-a.e.\ $\tilde{x}\in\XX$, there exists $\gamma,s$ and
  $\sigma>0$ such that equations \eqref{eq.short-constants} and
  \eqref{eq.short1} hold for all sequences $r_n$ with
  $\mu(B(\tilde{x},r_n))=O(n^{-\sigma})$.
\end{proposition}

\begin{remark}
  The main conclusion of Proposition~\ref{prop.piece-diff} is that
  (A2) applies to $\mu$-a.e.\ $\tilde{x}\in\XX$. It is possible to
  check (A2) point-wise under knowledge of recurrence properties of
  $\tilde{x}$, such as knowing that $\tilde{x}$ is pre-periodic to a
  hyperbolic fixed point. In these cases it is possible to remove some
  of the global assumptions, such as requiring existence of
  $\Lambda_+<\infty$, or requiring uniformity of the constants
  in~\eqref{eq:measure-reg} to all $x\in\XX$.
\end{remark}

\begin{remark}
  The proof we give is much shorter relative to the (general) methods
  presented in~\cite{Collet}. The main trick is that for piecewise
  differentiable systems it is sufficient to control the recurrence of
  typical points over a time window of order $\log n$. Previous
  methods have taken a longer time window of the order $(\log
  n)^\gamma$ for some $\gamma>1$.
\end{remark}

\begin{proof}
  To prove this result, consider for $p_n=n^s$ and $s>0$ the quantity
  $\Xi_{p_n,n}(r_n)$. Introducing an intermediate growing function
  $g(n)=o(p_n)$ we split up $\Xi_{p_n,n}(r_n)$ into two sums as
  follows:
  \begin{multline}\label{eq.xi-split1}
    \Xi_{p_n,n}(r_n)=
    \sum_{j=1}^{g(n)}\mu(f^{j}B(\tilde{x},r_n)\cap
    B(\tilde{x},r_n))
    \\ +\sum_{j=g(n)+1}^{n^s}\mu(f^{j}B(\tilde{x},r_n)\cap
    B(\tilde{x},r_n)).
  \end{multline}
  For $\sigma>(0,1)$ we assume that $\mu(B(\tilde{x},r_n)) =
  O(n^{-\sigma})$. For $g(n)=\kappa\log n$, the
  first sum on the right-hand side of \eqref{eq.xi-split1} is zero by the following claim.
  
  \begin{claim}
  	There exists $\kappa>0$, such that for all
  	$j\leq \kappa\log n$, we have $f^{j}B(\tilde{x},r_n)\cap
  	B(\tilde{x},r_n)=\emptyset$.
  \end{claim}

  \begin{proof}[Proof of Claim]
  We consider the set of closely returning points $E_{r,n}$
  defined by
  \[
  E_{r,n} = \{\, x : d(f^n (x), x) < r \,\}.
  \] 
  Using Lemma~\ref{Lemma:A:BV-Linf} (see Appendix), condition (A1),
  and the regularity condition~\eqref{eq:measure-reg} we deduce that
  \begin{align*}
    \mu (E_{r,n}) &\leq \int \mu (B(x,r)) \, \mathrm{d} \mu (x) +
    D e^{- \eta n} \\ & \leq c_2 r^{s_2} + D e^{- \eta n},
  \end{align*}
  for some $\eta>0$. Let $r=2^{-j}$, then by the regularity
  condition~\eqref{eq:measure-reg}, and the First Borel--Cantelli
  Lemma we have $\mu(\liminf E^{\complement}_{r_j,j})=1$. Hence for
  $\mu$-almost all $x$, there exists $j_0(x)$ such that
  $\mathrm{dist}(f^j(x),x)>2^{-j}$ for all $j\geq j_0$. Take
  $\tilde{x}$ to be a representative in this full measure set.

We impose a further restriction on the orbit of $\tilde{x}$ as
follows. Let
\[
F_j=\{\, x:\mathrm{dist}(f^j(x),\mathcal{S})<2^{-j} \,\},
\]
where $\mathcal{S}$ denotes the discontinuities of $f$. Then we
take $\tilde{x}\in\liminf F^{\complement}_j$. Again, this set also has
$\mu$-measure 1.

By \eqref{eq:measure-reg} and the assumption $\mu(B(\tilde{x},r_n)) =
O(n^{-\sigma})$, it follows that $r_n = O(n^{-\frac{\sigma}{s_1}})$.
Take $\tilde{x}$ to be in the set of $\mu$-measure 1 as described
above.  We consider a time $R\equiv R(\tilde{x}, r_n)$ such that
\begin{itemize}
\item[(i)] $f^j$ is continuous on $B(\tilde{x},r_n)$ for all $j\leq
  R$.
\item[(ii)] $f^jB(\tilde{x},r_n)\cap B(\tilde{x},r_n)=\emptyset$ for all
  $j\leq R$.
\end{itemize}
We provide a lower bound for $R$ such that the above two assumptions
are satisfied.  First, there is a time $j_0(\tilde{x})$ for which
simultaneously $d(f^j(\tilde{x}),\tilde{x})>2^{-j}$, and
$d(f^j(\tilde{x}),\mathcal{S})>2^{-j}$ hold for all $j\geq j_0$. The set of
such $\tilde{x}$ has $\mu$-measure one. We let $m_0(\tilde{x})$ denote
\[
m_0 (\tilde{x}) = \inf_{y\in
  B(\tilde{x},r_n)}\{\,d(f^j(y),\tilde{x}),\,d(f^j(y),\mathcal{S}) :
j\leq j_0 \,\}.
\]
For all $n$ sufficiently large, we have $m_0(\tilde{x})>r_n$ (perhaps
removing a further countable set of $\tilde{x}$ that meet
$\mathcal{S}$ before time $j_0(\tilde{x})$).

To bound $R$, we claim that there exists $\kappa_1>0$ with
$R>\kappa_1\log n$. Let $y\in B(\tilde{x},r_n)$, then for $j \geq j_0$
\begin{equation}
  \begin{split}
    d(f^j(y),\mathcal{S}) &\geq
    d(f^j(\tilde{x}),\mathcal{S})-d(f^j(y),f^j(\tilde{x}))\\ &\geq
    2^{-j}-2r_n \Lambda_{+}^j,
  \end{split}
\end{equation}
where $\Lambda_{+}$ is the upper bound for $|f'|$. Now for (i) to
hold, we require
\[
2^{-j}-2r_n\Lambda_{+}^j>0,
\]
for all $j\leq R$, otherwise the ball meets the singularity set prior
to time $R$. Hence, if we choose
\[
R < \frac{-\log(r_n)-\log 2}{\log 2 + \log \Lambda_+},
\]
then (i) is satisfied.

To verify item (ii), a similar bound is obtained. Similarly to above,
we have for $y\in B(\tilde{x},r_n)$ that for $j \geq j_0$
\begin{align*}
    d(f^j(y),\tilde{x}) &\geq
    d(f^j(\tilde{x}),\tilde{x})-d(f^j(y),f^j(\tilde{x}))\\ &\geq
    2^{-j}-2r_n\Lambda_{+}^j.
\end{align*}
We require that $2^{-j}-2r_n|\Lambda_{+}|^j > r_n$ for all $j \leq
R$. Hence, if we choose
\[
R < \frac{- \log r_n - \log 3}{\log 2 + \log \Lambda_+},
\]
then (ii) is satisfied.

Take any $\kappa< \frac{\sigma}{2s_1 \log(2\Lambda_+)}$ and let
$g(n)=\kappa\log n$. From the above two requirements together with
$r_n = O(n^{-\frac{\sigma}{s_1}})$, it follows that we may take $R =
\kappa \log n$ for large $n$.

Note that $\kappa$ depends on $\sigma$, but without loss we can
further restrict to $\sigma>1/2$ so that $\sigma$ is bounded away from
zero. It is immediate from the construction above that
$f^jB(\tilde{x},r_n)\cap B(\tilde{x},r_n)=\emptyset$ for all $j\leq
\kappa \log n$. 
\end{proof}

Hence in the estimate for $\Xi_{p_n,n}$, the first
sum on the right of \eqref{eq.xi-split1} is zero. For the second sum,
we use exponential decay of correlations for $\mathrm{BV}$ against
$L^{\infty}(\mu)$ in conjunction with
Proposition~\ref{prop.decay-bvlp}.  This gives
\begin{multline*}
  \sum_{j=g(n)+1}^{n^s}\mu(f^{j}B(\tilde{x},r_n)\cap
  B(\tilde{x},r_n)) \\ \leq
  n^s \mu(B(\tilde{x},r_n))^2+C_1\mu(B(\tilde{x},r_n))^{1/p'}e^{-\tau
    g(n)},
\end{multline*}
with $C_1>0$. By choice of $r_n$ the right is bounded by:
\begin{equation}\label{eq.bound-kappa}
n^{s-2\sigma}+ n^{-\sigma/p'}\cdot n^{-\kappa_2},
\end{equation}
where $\kappa_2$ depends on $\kappa$ and $\tau$.  Hence, there exists
a choice of constants $s,\sigma$ consistent with
\eqref{eq.short-constants} so that~\eqref{eq.bound-kappa} is bounded
by $n^{-1-\gamma}$ for some $\gamma>0$. This completes the proof.
\end{proof}

\subsection{Condition (A2) for piecewise expanding maps} \label{sec:piece-exp}

In this section we consider piecewise expanding maps. Relative to
Section~\ref{sec.piece-diff} we allow for unbounded derivative. This
allows us to cover the Gau\ss{} map. The set up is as follows.  Suppose
that $f \colon [0,1] \to [0,1]$ is a piecewise expanding map, with
finitely many pieces of continuity. There is then a partition
$\mathscr{P} = \{I_1, \ldots, I_m\}$ such that $f$ is differentiable
on each $I_k$. Let $\mathscr{P}_n$ be the corresponding partition for
$f^n$. Since the partition $\mathscr{P}$ is finite, there is a
$\delta_0 > 0$ such that every partition element of $\mathscr{P}$ has
a diameter of at least $\delta_0$. We let $S$ be the set of endpoints
of partition elements of $\mathscr{P}$. The set $S$ is $\delta_0$
separated.

Alternatively, we assume that the partition $\mathscr{P}$ is
countable, in which case we assume that there is a $\delta_0 > 0$
such that for all $n$ holds $|f^n (I)| \geq \delta_0$ whenever $I
\in \mathscr{P}_n$.

We assume that $f$ is uniformly expanding, i.e.\ that there is a
constant $\Lambda$ such that $|f'| \geq \Lambda$.  Moreover, we
assume that $f$ has bounded distortion, and that $\mu$ is an
ergodic measure $\mu$ with exponential decay of correlations for
functions of bounded variation against $L^1$. This means that
there exists a constant $C$ such that
\[
x,y \in I \in \mathscr{P}_n \qquad \Rightarrow \qquad C^{-1} \leq
\frac{D f^n (x)}{D f^n (y)} \leq C
\]
and 
\[
\biggl| \int \phi_1 \cdot \phi_2 \circ f^j \, \mathrm{d} \mu -
\int \phi_1 \, \mathrm{d} \mu \int \phi_2 \, \mathrm{d} \mu
\biggr| \leq C e^{-\tau j} \lVert \phi_1 \rVert_\mathrm{BV}
\lVert \phi_2 \rVert_1
\]
for some $\tau > 0$.

We will prove that for any such piecewise expanding map, the set
of points $\tilde{x}$ which satisfies assumption (A2) has full
measure.  Example of a systems satisfying our assumption are
piecewise expanding maps with finitely many pieces and an
absolutely continuous invariant measure $\mu$; the Gau\ss{} map
with the Gau\ss{} measure; or the first return map to
$[\frac{1}{2}, 1)$ for a Manneville--Pomeau map with an
absolutely continuous invariant measure $\mu$.

\begin{proposition} \label{prop:A2forpwexp}
  Suppose that there is a constant $c$ such that $\mu (I) \leq c |I|$
  for any interval $I$. Then for $\mu$-a.e.\ $\tilde{x} \in\XX$, there
  exists $\gamma,s$ and $\sigma>0$ such that equations
  \eqref{eq.short-constants} and \eqref{eq.short1} hold for all
  sequences $(r_n)$ with $\mu(B(\tilde{x},r_n))=O(n^{-\sigma})$, and
  satisfying the additional assumption: for any $t > 0$,
  \begin{equation} \label{eq:r-assumption}
    \limsup_{k \to \infty} \frac{r_{k^t}}{r_{(k+1)^t}} < \infty.
  \end{equation}
\end{proposition}
\begin{remark}
Requirement of assumption \eqref{eq:r-assumption} is a consequence of
the method of proof. Unlike in the proof of
Proposition~\ref{prop.piece-diff}, we cannot infer that
$f^jB(\tilde{x},r_n)\cap B(\tilde{x},r_n)=\emptyset$ for all $j=O(\log
n)$. However it is possible to check (A2) point-wise if certain
recurrence properties of $\tilde{x}$ are known (such as pre-periodic),
and in turn relax assumption \eqref{eq:r-assumption}.
\end{remark}
To prove Proposition~\ref{prop:A2forpwexp}, we will need two
lemmata. For the first lemma, we let $A_n (\delta) = \{\, I \in
\mathscr{P}_n : |f^n (I)| < \delta \,\}$.

\begin{lemma} \label{lem:smallimage}
  If $\mu$ satisfies $\mu (I) \leq c |I|$ for any interval $I$,
  then there exists a constant $K_0$ such that
  \[
  \mu (\cup A_n (\delta)) \leq K_0 \delta
  \]
  holds for any $\delta < \delta_0$.
\end{lemma}

\begin{proof}
  We only have to consider the case when $f$ is piecewise
  expanding with finitely many pieces, since in the case with
  countably many pieces, out assumptions imply that $\mu (\cup
  A_n (\delta)) = 0$ whenever $\delta < \delta_0$.
  
  Let $I_n (x)$ denote the partition element of $\mathscr{P}_n$
  which $x$ belongs to. If $|I_n(x)| < \delta < \delta_0$ then
  there are $j,k < n$ such that $j \neq k$ and both $f^j (x)$ and
  $f^k (x)$ are close to $S$. More precisely, we must have
  \[
  d (f^j (x), S) < \delta \Lambda^{-n+j} \qquad \text{and} \qquad
  d (f^k (x), S) < \delta \Lambda^{-n+k},
  \]
  since otherwise, $S$ would not have ``cut'' the partition
  element $I_n (x)$ in a way such that $|I_n (x)| < \delta$. We
  therefore have $\cup A_n (\delta) \subset B_n (\delta)$, where
  $B_n (\delta)$ is defined by
  \[
  B_n (\delta) = \bigcup_{0 \leq j < k < n} \bigl( f^{-j} S_{(
    \delta \Lambda^{-n+j})} \cap f^{-k} S_{(\delta
    \Lambda^{-n+k})} \bigr),
  \]
  and $S_{(\varepsilon)} = \{\, t \in [0,1] : d(t,S) <
  \varepsilon \,\}$. We shall estimate the measure of $B_n
  (\delta)$.

  By decay of correlations, we have for $j < k$ that
  \begin{align*}
    \mu &\bigl( f^{-j} S_{( \delta \Lambda^{-n+j})} \cap f^{-k}
    S_{(\delta \Lambda^{-n+k})} \bigr) = \mu \bigl( S_{( \delta
      \Lambda^{-n+j})} \cap f^{-(k-j)} S_{(\delta \Lambda^{-n+k})}
    \bigr) \\ &\leq \mu ( S_{( \delta \Lambda^{-n+j})} ) \mu (
    S_{(\delta \Lambda^{-n+k})} ) + C 2 (m+2) \mu ( S_{( \delta
      \Lambda^{-n+k})} ) e^{-\tau (k-j)} \\ & \leq c^2 \delta^{2}
    \Lambda^{- 2n + j + k} + c C 2 (m+2) \delta \Lambda^{-n+k}
    e^{-\tau (k-j)}.
  \end{align*}
  We obtain that
  \[
  \mu (B_n (\delta)) \leq \sum_{0 \leq j < k < n} \Bigl( c^2
  \delta^{2} \Lambda^{- 2n + j + k} + c C 2 (m+2) \delta
  \Lambda^{-n+k} e^{-\tau (k-j)} \Bigr) \leq K_0 \delta,
  \]
  for some constant $K_0$.
\end{proof}

We now consider the set
\[
E_{j,r} = \{\, x : d(x,f^jx) < 2 r \,\}.
\]
In the arguments that follow, we need to control the measure of this
set in terms of $r$ when $j$ is small. Thus we cannot use directly
Lemma~\ref{Lemma:A:BV-Linf}.
\begin{lemma} \label{lem:Ekmeasure}
  If $\mu$ satisfies $\mu (I) \leq c |I|$ for any interval $I$,
  then there exists a constant $K_1$ such that
  \[
  \mu (E_{j,r}) \leq K_0 \delta + K_1 r \delta^{-1}
  \]
  holds for any $4 r < \delta < \delta_0$. In particular, there is
  a constant $K_2$ such that
  \[
  \mu (E_{j,r}) \leq K_2 \sqrt{r},
  \]
  when $4 r < \delta_0^2$.
\end{lemma}

\begin{proof}
  Suppose that $4 r < \delta < \delta_0$. By bounded distortion,
  we have for any $I \in \mathscr{P}_j$ with $|f^j (I)| \geq
  \delta$ that $I \cap E_{j,r}$ is an interval of length at most
  $4 C r \delta^{-1} |I|$. Together with
  Lemma~\ref{lem:smallimage}, we get that
  \[
  \mu (E_{j,r}) \leq \mu (\cup A_j (\delta)) + K_1 r \delta^{-1}
  \leq K_0 \delta + K_1 r \delta^{-1}.
  \]

  When $4 r < \delta_0^2$, we may choose $\delta = \sqrt{r}$ to
  obtain
  \[
  \mu (E_{j,r}) \leq (K_0 + K_1) \sqrt{r}. \qedhere
  \]
\end{proof}

We are now in position to prove
Proposition~\ref{prop:A2forpwexp}.

\begin{proof}[Proof of Proposition~\ref{prop:A2forpwexp}]
  Let $\tilde\gamma > 1$, and put
  \[
  E_k = \bigcup_{j = 1}^{2 (\log k)^{\tilde\gamma}} E_{j, r_k}.
  \]
  Then
  \begin{equation} \label{eq:Ekmeasure}
    \mu (E_k) \leq K_2 (\log k)^{\tilde\gamma} \sqrt{r_k},
  \end{equation}
  for some constant $K_2$, by Lemma~\ref{lem:Ekmeasure}.

  Put
  \[
  g (x) = \sup_{r > 0} \frac{1}{2r} \int_{B(x,r)}
  \mathbbm{1}_{E_k} \, \mathrm{d} \mu.
  \]
  By the Hardy--Littlewood maximal inequality, the set
  \[
  F_k (c) = \{\, x : g(x) > c \,\}
  \]
  has Lebesgue measure at most $\frac{3}{c} \mu (E_k)$. Hence
  \[
  \mu (F_k (c)) \leq \frac{3C}{c} K_2 (\log k)^{\tilde\gamma} \sqrt{r_k}.
  \]

  Note that
  \[
  x \in F_k (c)^{\complement} \qquad \Rightarrow \qquad \mu (E_k \cap B
  (x,r_k)) \leq c \mu (B(x,r_k)).
  \]
  For constants $\alpha,\beta>0$, let $n_k = k^\beta$ and $c = n_k^{-\alpha}$. We obtain
  \[
  \mu (F_{n_k} (n_k^{-\alpha})) \leq 3C K_2 k^{\alpha \beta}
  r_{n_k}^\frac{1}{2} (\beta \log k)^{\tilde\gamma}.
  \]
  Assuming that $r_k = O(k^{-\sigma})$ for some $\sigma > 0$, we
  have
  \[
  \mu (F_{n_k} (n_k^{-\alpha})) \leq 3C K_2 k^{\frac{\beta}{2} (2
    \alpha - \sigma)} (\beta \log k)^{\tilde\gamma}.
  \]
  Take $0 < 2 \alpha < \sigma$, $\beta$ large enough that
  \[
  \sum_k \mu (F_{n_k} (n_k^{-\alpha})) < \infty.
  \]
  Hence $\mu (\limsup_{k\to \infty} F_{n_k} (n_k^{-\alpha})) = 0$
  and we have for a.e.\ $\tilde{x}$ that
  \[
  \mu (E_{n_k} \cap B (\tilde{x}, r_{n_k})) \leq n_k^{-\alpha}
  \mu (B (\tilde{x}, r_{n_k}))
  \]
  holds for all large $k$ (depending on $\tilde{x}$). Let such an
  $\tilde{x}$ be fixed. For piecewise expanding maps we can assume a stronger
	form of equation \eqref{eq:measure-reg}, namely we assume that $\tilde{x}$ is
  such that there exists a constant $c_0 > 0$ such that
  \begin{equation} \label{eq:localdensity}
    c_0^{-1} r < \mu (B(\tilde{x}, r)) < c_0 r
  \end{equation}
  holds for all $0 < r < 1$, since this is a property which holds
  for a.e.\ $\tilde{x}$.

  Consider any large $n > 0$ and take $k$ such that $n_k \leq n <
  n_{k+1}$. We then have
  \[
  \tilde{E}_n = \bigcup_{j=1}^{(\log n)^{\tilde\gamma}} E_{j,r_n} \subset
  E_{n_k}
  \]
  when $k$ is large. Since $r_n$ is a decreasing sequence, we
  also have $B (\tilde{x}, r_n) \subseteq B (\tilde{x},
  r_{n_k})$. Hence
  \begin{align*}
    \mu (B (\tilde{x}, r_n) \cap \tilde{E}_n) &\leq \mu (B
    (\tilde{x}, r_{n_k}) \cap E_{n_k}) \\ &\leq n_k^{-\alpha} \mu
    (B (\tilde{x}, r_{n_k})) \\ & \leq 2 n^{-\alpha} \mu (B
    (\tilde{x}, r_{n_k})),
  \end{align*}
  if $n$ and $k$ are large. By \eqref{eq:r-assumption} and
  \eqref{eq:localdensity}, it follows that there exists a
  constant $K$ such that
  \[
  \mu (B (\tilde{x}, r_n) \cap \tilde{E}_n) \leq K n^{-\alpha}
  \mu (B (\tilde{x},r_n))
  \]
  holds for all $n$.

  Suppose that $x \in B (\tilde{x}, r_n) \cap f^{-j} B(\tilde{x},
  r_n)$ for some $j \leq (\log n)^{\tilde\gamma}$. Then $d(x,f^j x) < 2
  r_n$ and hence $x \in B (\tilde{x}, r_n) \cap \tilde{E}_n$. We
  therefore have $B (\tilde{x}, r_n) \cap f^{-j} B(\tilde{x},
  r_n) \subset B (\tilde{x}, r_n) \cap \tilde{E}_n$ and 
  \[
  \mu (B (\tilde{x}, r_n) \cap f^{-j} B(\tilde{x}, r_n) ) \leq
  \mu (B (\tilde{x}, r_n) \cap \tilde{E}_n ) \leq K n^{-\alpha}
  \mu (B (\tilde{x},r_n)). 
  \]
	
To complete the proof, it suffices to estimate $\Xi_{p_n,n}$.  We can
split as in equation~\eqref{eq.xi-split1}, but this time take
$g(n)=(\log n)^{\tilde\gamma}$.  The arguments above show that the
first right-hand term of~\eqref{eq.xi-split1} is $O(n^{-1-\gamma})$
for a choice $\sigma$ consistent with
equation~\eqref{eq.short-constants}.  Similarly using condition (A1),
the second right-hand term of~\eqref{eq.xi-split1} is also
$O(n^{-1-\gamma})$, again for a choice of constants consistent
with~\eqref{eq.short-constants}.
\end{proof}

\subsection{Further remarks on Condition (A2) for quadratic maps} \label{sec:A2forquadratic}
We consider $f = f_a \colon [0,1] \to [0,1]$ defined by $f_a(x) =
a x (1-x)$. For some parameters, including the parameters
described by Benedicks and Carleson, there is an $f_a$-invariant
probability measure $\mu_a$ which is equivalent with respect to
Lebesgue measure.  When $a$ is a Benedicks--Carelson parameter,
$(f_a,\mu_a)$ has exponential decay of correlations for functions
of bounded variation against $L^1$ as proved by Young
\cite{Young_quadratic}.  As remarked upon in
Section~\ref{sec.piece-diff}, Proposition~\ref{prop.piece-diff}
applies to this family of maps. It is also possible to apply the
methods used in the proof of
Proposition~\ref{prop:A2forpwexp}. However this requires imposing
the regularity condition \eqref{eq:r-assumption} on the sequence
$r_n$. Indeed, under dynamical assumptions that capture the
quadratic map, Collet proved \cite[Corollary~2.4]{Collet} that
there exists a constant $\beta' \in (0,1)$ such that the set
\[
\tilde{E}_k = \{\, x : d(x,f_a^j (x)) < k^{-1} \text{ for some }
j \leq (\log k)^5 \,\}
\]
satisfies
\[
\mu (\tilde{E}_k) \leq C k^{- \beta'}.
\]
Let
\[
E_k := \{\, x : d(x,f_a^j (x)) < r_k \text{ for some }
j \leq (\log k)^4 \,\}.
\]
Since $\mu(B(\tilde{x},r_n))=O(n^{-\sigma})$, we can use the regularity conditions \eqref{eq:measure-reg}, \eqref{eq:r-assumption} for $\mu$ and
the sequence $r_n$ respectively to deduce that  $E_k \subset \tilde{E}_{k^b}$ for some $b>0$. Hence
\[
\mu (E_k) \leq C k^{- \beta_0},
\]
for some $\beta_0>0$. Replacing \eqref{eq:Ekmeasure} by the above estimate in the proof
of Proposition~\ref{prop:A2forpwexp} allows us to deduce that condition (A2) applies. 
(Within, let $\alpha\in(0,\beta_0)$ and take $\beta$ sufficiently large.)

\section{Appendix A --- The blocking argument}\label{sec.appendix1}

We follow \cite{Collet} to prove the blocking argument.

\subsection{Assumptions}

We consider a dynamical system $(\XX,f,\mu)$ where $\XX$ is an
interval and $\mu$ is a probability measure. In this section, we
will prove Proposition~\ref{prop:blocking}. To do so, we only
need to assume that $\mu$ is invariant, but when using
Proposition~\ref{prop:blocking} it shall be necessary to assume
mixing.

\subsubsection{Notation}

We have an observable $\phi \colon \XX \to \mathbbm{R}$. Let $X_k =
\phi \circ f^{k-1}$ and $M_n = \max \{X_1, \ldots, X_n \}$.

\subsection{Preparations}

\begin{lemma}[Collet {\cite[Proposition~3.2]{Collet}}]
  Let $t, r, m, k, p$ be non-negative integers. Then
  \begin{equation} \label{eq:blocking1}
    0 \leq \mu ( M_r < u ) - \mu ( M_{r+k} < u ) \leq k \mu
    ( \phi \geq u )
  \end{equation}
  and
  \begin{multline}
    \Bigl| \mu ( M_{m+p+t} < u ) - \mu (M_m < u ) +
    \sum_{j=1}^p \expectation (\mathbbm{1}_{\phi \geq u}
    \mathbbm{1}_{M_m < u} \circ f^{p+t-j}) \Bigr| \\ \leq t \mu
    (\phi \geq u) + 2 p \sum_{j=1}^p \expectation (
    \mathbbm{1}_{\phi \geq u} \mathbbm{1}_{\phi \geq u} \circ
    f^j).
    \label{eq:blocking2}
  \end{multline}
\end{lemma}

\begin{proof}
  We have $\{M_r < u\} \supset \{M_{r+k} < u \}$ and
  \[
  \{M_r < u\} \setminus \{M_{r+k} < u \} = \bigcup_{j =
    r+1}^{r+k} \{ X_j \geq u \}.
  \]
  Hence
  \[
  0 \leq \mu (M_r < u) - \mu (M_{r+k} < u ) \leq \sum_{j=
    r+1}^{r+k} \mu (X_j \geq u) = k \mu ( \phi \geq u ),
  \]
  which is \eqref{eq:blocking1}.

  We have
  \[
  \mathbbm{1}_{M_{m+p+t} < u} = \mathbbm{1}_{M_p < u}
  \mathbbm{1}_{M_t < u} \circ f^p \mathbbm{1}_{M_m < u} \circ
  f^{p+t}.
  \]
  Therefore,
  \begin{align*}
    0 \leq \mathbbm{1}_{M_p < u} &\mathbbm{1}_{M_m < u} \circ
    f^{p+t} - \mathbbm{1}_{M_{m+p+t} < u} \\ &= \mathbbm{1}_{M_p
      < u} \mathbbm{1}_{M_m < u} \circ f^{p+t} - \mathbbm{1}_{M_p
      < u} \mathbbm{1}_{M_t < u} \circ f^p \mathbbm{1}_{M_m < u}
    \circ f^{p+t} \\ &= \mathbbm{1}_{M_p < u} \mathbbm{1}_{M_m <
      u} \circ f^{p+t} (1 - \mathbbm{1}_{M_t < u} \circ f^p) \\ &
    \leq 1 - \mathbbm{1}_{M_t < u} \circ f^p = \mathbbm{1}_{M_t
      \geq u} \circ f^p.
  \end{align*}
  It then follows that
  \begin{align}
    \bigl| \expectation \mathbbm{1}_{M_{m+p+t} < u} &-
    \expectation ( \mathbbm{1}_{M_p < u} \mathbbm{1}_{M_m < u}
    \circ f^{p+t}) \bigr| \nonumber \\ &\leq \expectation
    (\mathbbm{1}_{M_t \geq u} \circ f^p) = \mu (M_t \geq u)
    \nonumber \\ &= \mu \biggl( \bigcup_{k=1}^t \{ \phi \circ f^k
    \geq u\} \biggr) \leq t \mu ( \phi \geq u
    ). \label{eq:part1}
  \end{align}

  Since
  \begin{multline*}
    \{ M_m \circ f^{p+t} < u \} \setminus \{M_m \circ f^{p+t} < u
    \text{ and } M_p < u \} \\ = \bigcup_{k=1}^p \{ X_k \geq u \}
    \cap \{M_m \circ f^{p+t} < u \},
  \end{multline*}
  we have
  \[
  \expectation \mathbbm{1}_{M_m < u} - \sum_{k=1}^p \expectation
  ( \mathbbm{1}_{X_k \geq u} \mathbbm{1}_{M_m < u} \circ f^{p+t}
  ) \leq \expectation ( \mathbbm{1}_{M_p < u} \mathbbm{1}_{M_m <
    u} \circ f^{p+t}).
  \]
  By the inclusion--exclusion inequality, we also have
  \begin{multline*}
    \expectation ( \mathbbm{1}_{M_p < u} \mathbbm{1}_{M_m < u}
    \circ f^{p+t}) \\ \leq \expectation \mathbbm{1}_{M_m < u} -
    \sum_{k=1}^p \expectation ( \mathbbm{1}_{X_k \geq u}
    \mathbbm{1}_{M_m < u} \circ f^{p+t} ) \\ + \sum_{k=1}^p
    \sum_{\substack{l=1 \\ l \neq k}}^p \expectation (
    \mathbbm{1}_{X_k \geq u} \mathbbm{1}_{X_l \geq u}
    \mathbbm{1}_{M_m < u} \circ f^{p+t} ).
  \end{multline*}
  It follows that
  \begin{align}
    \Bigl| \expectation ( \mathbbm{1}_{M_p < u} &
    \mathbbm{1}_{M_m < u} \circ f^{p+t}) - \expectation
    \mathbbm{1}_{M_m < u} + \sum_{k=1}^p \expectation (
    \mathbbm{1}_{X_k \geq u} \mathbbm{1}_{M_m < u} \circ f^{p+t}
    ) \Bigr| \nonumber \\ & \leq \sum_{k=1}^p \sum_{\substack{l=1
        \nonumber \\ l \neq k}}^p \expectation ( \mathbbm{1}_{X_k
      \geq u} \mathbbm{1}_{X_l \geq u} \mathbbm{1}_{M_m < u}
    \circ f^{p+t} ) \nonumber \\ & \leq \sum_{k=1}^p
    \sum_{\substack{l=1 \\ l \neq k}}^p \expectation (
    \mathbbm{1}_{X_k \geq u} \mathbbm{1}_{X_l \geq u} ) \leq 2 p
    \sum_{k=1}^p \expectation (\mathbbm{1}_{\phi \geq u}
    \mathbbm{1}_{X_k \geq u}). \label{eq:part2}
  \end{align}

  The estimates \eqref{eq:part1} and \eqref{eq:part2} together
  with the triangle inequality imply \eqref{eq:blocking2}.
\end{proof}

\subsection{Proof of Proposition~\ref{prop:blocking}}

We will now prove Proposition~\ref{prop:blocking}.

Let $l$ be a large number and $s \in (0,\frac{1}{2}]$. Put $p =
  [l^s]$ and write $l$ as $l = pq + r$ where $0 \leq r < p$. 

We have by \eqref{eq:blocking1} that
\[
\mu ( M_{pq} < u ) - \mu ( M_{q(p+t)} < u ) \leq qt \mu (
\phi \geq u ).
\]
If $r \leq qt$, then $l = p q + r \leq q (p + t)$ and $ \mu (
M_l < u ) - \mu ( M_{q(p+t)} < u ) \geq 0$. However, we have
$q \sim l^{1-s}$ and $t \geq 1$, so $r < p \leq q \leq qt$ holds
for all large enough $l$, since $1-s \geq s$.  Hence, when $l$ is
large, we have
\begin{align*}
  0 & \leq \mu ( M_l < u ) - \mu ( M_{q(p+t)} < u ) \\ & \leq
  \mu ( M_{pq} < u ) - \mu ( M_{q(p+t)} < u ) \leq qt \mu (
  \phi \geq u ),
\end{align*}
and
\[
| \mu ( M_l < u ) - \mu ( M_{q(p+t)} < u ) | \leq qt \mu (
\phi \geq u ).
\]

Let
\[
\Sigma_j = \sum_{k=1}^p \expectation ( \mathbbm{1}_{\phi \geq u }
\mathbbm{1}_{M_{ (j - 1) (p + t)} < u } \circ f^{p + t - k} ).
\]
By the triangle inequality we have
\begin{multline*}
  | \mu ( M_{j (p+t)} < u ) - (1 - p \mu ( \phi \geq u )) \mu
  ( M_{(j - 1) (p + t)} < u ) | \\ \leq \bigl| p \mu (\phi
  \geq u ) \mu (M_{(j-1) (p + t)} < u ) - \Sigma_j \bigr| \\ +
  \bigl| \mu (M_{j (p + t)} < u ) - \mu ( M_{(j - 1) (p + t)}
  < u ) + \Sigma_j \bigr|,
\end{multline*}
and \eqref{eq:blocking2} with $m = (j - 1) (p + t)$ implies that
\begin{multline}
  | \mu ( M_{j (p + t)} < u ) - (1 - p \mu ( \phi \geq u ))
  \mu ( M_{(j - 1) (p + t)} < u ) | \\ \leq \Gamma_j := \bigl|
  p \mu (\phi \geq u ) \mu (M_{(j - 1) (p + t)} < u ) -
  \Sigma_j \bigr| \\ + t \mu ( \phi \geq u ) + 2 p \sum_{k=1}^p
  \expectation (\mathbbm{1}_{\phi \geq u} \mathbbm{1}_{\phi \geq
    u} \circ f^k). \label{eq:iterationstep}
\end{multline}

Let $\eta = 1 - p \mu ( \phi \geq u )$. Now, using
\eqref{eq:iterationstep} iteratively, we get
\begin{align*}
  | \mu ( M_{q(p+t)} &< u ) - \eta^q | \\ & \leq | \mu (
  M_{q(p+t)} < u ) - \eta \mu ( M_{(q-1) (p+t)} < u ) | \\ &
  \phantom{=} + | \eta \mu ( M_{(q-1) (p+t)} < u ) - \eta^q |
  \\ & \leq \Gamma_q + \eta | \mu ( M_{(q-1) (p+t)} < u ) -
  \eta^{q-1} | \\ & \ldots \\ & \leq \Gamma_q + \eta \Gamma_{q-1}
  + \ldots + \eta^{q-1} \Gamma_1.
\end{align*}
This proves Proposition~\ref{prop:blocking}.

\section{Appendix B --- On correlation decay and recurrence.}\label{sec.appendix2}
In this section we collect some
useful results on decay of correlation estimates, and recurrence time
distributions. In particular, these results are used for checking
condition (A2). These results might also have broader interest.

\subsection{Decay of correlation estimates.}
In this section we explain how condition (A1) can be improved to
having (exponential) decay of correlations for
$\mathcal{B}_1=\mathrm{BV}$ versus $\mathcal{B}_2=L^p$, (with $p>1$).

The set up is a interval map $f \colon \XX \to \XX$ with an invariant
probability measure $\mu$. For $\varphi:\XX\to\mathbb{R}$, recall the
$L^p$ norms for $p \in [1,\infty]$ is defined by
\[
\lVert \varphi \rVert_p = \biggl( \int |\varphi|^p \, \mathrm{d} \mu
\biggr)^\frac{1}{p}
\]
for $p < \infty$, and $\lVert \varphi \rVert_\infty = \sup
|\varphi|$. The bounded variation norm $\lVert \varphi
\rVert_\mathrm{BV} = \lVert \varphi \rVert_\infty + \var \varphi$,
where $\var \varphi$ is the total variation of $\varphi$ on $\XX$. We
have the following result.

\begin{proposition}\label{prop.decay-bvlp}
  Suppose that correlations decay exponentially for $\mathrm{BV}$ versus
  $L^{\infty}$. For any $p > 1$, correlations decay exponentially for
  $\mathrm{BV}$ versus $L^p$.
\end{proposition}

Our proof relies on the Banach--Steinhaus theorem. Hence we
assume that the axiom of choice is valid.

\begin{proof}
  Fix $p > 1$ and suppose that $\varphi \in L^p$ and $\psi \in
  \mathrm{BV}$. (Within this section, $\psi$ will denote such a
  $\mathrm{BV}$ function: it is not to be confused with the observable
  used in previous sections.)  We note that $\varphi \in L^q$ for any
  $q \leq p$, that $\lVert \varphi \rVert_q \leq \lVert \varphi
  \rVert_p$ for such $q$, and that
  \begin{equation} \label{eq:weak}
    \mu \bigl(\{\, x : |\varphi (x)| \geq t \, \}\bigr) \leq \frac{1}{t^q} \lVert
    \varphi \rVert_q^q
  \end{equation}
  for any $t > 0$.

  Take a positive number $m$, which will be chosen more precisely
  later. We write $\varphi $ as a sum $\varphi = \varphi_1 +
  \varphi_2$, where $\varphi_1$ is defined by
  \[
  \varphi_1 = \mathbbm{1}_{\{\, x : |\varphi(x)| \leq m \,\}} \varphi.
  \]
  Then $\lVert \varphi_1 \rVert_\infty \leq m$ and $\varphi_2 \in L^p$ with
  support in the set $\{\, x : |\varphi(x)| \geq m \,\}$. We shall
  first estimate the $L^q$ norm of $\varphi_2$ for $q < p$.

  Take $q < p$ and let $r, s > 1$ be such that $\frac{1}{r} +
  \frac{1}{s} = 1$ and $sq \leq p$. We have by H\"{o}lder's
  inequality and \eqref{eq:weak} that
  \begin{align}
  	\lVert \varphi_2 \rVert_q & \leq \biggl( \int_{\{\, x :
  		|\varphi(x)| \geq m \,\}} |\varphi|^q \, \mathrm{d} \mu
  	\biggr)^\frac{1}{q} = \biggl( \int \mathbbm{1}_{\{\, x :
  		|\varphi(x)| \geq m \,\}} |\varphi|^q \, \mathrm{d}{\mu}
  	\biggr)^\frac{1}{q} \nonumber \\ &\leq (\mu (\{\, x : |\varphi(x)|
  	\geq m \,\}) )^\frac{1}{rq} \lVert \varphi \rVert_{sq} \leq
  	\frac{1}{m^\frac{p}{rq}} \lVert \varphi \rVert_p^\frac{p}{rq}
  	\lVert \varphi \rVert_{sq} \leq \frac{1}{m^\frac{p}{rq}} \lVert
  	\varphi \rVert_p^{1 + \frac{p}{rq}}.
  	\label{eq:Lq}
  \end{align}

  We now consider the correlation between $\varphi$ and $\psi$.
  Let
  \[
  C(\varphi,\psi,n) = \biggl| \int \varphi \circ f^n \psi \,
  \mathrm{d} \mu - \int \varphi \, \mathrm{d} \mu \int \psi \,
  \mathrm{d} \mu \biggr|.
  \]
  By the decomposition $\varphi = \varphi_1 + \varphi_2$ and the
  triangle inequality, we have
  \[
    C (\varphi,\psi,n) \leq C(\varphi_1, \psi,n) + C(\varphi_2,\psi,n).
  \]
  Using the decay of correlations for $\mathrm{BV}$ against
  $L^\infty$, we get
  \[
  C(\varphi_1,\psi,n) \leq C e^{-\tau n} \lVert \varphi_1
  \rVert_\infty \lVert \psi \rVert_\mathrm{BV} \leq C e^{-\tau n} m
  \lVert \psi \rVert_\mathrm{BV}.
  \]
  The correlation with $\varphi_2$ is estimated using the triangle
  inequality and \eqref{eq:Lq} with $q = 1$ and $s = p$. We get
  \begin{align*}
    C(\varphi_2, \psi, n) & \leq \biggl| \int \varphi_2 \circ f^n \psi
    \, \mathrm{d} \mu \biggr| + \biggl| \int \varphi_2 \, \mathrm{d}
    \mu \int \psi \, \mathrm{d} \mu \biggr| \\ & \leq 2 \lVert
    \varphi_2 \rVert_1 \lVert \psi \rVert_\infty \leq
    \frac{2}{m^\frac{p}{r}} \lVert \varphi \rVert_p^{1 + \frac{p}{r}}
    \lVert \psi \rVert_\mathrm{BV} = \frac{2}{m^{p-1}} \lVert \varphi
    \rVert_p^p \lVert \psi \rVert_\mathrm{BV}.
  \end{align*}

  Combining these estimates, we get
  \[
  C(\varphi,\psi,n) \leq \Bigl( C e^{-\tau n} m + \frac{2}{m^{p-1}}
  \lVert \varphi \rVert_p^p \Bigr) \lVert \psi \rVert_\mathrm{BV}.
  \]
  Choose $m = e^{\frac{\tau}{p} n}$. Then
  \[
  C(\varphi,\psi,n) \leq (C + 2 \lVert \varphi \rVert_p^p) e^{- (1 -
    \frac{1}{p}) \tau n} \lVert \psi \rVert_\mathrm{BV}.
  \]
  In particular, for any $\varphi \in L^p$ and $\psi \in \mathrm{BV}$
  there is a constant $c (\varphi,\psi)$ such that
  \[
  C(\varphi,\psi,n) \leq c(\varphi,\psi) e^{- (1 - \frac{1}{p}) \tau
    n}.
  \]

  Now, an argument by Collet \cite{Collet0}, using the
  Banach--Steinhaus theorem, implies that there is a constant $c$ such
  that for any $\varphi \in L^p$ and $\psi \in \mathrm{BV}$ holds
  \[
  C (\varphi,\psi,n) \leq c e^{- (1 - \frac{1}{p}) \tau n} \lVert
  \varphi \rVert_p \lVert \psi \rVert_\mathrm{BV}.
  \]
  Hence $(f,\mu)$ has exponential decay of correlations for $L^p$
  against $\mathrm{BV}$.
\end{proof}

\subsection{Estimates on recurrence time statistics.}
A key argument in checking condition (A2) is understanding the
distribution of recurrent points in the sense of finding the measure
of the set:
\[
  E_{r,n} = \{\, x : d(f^n (x), x) < r \,\},
\]
in terms of $n$ and $r$. We have the following result.
\begin{lemma}\label{Lemma:A:BV-Linf}
  Suppose that $([0,1], f, \mu)$ has exponential decay of
  correlations for $L^\infty$ against $BV$, that is
  \[
  \biggl| \int \phi \circ f^n \psi \, \mathrm{d} \mu - \int \phi
  \, \mathrm{d} \mu \int \psi \, \mathrm{d} \mu \biggr| \leq C
  \lVert \phi \rVert_\infty \lVert \psi \rVert_{BV} e^{- \tau n}.
  \]
  Assume that $\mu$ satisfies $\mu(B(x,r)) \leq c r^s$ for some
  constants $c,s>0$ and any ball $B(x,r)$.
  
  Then there exists a constant $D$ and a number $\eta \in
  (0,\tau)$ such that for any $r > 0$ and
  \[
  E_{r,n} = \{\, x : d(f^n (x), x) < r \,\}
  \]
  we have
  \begin{align*}
    \mu (E_{r,n}) &\leq \int \mu (B(x,r)) \, \mathrm{d} \mu (x) +
    D e^{- \eta n} \\ & \leq c r^s + D e^{- \eta n}.
  \end{align*}
\end{lemma}
\begin{remark}
  This result builds upon those stated within \cite[Section~4]{KKP2}.
\end{remark}
\begin{proof}
  Let $\{I_k\}$ be a partition of $[0,1]$ into $e^{\frac{\tau}{2}
    n}$ intervals of equal length. Let $y_k$ be the mid point of
  $I_k$. Put $\delta = \frac{1}{2} e^{-\frac{\tau}{2} n}$.

  The function
  \[
  F(x,y) = \left\{ \begin{array}{ll} 1 & \text{if } d(x,y) < r
    \\ 0 & \text{otherwise} \end{array} \right.
  \]
  is such that $\mu(E_{r,n}) = \int F(f^n (x),x)\, \mathrm{d}
  \mu(x)$. We approximate $F$ by $\tilde{F}$ defined by
  \[
  \tilde{F} (x,y) = \sum_k J_k (x) \mathbbm{1}_{I_k} (y),
  \]
  where $J_k = \mathbbm{1}_{(y_k - r - \delta, y_k + r +
    \delta)}$. Then $F \leq \tilde{F}$ holds and
  \[
  \sum_k \int J_k \, \mathrm{d} \mu \int \mathbbm{1}_{I_k} \, \mathrm{d} \mu =
  \iint \tilde{F} \, \mathrm{d} \mu \mathrm{d} \mu.
  \]

  Using decay of correlations we get
  \begin{align*}
    \mu (E_{r,n}) & = \int F(f^n (x), x) \, \mathrm{d} \mu (x)
    \\ & \leq \int \tilde{F} (f^n (x), x) \, \mathrm{d} \mu (x)
    \\ & = \sum_k \int J_k (f^n (x)) \mathbbm{1}_{I_k} (x) \, \mathrm{d} \mu
    (x) \\ &\leq \sum_k \biggl( \int J_k \, \mathrm{d} \mu \int
    \mathbbm{1}_{I_k} \, \mathrm{d} \mu + 3 C e^{- \tau n} \biggr).
  \end{align*}
  Since the sum contains $e^{\frac{\tau}{2} n}$ terms, we obtain
  \[
  \mu (E_{r,n}) \leq \iint \tilde{F} \, \mathrm{d} \mu \mathrm{d}
  \mu + 3 C e^{- \frac{\tau}{2} n}.
  \]

  Finally, if we let
  \[
    G(x,y) = \left\{ \begin{array}{ll} 1 & \text{if } d(x,y) < r+
      \delta \\ 0 & \text{otherwise} \end{array} \right.
  \]
  then $\tilde{F} \leq G$ and
  \begin{align*}
    \iint \tilde{F} \, \mathrm{d} \mu \mathrm{d} \mu &\leq \iint
    G \, \mathrm{d} \mu \mathrm{d} \mu = \int \mu (B(x,r+\delta))
    \, \mathrm{d}\mu (x) \\ &\leq \int \mu (B(x,r)) \, \mathrm{d}
    \mu(x) + 2 c \delta^s \\ & = \int \mu (B(x,r)) \, \mathrm{d}
    \mu(x) + 2^{1-s} c e^{- \frac{s \tau}{2} n}.
  \end{align*}
  This proves the lemma with $D = 3C + 2^{1-s} c$ and $\eta =
  \min(\frac{\tau}{2},\frac{s \tau}{2})$.
\end{proof}


\begin{thebibliography}{00}
	
\bibitem{Athreya} J.\ S.\ Athreya, \emph{Logarithm laws and
  shrinking target properties}, Proc. Indian
  Acad. Sci. Math. Sci. 119 (2009), no. 4, 541--557.

\bibitem{Barndorff} O. Barndorff-Neilson, \emph{On the rate of
  growth of the partial maxima of a sequence of independent
  identically distributed random variables}. Math. Scand. {\bf
  9}, (1961), 383--394.
  
\bibitem{BPS} L. Barreira, Ya. Pesin and J. Schmeling,
  \emph{Dimension and product structure of hyperbolic measures},
  Ann. of Math. 149 (1999), no. 3, 755--783.

\bibitem{BY} M. Benedicks and L.-S. Young, \emph{Markov
  extensions and decay of correlations for certain H\'enon maps},
  Asterisque No. 261 (2000), xi, 13--56.
	
\bibitem{BGT} N.~H. Bingham, C.~M. Goldie, and J.~L. Teugels,
  \emph{Regular variation}, {\bf 27}, of {\em Encyclopedia of
    Mathematics and its Applications}. Cambridge University
  Press, Cambridge, 1989.
  
\bibitem{CNZ} M. Carney, M. Nicol and H. K. Zhang, \emph{Compound
  Poisson law for hitting times to periodic orbits in
  two-dimensional hyperbolic systems}, J. Stat. Phys. 169 (2017),
  804--823.

\bibitem{CC} J. Chazzottes and P. Collet, \emph{Poisson
  approximation for the number of visits to balls in
  non-uniformly hyperbolic dynamical systems}, Ergodic Theory
  Dynam. Systems 33 (2013), no. 1, 49--80.
	
\bibitem{ChernovKleinbock} N.\ Chernov, D.\ Kleinbock,
  \emph{Dynamical Borel-Cantelli lemmas for Gibbs measures},
  Isr. J. Math. 122 (2001), no. 1, 1--27.

\bibitem{Collet0} P. Collet, \emph{A remark about uniform
  de-correlation prefactors}, unpublished
  note.

\bibitem{Collet} P. Collet, \emph{Statistics of closest return
for some non-uniformly hyperbolic systems}, {Ergodic Theory
  Dynam. Systems}, \textbf{21} (2001), 401--420.
  

\bibitem{CFFHN} M. Carvalho, A. C. M. Freitas, J. M. Freitas,
  M. Holland and M. Nicol, \emph{Extremal dichotomy for
    hyperbolic toral automorphisms}, Dyn. Syst. 30 (2015), no. 4,
  383--403.

\bibitem{Embrechts} P. Embrechts, C. Kl\"upperlberg, and
  T. Mikosch, \emph{Modelling extremal events. For insurance and
    finance}, Applications of Mathematics 33, Springer-Verlag,
  Berlin, 1997, ISBN: 3-540-60931-8.

\bibitem{F} J. Freitas, \emph{Extremal behaviour of chaotic
  dynamics}, Dyn. Syst. 28 (2013), no. 3, 302--332.
  
 
\bibitem{FFT1} J. Freitas, A. Freitas and M. Todd, \emph{Hitting
  times and extreme value theory}, Probab. Theory Related Fields
  {\bf 147(3)}, 675--710, 2010.

\bibitem{FFT2} A. C. M. Freitas, J. M. Freitas, M. Todd,
  \emph{Extremal index, hitting time statistics and periodicity},
  Adv. Math. {231}, no. 5, 2012, 2626--2665.

\bibitem{FFT3} A. C. M. Freitas, J. M. Freitas, M. Todd,
  \emph{Speed of convergence for laws of rare events and escape
    rates}, Stochastic Process. Appl. 125 (2015), no. 4,
  1653--1687.

\bibitem{FFTV} A. C. M. Freitas, J. M. Freitas, M. Todd and
  S. Vaienti, \emph{Rare events of the Manneville--Pomeau map},
  Stochastic Process. Appl. 126 (2016), no. 11, 3463--3479.

\bibitem{Galambos} J. Galambos, \emph{The Asymptotic Theory of
  Extreme Order Statistics}, John Wiley and Sons, 1978.

\bibitem{GHPZ} S. Galatolo, M. P. Holland, T. Persson and
  Y. Zhang, \emph{Birkhoff sums of infinite observables and
    anomalous time-scaling of extreme events in infinite
    systems}, Discrete and Continuous Dynamical Systems 41
  (2021), no. 4, 1799--1841.

\bibitem{GP} Ch. Ganotaki, T. Persson, \emph{On eventually always
  hitting points}, arXiv:2010.07714.
  
\bibitem{GKR} A. Ghosh, M. Kirsebom, and P. Roy, \emph{Continued
  fractions, the Chen-Stein method and extreme value theory},
  Ergodic Theory Dynam. Systems {41}, (2), (2021),
  461--470.

  
\bibitem{GHN} C. Gupta, M. P. Holland and M. Nicol, \emph{Extreme
  value theory for dispersing billiards, Lozi maps and Lorenz
  maps}, Ergodic Theory Dynam. Systems {31}, (5),
  (2011), 1363--1390.

\bibitem{GNO} C.~Gupta, M.~Nicol and W.~Ott, \emph{A
  Borel-Cantelli lemma for non-uniformly expanding dynamical
  systems}, Nonlinearity {23}, (8), (2010), 1991--2008.

\bibitem{HNPV} N. Haydn, M. Nicol, T. Persson and S. Vaienti,
  \emph{A note on Borel--Cantelli lemmas for non-uniformly
  hyperbolic dynamical systems}, Ergodic Theory Dynam.  Systems
  33, no. 2, (2013), 475--498.


\bibitem{HN} M. P. Holland and M. Nicol, \emph{Speed of
  convergence to an extreme value distribution for non-uniformly
  hyperbolic dynamical systems}, Stochastics and Dynamics, 15,
  No. 4 (2015).

\bibitem{HNT} M. P. Holland, M. Nicol and A. T\"or\"ok,
  \emph{Extreme value distributions for non-uniformly expanding
    dynamical systems}, Trans. Amer. Math. Soc., 364 (2012),
    661--688.

\bibitem{HNT2} M. P. Holland, M. Nicol and A. T\"or\"ok,
  \emph{Almost sure convergence of maxima for chaotic dynamical
    systems}, Stochastic Process. Appl. 126 (2016), no. 10,
  3145--3170.

\bibitem{HRS} M. P. Holland, P. Rabassa and A. E. Sterk,
  \emph{Quantitative recurrence statistics and convergence to an
    extreme value distribution for non-uniformly hyperbolic
    dynamical systems}, Nonlinearity 29 (2016), no. 8,
  2355--2394.

\bibitem{HW} N. Haydn and K. Wasilewska, \emph{Limiting
  distribution for error terms for the number of visits to balls
  in non-uniformly hyperbolic dynamical systems}, Discrete
  Contin. Dyn. Syst. 36 (2016), no. 5, 2585--2611.
  
\bibitem{Kelm} D.\ Kelmer, \emph{Shrinking targets for discrete
  time flows on hyperbolic manifolds}, Geom. Funct. Anal. 27
  (2017), no. 5, 1257--1287.
  
\bibitem{OhKelmer} D.\ Kelmer and H.\ Oh, \emph{Exponential
  mixing and shrinking targets for geodesic flow on geometrically
  finite hyperbolic manifolds}, arXiv:1812.05251.
  
\bibitem{KeYu} D.\ Kelmer and S.\ Yu, \emph{Shrinking targets
  problems for flows on homogeneous spaces},
  Trans. Amer. Math. Soc. 372 (2019), 6283--6314.

\bibitem{Kim} D.~Kim, \emph{The dynamical Borel--Cantelli lemma
  for interval maps}, Discrete Contin. Dyn. Syst. 17 (2007),
  no.~4, 891--900.

\bibitem{KKP} M. Kirsebom, P. Kunde, and T. Persson,
  \emph{Shrinking targets and eventually always hitting points
    for interval maps}, Nonlinearity 33 (2020), no. 2, 892--914.
								
\bibitem{KKP2} M. Kirsebom, P. Kunde, and T. Persson, \emph{On
shrinking targets and self-returning points}, arXiv:2003.01361.

\bibitem{Klass1} M. Klass, \emph{The minimal growth rate of
  partial maxima}, Ann. Prob. {12}, (1984), 380--389.

\bibitem{Klass2} M. Klass, \emph{The Robbins--Siegmund series
  criterion for partial maxima}, Ann. Prob. {13}, (4), (1985),
  1369--1370.

\bibitem{Kleinbock} D. Kleinbock, I. Konstantoulas and
  F. K. Richter, \emph{Zero--one laws for eventually always
    hitting points in mixing systems}, arXiv:1904.08584.

\bibitem{KleinWad} D.\ Kleinbock, N.\ Wadleigh, \emph{An
  inhomogeneous Dirichlet Theorem via shrinking targets},
  Compositio Mathematica, 155 (2019), no. 7, 1402--1423.

\bibitem{LasotaYorke} A. Lasota, J. A. Yorke, \emph{On the
  existence of invariant measures for piecewise monotonic
  transformations}, Trans. Amer. Math. Soc. 186 (1973), 481--488
  (1974).

\bibitem{LLR} Leadbetter, M. R., Lindgren, G., and Rootzen,
  H. (1983), Extremes and Related Properties of Random Sequences
  and Processes. Springer-Verlag, New York.
  
\bibitem{Liveranietal} C. Liverani, B. Saussol and S. Vaienti,
  \emph{Conformal measure and decay of correlation for covering
    weighted systems}, Ergod. Theory Dynam. Systems 18(6) (1998),
  1399--1420.

\bibitem{Letal} V. Lucarini et al., \emph{Extremes and Recurrence
  in Dynamical Systems}, Pure and Applied Mathematics (Hoboken),
  John Wiley \& Sons, Inc., Hoboken, NJ, 2016, ISBN:
  978-1-118-63219-2.
	
\bibitem{PS} F. P\`{e}ne and B. Saussol, \emph{Poisson law and
  some nonuniformly hyperbolic systems with polynomial rate of
  mixing}, Ergodic Theory Dynam. Systems 36 (2016), no. 8,
  2602--2626.
  
\bibitem{Philipp} W.~Philipp, \emph{A conjecture of Erd\"os on
  continued fractions}, Acta Arithmetica 28 (1976), Issue: 4,
  379--386.


\bibitem{Rychlik} M. Rychlik, \emph{Bounded variation and
  invariant measures}, Studia Mathematica 76 (1983), 69--80.
  

\bibitem{Young_quadratic} L.-S. Young, \emph{Decay of
  Correlations for Certain Quadratic Maps}, Communications in
  Mathematical Physics 146 (1992), 123--138.
  
\bibitem{Young} L.-S. Young, \emph{Statistical properties of
  dynamical systems with some hyperbolicity}, {Ann. of Math.}
  {147} (1998) 585--650.

\end{thebibliography}
\end{document}